\documentclass[12pt,BCOR=9.20mm,paper=a4]{scrartcl}
\usepackage{ayv_options}

\newcommand{\prd}{\mathcal{P}^p(\rd)}

\newcommand{\probspace}{(\Omega,\mathcal{F},\mathbb{P})}

\newcommand{\expect}{\mathbb{E}}

\newcommand{\pp}{\mathbb{P}}
\newcommand{\up}{\mathcal{U}^p}

\newcommand{\eup}[1]{\mathscr{U}^p_E[#1]}
\newcommand{\erp}[1]{\mathscr{V}^p_E[#1]}
\newcommand{\norm}[2]{\|#1\|_{#2}}
\newcommand{\xp}{\mathcal{X}^p}
\newcommand{\Lpspace}[4]{L^{#1}([#2,#3]\times \Omega;\mathcal{B}_\lambda([#2,#3])\otimes\mathcal{F},\lambda\otimes\pp;#4)}
\newcommand{\Lpspacedouble}[4]{L^{#1}([#2,#3]\times \Omega\times\Omega;\mathcal{B}_\lambda([#2,#3])\otimes\mathcal{F}\otimes\mathcal{F},\lambda\otimes\pp\otimes\pp;#4)}

\usepackage{bbm}
\usepackage[bbgreekl]{mathbbol}

\DeclareSymbolFontAlphabet{\mathbb}{AMSb}
\DeclareSymbolFontAlphabet{\mathbbl}{bbold}

\title{Pontryagin maximum principle for the deterministic mean field type  optimal control problem via the Lagrangian approach}

\author{Yurii Averboukh}\address{Krasovskii Institute of Mathematics and Mechanics, \\ & Yekaterinburg, Russia} \email{averboukh@gmail.com}
\author{Dmitry Khlopin}\address{Krasovskii Institute of Mathematics and Mechanics, \\ & Yekaterinburg, Russia}
\email{khlopin@imm.uran.ru}
\date{}

\begin{document}

	\maketitle
	
	\begin{abstract}
		We study  necessary optimality conditions for the deterministic mean field type free-endpoint optimal control problem. Our study relies on the Lagrangian approach that treats the mean field type control system as a crowd of infinitely many agents who are labeled by elements of some probability space. First, we derive the Pontryagin maximum principle in the Lagrangian form. Furthermore, we consider the Kantorovich and Eulerian formalizations which describe  mean field type control systems via distributions on the set of trajectories and nonlocal continuity equation respectively. We prove that  local minimizers in the Kantorovich or Eulerian formulations determine  local minimizers within the Lagrangian approach. Using this, we deduce the Pontryagin maximum principle in the Kantorovich and Eulerian forms. To illustrate the general theory, we examine a model system of mean field type linear quadratic regulator. We show that the optimal strategy in this case is determined by a linear feedback.
			\keywords{mean field type control, Pontryagin maximum principle, Lagrangian approach, Kantorovich approach, Eulerian approach, Pontryagin minimum}
			\msccode{49N80, 49K21, 49K15, 93C25, 34K27}
	\end{abstract}
	
\tableofcontents

\section{Introduction}
The main object of the paper is a system consisting of many identical agents who interacts via some external media and try to achieve a common goal. We study this system using the mean field approach that comes from  the statistical physics and examine the limit system where the number of agents tends to infinity. The latter can be regarded as a dynamical system in the space of probability measures. First, the mean field interacting dynamical systems appeared as models of plasma \cite{Vlasov, Vlasov_book} (see also \cite{McKean, Sznitman} for the mathematical theory of the mean field interacting systems). Recently, such models found applications in studies of crowds and flocks behavior, opinion dynamics, etc. \cite{Bellomo2012, Bullo2009, Colombo2011, Colombo2005, Cristiani2014}. 

The many agent systems with mean field interaction in the presence of controls can be treated in three ways. First, one can assume that each agent chooses their  control  to optimize  their own utility. This assumption leads to the mean field game theory proposed by Lasry, Lions \cite{Lions01, Lions02} and (independently) by Huang,  Malhame, Caines \cite{Huang5}. The second approach appears if we consider the many agent systems affected by one external control. Notice that this class includes systems where the control of each agent depends on their state in a smooth way. To see this, it suffices to consider a smooth profile of control as an external control. The mean field type control is a mixture of these aforementioned approaches. On one hand, it implies that each agent has their own control. On the other hand, the mean field type control theory assumes that the agents behave collectively to achieve a common goal. Equivalently, one can imagine a  central planner who tells the agents what to do in order to optimize some objective function.  Thus, the main object of mean field type control theory is a system of intellectual agents acting cooperatively. Moreover, since the agents move independently, one can expect in this case discontinuous dependences of controls and velocities on the agent's state.

The mean field type control theory inherited such problems as existence of optimal control, dynamic programming and necessary optimality conditions in the Pontryagin maximum principle form from the classical optimal theory. Papers \cite{ahmed_ding_controlld, Bahlali_Mezerdi_Mezerdi_existence} provide the existence of the mean field type optimal control. The dynamical programming principle and the representation of the value function of the mean field type control problem as the solution of a Bellman equation in the space of probability measures is discussed in  \cite{Badreddine_Frankowska,Bayraktar_Cosso_Pham_randomized, Bensoussan_Frehse_Yam_book, Bensoussan_Frehse_Yam_equation, Cavagnari_et_al_approaches, Marigonda_et_al_2015, Cavagnari_Marigonda_Piccoli1, Lauriere_Pironneau_DPP_MF_control, Pham_Wei_2015_DPP_2016, Pham_Wei_2015_Bellman}. The Pontryagin maximum principle for the stochastic mean field type control problem was derived \cite{Andersson_Djehiche_2011, Buckdahn_Boualem_Li_PMP_SDE}. Surprisingly,  the derivation of the necessary optimality conditions for the deterministic mean field type optimal control is more involved than this problem for the pure stochastic case. Nowadays, the Pontryagin maximum principle is obtained for the case when all agents are affected by the same control \cite{Bonnet2021, Pogodaev2016a} or when the control depends smoothly on the agent's state \cite{Bonnet2019,  Bonnet2019a}. The latter case, as mentioned above, can be reduced to systems with an external force if one regards the entire dependence of the control on the state variable as a new control.
Additionally,  paper \cite{Bongini2017} gives the Pontryagin maximum principle for a system consisting of finite-dimensional and mean field parts affected by the same external control.   Finally, the mean field type control theory raises its own questions. Among them  is the finite agent approximation problem \cite{Fornasier_Lisini_Orrieri_Savar,Fornasier_Solombrino, Gangbo2021, Lacker_limit} that provides the consistency of the mean field type control theory. 

Notice that the dynamic programming principle, Pontryagin maximum principle as well as finite agent approximations of the mean field type control problems require the technique of differential and sub-differential calculus in the space of probability measures. We refer to papers  \cite{Gangbo_Tudorascu, Master_cdll} for the detailed exposition of various approaches of this field.

The paper is concerned with the necessary optimality condition for a mean field type optimal control problem, where the evolution of each agent is driven by an ordinary differential equation.
We aims to derive the Pontryagin maximum principle for the general deterministic mean field type optimal control problem including, in particular, the case of unbounded control space. To this end, we use the Lagrangian approach \cite{Cavagnari_et_al_approaches} that implies the labeling of agents by elements of some probability space and, formally, reduces  the original problem to the certain control problem on the  space of functions.
Recall that the deterministic mean field type control problems can be also formalized within the Eulerian and Kantorovich approaches \cite{Cavagnari_et_al_approaches}.  

The Eulerian approach relies on the description of the evolution of the distribution of agents through the nonlinear continuity equation and regards the mean field type control problem as a control problem in the space of probability measures. In the case of mean field type control systems, the velocity field may be discontinuous. Notice that the Eulerian approach is even more natural for the case of systems affected by an external force.  It leads to a controlled continuity equation with regularity conditions on the velocity field.  Moreover, a  continuity equation with a control that smoothly depends on the agent's state can be treated as a system affected by an external control if one chooses the control space to be an appropriate class of functions of the state.

Finally, one can consider the mean field type optimal control problem as an optimization problem for  distributions on the set of curves under  constrain that these distributions are concentrated on the set of admissible curves. This idea leads to the Kantorovich approach. The equivalence between the Kantorovich and Eulerian approaches was proved in \cite[Theorem 1]{Jimenez2020} under the convexity assumption.  The value functions within all these aforementioned approaches coincide  under the same assumption \cite{Cavagnari_et_al_approaches}.

In this paper, we consider the deterministic mean field type optimal control problem with free-endpoint assuming that the dynamics and the payoff functions are continuously differentiable w.r.t. the state of each agent and the measure describing the distribution of all agents. We adopt the concept of intrinsic derivative w.r.t. probability measure proposed in \cite{Master_cdll}. The key result of the paper is the Pontryagin maximum principle for the Lagrangian formulation of the mean field type optimal control problem. In this case, the costate variable is described by a  process coupled with the original mean field type control process.  Furthermore, we extend the results of \cite{Cavagnari_et_al_approaches} and prove that the local minimizers within the Kantorovich and Eulerian approaches correspond to  local minimizers in the Lagrangian framework. Using this, we obtain the Pontryagin maximum principle for the Kantorovich and Eulerian approaches. In the latter case, the costate equation is replaced by the continuity equation both on state and costate variables. Additionally, we apply  the Pontryagin maximum principle in the Lagrangian framework to analyze  the mean field type linear-quadratic regulator. In this model, we assume that the motion of each agent is given by a linear differential equation while the payoff combines the averaged cost of the agents' controls and the terms describing the collective behavior of all agents. We show that the optimal control in this model problem can be chosen in the feedback form. Moreover, the control of each agent is determined by the mean state of all agents and the deviation of the agent's state from this  mean.

Notice that  originally the Pontryagin maximum principle was obtained as the necessary condition for  strong extrema \cite{Pontryagin1964a}. Later, it was shown that the Pontryagin maximum principle  corresponds to the more subtle notion of extremum called a Pontryagin extremum \cite{Dubovitskiui1967}. It lies between the strong and weak extrema. We follow this approach and extend the notion of Pontryagin extremum to the Lagrangian formulation of mean field type optimal control problem. As for the finite dimensional case, the Pontryagin maximum principle serves as a necessary condition for the Pontryagin minimizer. This reveals some similarities between finite-dimensional control systems and the Lagrangian formalization of mean field type control systems.

The paper is organized as follows. In Section~\ref{sect:preliminaries}, we introduce the general notation, the state and control spaces. Additionally, in that section we recall the definition of the intrinsic derivative w.r.t. measure variable. Section~\ref{sect:lagrangian} is concerned with the Lagrangian approach to the mean field type control systems. Here, in particular,  we introduce the concepts of strong and Pontryagin local minima for the Lagrangian formalization of the mean field type control problem. In Section~\ref{sect:PMP_Lagrangian}, we give the statement of the Pontryagin maximum principle in this case. The next two sections are concerned with the proof of this result. To this end, we study  spike variations of the mean field type optimal control processes within the Lagrangian approaches (see Section~\ref{sect:Spike}). In Section~\ref{sect:proof_PMP_Lagrangian}, we derive the costate equation, transversality and maximization conditions which constitute the Pontryagin maximum principle for the Lagrnagian formalization. The Kantorovich approach is examined in Section~\ref{sect:Kantorovich}. Here we study the relationship between  strong local extrema within the Kantorovich and Lagrangian frameworks and derive the Pontryagin maximum principle in the Kantorovich formulation. Using the same scheme, we show that  local Eulerian minimizers correspond to  Lagrangian ones and deduce the Eulerian version of the Pontryagin maximum principle in Section~\ref{sect:Euler}. Finally, Section~\ref{sect:MFT_LQR} provides the analytical study of the model mean field type linear quadratic regulator. 

\section{Preliminaries}\label{sect:preliminaries}
\subsection{General notation}
\begin{itemize}
	\item If ${\mathbbm{X}}_1,\ldots, {\mathbbm{X}}_n$ are sets, $i_1,\ldots,i_k$ are some indices from $\{1,\ldots,n\}$, then we denote by
	$\operatorname{p}^{i_1,\ldots,i_k}$ the natural projector from ${\mathbbm{X}}_1\times\ldots\times {\mathbbm{X}}_n$ onto ${\mathbbm{X}}_{i_1}\times\ldots\times {\mathbbm{X}}_{i_k}$, i.e.,
	\[\operatorname{p}^{i_1,\ldots,i_k}(x_1,\ldots,x_n)=(x_{i_1},\ldots,x_{i_k}).\]
	\item If $\mathbbm{X}$ is a set, $\Upsilon\subset \mathbbm{X}$, then $\mathbbm{1}_\Upsilon$ stands for the indicator function of the set $\Upsilon$.
	\item If $(\Omega',\mathcal{F}')$, $(\Omega'',\mathcal{F}'')$ are measurable spaces, $m$ is a probability on $\mathcal{F}'$, $h:\Omega'\rightarrow\Omega''$ is a $\mathcal{F}'/\mathcal{F}''$-measurable function, then we denote by $h\sharp m$ the push-forward measure that is the probability on $\mathcal{F}''$ defined by the rule: for $\Upsilon\in \mathcal{F}''$,
	\[(h\sharp m)(\Upsilon)\triangleq m(h^{-1}(\Upsilon)).\]
	\item If $(\Omega,\mathcal{F})$ is a measurable space, $m$ is a measure on $\mathcal{F}$, then $\mathcal{F}_{m}$ stands for the completion of $\mathcal{F}$ w.r.t. to the measure $m$. The extension of this measure onto $\mathcal{F}_m$ is still denoted by $m$.
	\item If $(\Omega,\mathcal{F},\pp)$ is a probability space, $({\mathbbm{X}},\rho_{\mathbbm{X}})$ is a metric space, while $g:\Omega\rightarrow {\mathbbm{X}}$ is a $\mathcal{F}/\mathcal{B}({\mathbbm{X}})$-measurable function, then we denote by $\mathbb{E}g$ the expectation of $g$ according to the probability $\pp$, i.e.,
	\[\mathbb{E}g\triangleq \int_\Omega g(\omega)\pp(d\omega).\]
	
	\item If $(\Omega,\mathcal{F})$ is a measurable space, $\mathbbm{Y}$ is a closed subset of a normed space $(\mathbbm{X},\norm{\cdot}{\mathbbm{X}})$, then we denote by $B(\Omega,\mathcal{F};\mathbbm{Y})$ the set of all $\mathcal{F}/\mathcal{B}(\mathbbm{Y})$-measurable functions from $\Omega$ to $\mathbbm{Y}$. If, additionally, $m$ is a measure on $\mathcal{F}$, $p\geq 1$, then we denote by $L^p(\Omega,\mathcal{F},m;{\mathbbm{Y}})$ the set of functions $g\in B(\Omega,\mathcal{F};{\mathbbm{Y}})$ such that
	\[\expect\|g\|^p=\int_\Omega\norm{g(\omega)}{{\mathbbm{X}}}^pm(d\omega)<+\infty.\] Notice that $L^p(\Omega,\mathcal{F},m;{\mathbbm{Y}})\subset L^p(\Omega,\mathcal{F},m;{\mathbbm{X}})$. The norm of an element $g\in L^p(\Omega,\mathcal{F},m;{\mathbbm{Y}})$ is given by
	\[\|g\|_{L^p}\triangleq  \bigg[\int_\Omega\|g\|_{\mathbbm{X}}^pm(d\omega)\bigg]^{1/p}.\]
	\item If $({\mathbbm{X}},\rho_{\mathbbm{X}})$, $({\mathbb{Y}},\rho_{\mathbb{Y}})$ are Polish spaces,  $C({\mathbbm{X}};{\mathbb{Y}})$ stands for the set of continuous function from ${\mathbbm{X}}$ to~${\mathbb{Y}}$. Furthermore, $C_b({\mathbbm{X}};{\mathbb{Y}})$ denotes the set of all continuous and bounded function. We will consider the usual $\sup$-norm on $C_b({\mathbbm{X}};{\mathbb{Y}})$. If ${\mathbb{Y}}=\mathbb{R}$, we omit the second argument. 
	\item If $({\mathbbm{X}},\rho_{\mathbbm{X}})$ is a Polish space, then $\mathcal{B}({\mathbbm{X}})$ denotes  the Borel $\sigma$-algebra on ${\mathbbm{X}}$. Moreover, if $m$ is a measure on $\mathcal{B}(X)$, then $\mathcal{B}_m(\mathbbm{X})$ is the $m$-completion of $\mathcal{B}(\mathbbm{X})$. 
	\item When, as above, $({\mathbbm{X}},\rho_{\mathbbm{X}})$ is a Polish space, we denote by $\mathcal{P}({\mathbbm{X}})$ the space of all Borel probabilities on it. We endow $\mathcal{P}({\mathbbm{X}})$ with the topology of narrow convergence. Recall that a sequence $\{m_n\}_{n=1}^\infty\subset \mathcal{P}({\mathbbm{X}})$ narrowly converges to $m$, iff, for every $\phi\in C_b({\mathbbm{X}})$,
	\[\int_{\mathbbm{X}}\phi(x)m_n(dx)\rightarrow \int_{\mathbbm{X}}\phi(x)m(dx)\text{ as }n\rightarrow\infty.\]
	\item If $(\Omega_1,\mathcal{F}_1)$ and $(\Omega_2,\mathcal{F}_2)$ are measurable  spaces,  then $\mathcal{F}_1\otimes\mathcal{F}_2$ is the product  $\sigma$-algebra, i.e., the $\sigma$-algebra generated by the family $\{\Upsilon_1\times\Upsilon_2:\, \Upsilon_1\in\mathcal{F}_1,\, \Upsilon_2\in\mathcal{F}_2\}$. Furthermore, if $m_1$ and $m_2$ are measures on $\mathcal{F}_1$ and $\mathcal{F}_2$ respectively, then $m_1\otimes m_2$ stands for the  product of measures defined by the rule: for every  $\Upsilon_1\subset \mathcal{F}_1$, $\Upsilon_2\subset \mathcal{F}_2$,
	\[(m_1\otimes m_2)(\Upsilon_1\times \Upsilon_2)\triangleq m_1(\Upsilon_1)\cdot m_2(\Upsilon_2).\] 
	\item When $(\Omega_1,\mathcal{F}_1)$, $(\Omega_2,\mathcal{F}_2)$ are measurable spaces, $m_1$ is a finite measure on $\mathcal{F}_1$ and   $h:\Omega_1\times\mathcal{F}_2\rightarrow \mathbb{R}$ is such that, for every $\omega_1 \in \Omega_1$, $h(\omega_1,\cdot)$ is a probability on $\mathcal{F}_2$, while, for each $\Upsilon_2\in\mathcal{F}_2$, the mapping $\Omega_1\ni\omega_1\mapsto h(\omega_1,\Upsilon_2)$ is measurable w.r.t. $\mathcal{F}_1$, we denote by $m_1\otimes (h(\omega_1))_{\omega_1\in\Omega_1}$ the measure $\mu$ on $\mathcal{F}_1\otimes\mathcal{F}_2$ such that, for every $\Upsilon_1\in\mathcal{F}_1$, $\Upsilon_2\in\mathcal{F}_2$,
	\[\mu(\Upsilon_1\times\Upsilon_2)=\int_{\Upsilon_1}h(\omega_1,\Upsilon_2)m_1(d\omega_1). \] The existence and uniqueness of such measure directly follows from \cite[Theorem 10.7.2]{Bogachev}. Moreover, if $\phi$ is measurable w.r.t. $\mathcal{F}_1\otimes\mathcal{F}_2$, then 
	\[\begin{split}
		\int_{\Omega_1\times\Omega_2}\phi(\omega_1,&\omega_2)(m_1\otimes (h(\omega_1))_{\omega_1\in\Omega_1})(d(\omega_1,\omega_2))\\&=\int_{\Omega_1}
		\int_{\Omega_2}\phi(\omega_1,\omega_2)h(\omega_1,d\omega_2)m_1(d\omega_1).\end{split} \]
	Notice that the direct product of measures appears when one chooses $h(\omega_1,\Upsilon_2)\triangleq m_2(\Upsilon_2)$.
	\item If $({\mathbbm{X}},\rho_{\mathbbm{X}})$ is a Polish space, $p\geq 1$, then we denote by $\mathcal{P}^p({\mathbbm{X}})$ the set of probability measures with the finite $p$-th moment, i.e., $m\in\mathcal{P}({\mathbbm{X}})$ lies in $\mathcal{P}^p({\mathbbm{X}})$ if, for some $x_*\in {\mathbbm{X}}$,
	\[\mathcal{M}_p^p(m)\triangleq \int_{\mathbbm{X}}(\rho(x,x_*))^pm(dx)<+\infty.\] If ${\mathbbm{X}}$ is Banach, we will choose $x_*=0$. Below, $\mathcal{M}_p(m)$ denotes the $p$-th root of $\mathcal{M}_p^p(m)$. 
	\item The space $\mathcal{P}^p({\mathbbm{X}})$ is endowed with the $p$-th Wasserstein metric defined by the rule: for $m',m''\in\mathcal{P}^p({\mathbbm{X}})$,
	\[\begin{split}
		W_p(&m',m'') \\&\triangleq\inf\bigg\{\left[\int_{{\mathbbm{X}}\times {\mathbbm{X}}}\big(\rho_{\mathbbm{X}}(x',x'')\big)^p\pi(d(x',x''))\right]^{1/p}:\pi\in\Pi(m',m'')\bigg\},\end{split}\] where $\Pi(m',m'')$ stands for the set of all plans between $m'$ and $m''$, i.e., $\pi\in \mathcal{P}({\mathbbm{X}}\times {\mathbbm{X}})$ if $\operatorname{p}^1\sharp \pi=m'$ and $\operatorname{p}^2\sharp \pi=m''$. Recall that the sequence $\{m_n\}_{n=1}^\infty\subset \mathcal{P}^p({\mathbbm{X}})$ converges to $m\in \mathcal{P}^p({\mathbbm{X}})$ in the $p$-th Wasserstein metric iff $m_n$ converges to $m$ narrowly and $\{m_n\}_{n=1}^\infty$ has uniformly integrable $p$-th moment~\cite{Ambrosio}. 
	\item We assume that $\rd$ is the Euclidean space of column-vectors, when  $\rds$ stands for the space of row-vectors.
	\item If $\phi:\rd\rightarrow\mathbb{R}$ is a $C^1$-function, then $\nabla_x\phi(x)$ denotes the row-vector of its partial derivatives. In the case where $\phi$ takes values in $\rd$,  $\nabla_x\phi$ is assumed to be a matrix.
	\item $\lambda$ stands for the Lebesgue measure on the time interval $[0,T]$, $T>0$, $\mathscr{B}_T$ denotes the Lebesgue $\sigma$-algebra on $[0,T]$, i.e., $\mathscr{B}_T\triangleq \mathcal{B}_\lambda([0,T])$. Additionally, $L^p([0,T];{\mathbbm{X}})\triangleq L^p([0,T],\mathscr{B}_T,\lambda;{\mathbbm{X}})$;
	\item If $({\mathbbm{X}},\rho_{\mathbbm{X}})$ is a Polish space, $p\geq 1$, we denote by $\operatorname{AC}^p([0,T];{\mathbbm{X}})$  the set of absolutely continuous functions from $[0,T]$ to ${\mathbbm{X}}$ with the metric derivative lying in $L^p([0,T];\mathbb{R})$ (see \cite[\S 1.1]{Ambrosio} for details). 
	\item Below we fix $p>1$ and denote by $q$ the exponent dual to $p$, i.e., $1/p+1/q=1$.
\end{itemize}

\subsection{Calculus on the space of probability measures} 
In the paper, we consider the concept of intrinsic derivative. Let $\Phi:\mathcal{P}^p(\rd)\rightarrow \mathbb{R}$. The following definition is borrowed from  \cite[Definition 2.2.1.]{Master_cdll}.
\begin{definition}\label{def:var_derivatve}
	The function $\Phi$ is called  of the class $C^1$ if there exists a continuous function $\frac{\delta\Phi}{\delta m}:\prd\times\rd\rightarrow\mathbb{R}$ such that, for any $m'\in\prd$,
	\[\lim_{s\downarrow 0}\frac{\Phi((1-s)m+sm')-\Phi(m)}{s}=\int_{\rd}\frac{\delta\Phi}{\delta m}(m,y)\big[m'(dy)-m(dy)\big].\]
\end{definition}
The function $\frac{\delta\Phi}{\delta m}$ is called the flat derivative of the function $\Phi$.

For the $C^1$-function $\Phi$, we, in particular, have the following equality:
\begin{equation}\label{equality:Phi_integral}
	\Phi(m')-\Phi(m)=\int_0^1\int_{\rd}\frac{\delta\Phi}{\delta m}((1-s)m+sm',y)[m'(dy)-m(dy)]ds.
\end{equation}

Notice that the function $\frac{\delta\Phi}{\delta m}$ is defined up to an additive constant. Following \cite[Definition 2.2.2]{Master_cdll}, we assume the normalization: for each $m\in\prd$,
\[\int_{\rd}\frac{\delta \Phi}{\delta m}(m,y)m(dy)=0.\] 

The following definition also is proposed in \cite{Master_cdll} (see Definition 2.2.2 there).
\begin{definition}\label{def:L_derivative}
	If the function $\rd\ni y\mapsto \frac{\delta\Phi}{\delta y}(m,y)$ is $C^1$, then the function $\nabla_m \Phi$ defined by the rule
	\[\nabla_m\Phi(m,y)\triangleq \nabla_y\frac{\delta\Phi}{\delta m}(m,y)\] is called an intrinsic derivative of the function $\Phi$.
\end{definition} In the following, we assume that $\nabla_m\Phi$ takes values in the space of row-vectors $\rds$. When $\nabla_m\Phi$ exists and is continuous, we say that $\Phi$ is continuously differentiable.

Similarly to the finite dimensional case, the boundness  of the derivative w.r.t. probability implies the Lipschitz continuity w.r.t. to the Wasserstein distance. This property is proved in Proposition~\ref{prop:lipschitz} (see~\ref{appendix:derivative}). Additionally, in that Appendix, we compute the intrinsic derivative for two basic examples of functionals over measures, and find  the Gateaux derivative  of a function that depends on a distribution of a random variable.

\subsection{State and control spaces} 

As we mentioned above,  the state space for each agent is $\rd$. We follow approach first proposed by Gamkrelidze \cite{Gamkrelidze} and assume that an adjoint variable lies the dual space to $\rd$ that is the space of row-vector denoted by $\rds$. 

We denote the set of all trajectories on $[0,T]$ by $\Gamma$, i.e.,
\[\Gamma\triangleq C([0,T];\rd).\] We endow $\Gamma$ with the usual $\sup$-norm denoted by $\|\cdot\|_{\infty}$. The set of continuous functions defined on $[0,T]$ with values in $\rds$ will be denoted by $\Gamma^\star$. As above, on $\Gamma^\star$ we consider the $\sup$-norm still denoted by~$\|\cdot\|_{\infty}$.

We denote the evaluation operator by $e_t$, i.e, for each $t\in [0,T]$, $e_t:\Gamma\rightarrow\rd$ acts by the rule:
\[e_t(\gamma)\triangleq\gamma(t).\] With some abuse of notation, we use the same symbol $e_t$ for the evaluation operators defined on $\Gamma^\star$ and $\Gamma\times\Gamma^\star$. In those case, $e_t$ takes values either in $\rds$ or in $\rd\times\rds$.

If $b:[0,T]\times \rd\rightarrow\rd$, then we say that $x(\cdot):[0,T]\rightarrow \rd$ satisfies the differential equation
\[\frac{d}{dt}x(t)=b(t,x(t))\] if, for every $t\in [0,T]$,
\[x(t)=x(0)+\int_0^t b(\tau,x(\tau))d\tau.\] 

In the paper, we primarily deal with the Lagrangian approach which describe the motion and open-loop strategy of the mean field type control system as   processes $X$ and $u$ respectively defined on some standard probability space $\probspace$. Throughout the paper, we  follow the conventions of  probability theory and  omit the dependence on $\omega$  when no confusion  arises. Additionally, as it was mentioned above, if $g$ is a random variable, we primarily write $\expect g$ instead of $\int_\Omega g(\omega)\pp(d\omega)$.

We assume that a process describing a motion of the system has continuous paths with the $\sup$-norms lying in $L^p$ for some $p>1$, i.e, we  work with the space $\xp$ that contains all functions $X$ defined on $[0,T]\times \Omega$ with values in~$\rd$ satisfying the following condition: the mapping $\widehat{X}$ that assigns to $\omega\in\Omega$ the whole path $X(\cdot,\omega)$ takes values in $\Gamma$  $\pp$-a.s. and lies in $L^p(\Omega,\mathcal{F},\pp;\Gamma)$.
The norm on $\xp$ is equal to
\[\norm{X}{\xp}\triangleq \big(\expect[\norm{\widehat{X}}{\infty}^p]\big)^{1/p}=\left(\int_{\Omega}\sup_{t\in [0,T]}\|X(t,\omega)\|^p\pp(d\omega)\right)^{1/p}.\]
Notice that $X\in\xp$ is entirely determined by an element of $L^p(\Omega,\mathcal{F},\pp;\Gamma)$. To show this, it suffices, given $\widehat{X}\in L^p(\Omega,\mathcal{F},\pp;\Gamma)$,  let $X(t,\omega)\triangleq \widehat{X}(\omega)(t)$. Moreover, each $X\in\xp$ is measurable w.r.t. $\mathscr{B}_T\otimes \mathcal{F}$.


If $X\in \xp$, then, for each $t\in [0,T]$, the mapping $\Omega\ni\omega \mapsto X(t,\omega)$ is an element of $L^p(\Omega,\mathcal{F},\pp;\rd)$. Due to the convention of probability theory, we will widely use $X(t)$ both to denote $X(t,\omega)$ and the mapping $X(t,\cdot)$ when their meanings are clear. In particular, $X(t)\sharp P$ means the push-forward measure of the probability $\pp$ by the mapping $X(t,\cdot)$. In this case,
\begin{equation}\label{ineq:M_X_t_X_xp}\mathcal{M}_p^p(X(t)\sharp \pp)=\norm{X(t)}{L^p}^p\leq \norm{X}{\xp}^p.\end{equation}

In the paper, we consider the case where the set of instantaneous controls $U$ is a closed subset of some normed space. 
Generally, the set $U$ can be unbounded while the payoff can grow superlinearly 
(see assumptions~\ref{assumption:Uclosed},~\ref{assumption:sublinear} below). 
Thus, it is reasonable to assume that the agents use  controls with  finite $L^p$-norm. 
Therefore, within the Lagrangian approach, the function assigning to the agent's label and time instant a control is 
chosen from the  set \[\mathcal{U}^p\triangleq L^p([0,T]\times \Omega,\mathscr{B}_T\otimes\mathcal{F},\lambda\otimes\pp;U).\] 
Recall that the norm of an element $u\in\mathcal{U}^p$ is given by the formula:
\[\norm{u}{\up}\triangleq \left(\mathbb{E}\left[\int_0^T\|u(t)\|^pdt\right]\right)^{1/p}
=\left(\int_0^T\int_\Omega\|u(t,\omega)\|^p\pp(d\omega)dt\right)^{1/p}.\]

\section{Lagrangian formulation of the mean field type control problem}\label{sect:lagrangian}
We consider the mean field type control problem with the dynamics of each agent given by the ordinary differential equation 
\begin{equation*}\label{eq:system}
	\begin{split}
		\frac{d}{dt}x(t)=f(t,x(t&),m(t),u(t)),\\ &t\in [0,T],\ \ x(t)\in\rd, \ \ m(t)\in\prd,\ \ u(t)\in U. 
	\end{split}
\end{equation*} Here $x(t)$ is the state, while $u(t)$ is the control of the agent at time $t$. Additionally, $m(t)$ describes the distribution of all agents at time $t$. The initial distribution of agents is assumed to be fixed and equal to $m_0$. The agents try to minimize the averaged  individual cost. The latter is equal to
\begin{equation*}\label{system:payoff}
	\sigma(x(T),m(T))+\int_0^Tf_0(t,x(t),m(t),u(t))dt.
\end{equation*} 

In the hypotheses formulated below, we use $\nabla_m f(t,x,m,y,u)$ for the derivative of $f$ w.r.t. measure variable for fixed $t$, $x$ and $u$. Recall that this derivative is a function of extra variable $y\in\rd$. The same concerns $\nabla_m f_0(t,x,m,y,u)$ and $\nabla_m\sigma(x,m,y)$. 

Throughout this paper, we assume the following.
\begin{enumerate}[label=(H\arabic*), series=listCond]
	\item \label{assumption:Uclosed}  $U$ is a closed subset of a separable Banach space;
	\item\label{assumption:function} the functions $f$, $f_0$  are Lebesgue measurable w.r.t. $t$ and continuous w.r.t. phase, measure and control variables; 
	\item\label{assumption:sublinear} there exists a constant $C_\infty$ such that
	\[\|f(t,x,m,u)\|\leq C_\infty(1+\|x\|+\mathcal{M}_p(m)+\|u\|),\]
	\[|f_0(t,x,m,u)|\leq C_\infty(1+\|x\|^p+\mathcal{M}_p^p(m)+\|u\|^p),\]
	\[|\sigma(x,m)|\leq C_\infty(1+\|x\|^p+\mathcal{M}_p^p(m));\]
	\item\label{assumption:derivative_f}  the function $f$ is continuously differentiable w.r.t. $x$ and $m$; its derivatives $\nabla_x f$ and $\nabla_m f$ are bounded by constants $C_x$ and $C_m$ respectively;
	\item\label{assumption:derivative_f_0} the function $f_0$ is continuously differentiable w.r.t. $x$ and $m$; the derivatives $\nabla_x f_0$ and $\nabla_m f_0$ satisfy the following growth conditions with constants $C_x^0$, $C_m^0$:
	\[\|\nabla_x f_0(t,x,m,u)\|^q\leq C_x^0(1+\|x\|^p+\mathcal{M}_p^p(m)+\|u\|^p),\]
	\[\|\nabla_m f_0(t,x,m,y,u)\|^q\leq C_m^0(1+\|x\|^p+\|y\|^p+\mathcal{M}_p^p(m)+\|u\|^p);\]
	\item \label{assumption:derivative_sigma} the terminal payoff $\sigma$ is continuously differentiable; the functions $\nabla_x \sigma$ and $\nabla_m \sigma$ satisfy the following estimates with some nonnegative constants $C_x^\sigma$, $C_m^\sigma$:
	\[\|\nabla_x \sigma(x,m)\|^q\leq C_x^\sigma(1+\|x\|^p+\mathcal{M}_p^p(m)),\]
	\[\|\nabla_m \sigma(x,m,y)\|^q\leq C_m^\sigma(1+\|x\|^p+\|y\|^p+\mathcal{M}_p^p(m)).\] 
\end{enumerate} 
In conditions \ref{assumption:derivative_f_0}, \ref{assumption:derivative_sigma}, $q$ stands for the exponent dual to $p$, i.e., $\frac{1}{p}+\frac{1}{q}=1.$

Let us introduce the Lagrangian approach to the mean field type  control problems (see \cite{Cavagnari_et_al_approaches} for details). It relies on labeling of the agents by elements of a set $\Omega$. In the following, let $(\Omega,\mathcal{F},\pp)$ be a standard probability space.

\begin{definition}\label{def:Lagrangian_admissible}
	We say that a pair $(X,u)$, where $X\in \xp$, $u\in \up$, is a Lagrangian control process if, for $\pp$-a.e. $\omega\in\Omega$, $X(\cdot,\omega)$ solves the differential equation
	\[\frac{d}{dt}X(t,\omega)=f(t,X(t,\omega), X(t)\sharp \pp, u(t,\omega)).\]
\end{definition} The payoff function within the Lagrangian approach is computed by the formula:
\begin{equation}\label{intro:def_J_L}J_L(X,u)\triangleq \mathbb{E}\left[\sigma(X(T),X(T)\sharp \pp)+\int_0^T f_0(t,X(t),X(t)\sharp \pp,u(t))dt\right].\end{equation} 

\begin{remark}\label{remark:well_posedness} Due to assumption \ref{assumption:sublinear} the functional $J_L(X,u)$ is finite for every $X\in \xp$, $u\in \up$. 
\end{remark}

Notice that, if $(X,u)$ is a Lagrangian control process, then the paths $\widehat{X}$ are $\pp$-a.s. absolutely continuous function. However, it is more convenient to work with a larger class of continuous functions.
This will be used in Sections~\ref{sect:Kantorovich},~\ref{sect:Euler} to establish links of Lagrangian approach with Kantorovich and Eulerian formalizations.

For the Lagrangian formulation of the optimal control problem we will consider two type of  initial conditions. First, assume that the initial assignment of agents $X_0\in L^p(\Omega,\mathcal{F},\pp;\rd)$ is given, whilst the second approach fixes only the initial distribution. 
\begin{definition}\label{def:boundary_condition} We say that a Lagrangian control process $(X,u)$ meets the initial condition for the given assignment $X_0$ where $X_0\sharp \pp=m_0$ if
	\[X(0)=X_0,\ \ \pp\text{-a.s.}\] Given $X_0$, we denote the set of control processes satisfying initial assignment condition by $\mathcal{A}_L(X_0)$.
	
	We say that a process $(X,u)$ satisfies the initial distribution conditions if
	\[X(0)\sharp \pp=m_0.\] The set of control processes satisfying initial distribution condition is denoted by $\operatorname{Adm}_L(m_0)$.
\end{definition} Notice that,
\[\operatorname{Adm}_L(m_0)=\bigcup_{X_0\in L^p(\Omega,\mathcal{F},\pp;\rd): X_0\sharp \pp=m_0}\mathcal{A}_L(X_0).\]
Simultaneously, an initial assignment condition can detail a feature of the initial distribution in the case when the probability space $\probspace$ is sufficiently rich.  

In this paper, we examine both strong and Pontryagin minima. In the latter case, we use  concepts borrowed from \cite{Arutyunov2004}.

\begin{definition}\label{def:minima_L_p_strong}
	Given an initial assignment $X_0\in L^p(\Omega,\mathcal{F},\pp;\rd)$, we say that a control process $(X^*,u^*)\in\mathcal{A}_L(X_0)$ is a strong local  $L^p$-minimizer at~$X_0$ if there exists $\varepsilon>0$ satisfying the following condition: for every $(X,u)\in\mathcal{A}_L(X_0)$ such that $\|X-X^*\|_{\xp}\leq\varepsilon$,
	\begin{equation}\label{ineq:J_L_minima}J_L(X^*,u^*)\leq J_L(X,u).\end{equation}
\end{definition}
\begin{definition}\label{def:minima_W_strong}
	A control process $(X^*,u^*)\in\operatorname{Adm}_L(m_0)$ is called a strong local  $W_p$-minimizer at $m_0\in\prd$ if one can find $\varepsilon>0$ such that~(\ref{ineq:J_L_minima}) holds true for every $(X,u)\in\operatorname{Adm}_L(m_0)$ satisfying $W_p(X(t)\sharp \pp,X^*(t)\sharp \pp)\leq\varepsilon$ when $t\in [0,T]$.
\end{definition}
\begin{definition}\label{def:minima_L_p_Pontr}
	Given an initial assignment $X_0\in L^p(\Omega,\mathcal{F},\pp;\rd)$,  a control process $(X^*,u^*)\in\mathcal{A}_L(X_0)$ is said to be a  Pontryagin local $L^p$-minimizer at~$X_0$ if there exists $\varepsilon>0$ satisfying the following condition: for each $(X,u)\in\mathcal{A}_L(X_0)$ such that $\|X-X^*\|_{\xp}\leq\varepsilon$ and $(\lambda\otimes\pp)\{(t,\omega)\in [0,T]\times \Omega:u^*(t,\omega)\neq u(t,\omega)\}\leq\varepsilon$, inequality
	\eqref{ineq:J_L_minima} is fulfilled.
\end{definition}
\begin{definition}\label{def:minima_W_Pontr}
	A control process $(X^*,u^*)\in\operatorname{Adm}_L(m_0)$ is called a  Pontryagin local $W_p$-minimizer at $m_0$ if one can find $\varepsilon>0$ such that~(\ref{ineq:J_L_minima}) holds true for every $(X,u)\in\operatorname{Adm}_L(m_0)$ satisfying $\sup_{t\in [0,T]}W_p(X(t)\sharp \pp,X^*(t)\sharp \pp)\leq\varepsilon$ and $(\lambda\otimes\pp)\{(t,\omega)\in [0,T]\times \Omega:u^*(t,\omega)\neq u(t,\omega)\}\leq\varepsilon$.
\end{definition} Let us discuss the relationship between the minima introduced above. 
\begin{proposition}\label{prop:minimizers} Let $X_0\in L^p(\Omega,\mathcal{F},\pp;\rd)$ and let $m_0\in \prd$ be such that $m_0=X_0\sharp \pp$.
	\begin{enumerate}
		\item If $(X^*,u^*)\in \mathcal{A}_L(X_0)$ is a strong local $W_p$-minimizer at $m_0$, then it is a strong local $L^p$-minimizer at $X_0$. 
		\item If $(X^*,u^*)\in \mathcal{A}_L(X_0)$ is a  Pontryagin local $W_p$-minimizer at $m_0$, it is a  Pontryagin local $L^p$-minimizer at $X_0$.
		\item Every  strong local $L^p$-minimizer at $X_0$ is a  Pontryagin local $L^p$-minimizer at $X_0$;
		\item Every strong $W_p$-minimizer at $m_0$ is a Pontryagin local $W^p$-minimizer at $m_0$.
	\end{enumerate}
\end{proposition}
\begin{proof}
	We will consider only the first statement as the second one is proved in the same fashion, whilst the third and fourth statements are obvious. 
	
	Since $(X^*,u^*)\in \mathcal{A}_L(X_0)$ is a strong local $W_p$-minimizer at $m_0$,  there exists $\varepsilon>0$ such that, for every $(X,u)\in \operatorname{Adm}_L(m_0)$ satisfying $W_p(X(t)\sharp\pp,X^*(t)\sharp \pp)\leq\varepsilon$, one has
	\[J_L(X^*,u^*)\leq J_L(X,u).\] 
	
	Now let $(X,u)\in \mathcal{A}_L(X_0)$ be such that  $\norm{X-X^*}{\xp}\leq\varepsilon$. Given $t\in [0,T]$, we choose a plan $\pi\triangleq (X(t),X^*(t))\sharp\pp$. By construction, $\pi\in \Pi(X(t)\sharp\pp,X^*(t)\sharp\pp)$. Thus, we have that 
	\[\begin{split}
		W_p^p(X(t)\sharp\pp,X^*(t)\sharp \pp)&\leq \int_{\rd\times\rd}\|x-x^*\|^p\pi(d(x,x^*))\\&=\norm{X(t)-X^*(t)}{L^p}^p\leq \norm{X-X^*}{\xp}^p\leq\varepsilon^p.\end{split}\] This together with  assumption that $(X^*,u^*)$ is a  local $W_p$-minimizer at $m_0$ gives the first statement of the proposition.
\end{proof}

The Pontryagin maximum principle for the Lagrangian formalization is derived for the mildest concept of minimum that is the Pontryagin $L^p$-minimizer. This is the main motivation to introduce this concept. Notice that it utilizes the class $\mathcal{A}_L(X_0)$. At the same time, the class of processes $\operatorname{Adm}_L(m_0)$ and the corresponding concept of Lagrangian $W_p$-minima fit both the Kantorovich and Eulerian approaches. As we will see below in Theorems \ref{th:Kantorovich_min_Lagrangian} and \ref{th:Euler_min_Lagrangian}, for each minimizer within these approaches, one can find an appropriate Lagrangian $W_p$-minimizer that is a Lagrangian $L^p$-minimizer for some initial assignment~$X_0$. However, it follows from  \cite[\S 8.3]{Cavagnari_et_al_approaches} that, generally, there is  an initial assignment  that  does not allow a Lagrangian $L^p$-minimizer lying in $\mathcal{A}_L(X_0)$, whilst the Kantorovich and Eulerian minimizers for the initial measure $X_0\sharp\pp$ exist.

\section[PMP for the Lagrangian formulation]{Pontryagin\- maximum\- principle\- for\- the\- Lagrangian\- formulation\- of\- mean\- field\- type\- optimal\- control\- problem}\label{sect:PMP_Lagrangian} In this section, we assume that we are given with a  standard probability space $(\Omega,\mathcal{F},\pp)$, initial assignment $X_0\in L^p(\Omega,\mathcal{F},\pp;\rd)$, and a control process $(X^*,u^*)$ defined on this probability space that is a $L^p$-minimizer. Furthermore, $m_0=X_0\sharp\pp$.

To formulate the Pontryagin maximum principle, we define two  Pontryagin functions (Hamiltonians).
\begin{itemize}
	\item A \textit{local Pontryagin function} is a mapping defined for $t\in [0,T]$, $x\in\rd$, $m\in\prd$, $\psi\in \rds$, $u\in U$ by the rule
	\begin{equation}\label{intro:hamiltinian}
		H(t,x,m,\psi,u)\triangleq \psi f(t,x,m,u)-f_0(t,x,m,u).
	\end{equation} 
	\item A \textit{$L^p$-Pontryagin function} is a mapping $\mathbb{H}:[0,T]\times L^p(\Omega,\mathcal{F},\pp;\rd)\times L^q(\Omega,\mathcal{F},\pp;\rds)\times L^p(\Omega,\mathcal{F},\pp;U)\rightarrow\mathbb{R}$ defined by the formula: 
	\[\mathbb{H}(t,X,\Psi,u)\triangleq \expect H(t,X,\Psi,X\sharp \pp,u).\] 
\end{itemize}

Furthermore, let $\mathcal{Y}^q$ be the set of function $\Psi:[0,T]\times\Omega\rightarrow\rds$ such that  $\widehat{\Psi}\in L^q(\Omega,\mathcal{F},\pp;\Gamma^\star)$. As above, we denote by $\widehat{\Psi}$ the mapping assigning to $\omega\in\Omega$ the whole path $\Psi(\cdot,\omega)$. The norm of an element $\Psi\in\mathcal{Y}^q$ is given by the formula:
\[\norm{\Psi}{\mathcal{Y}^q}\triangleq \Big[\expect\norm{\widehat{\Psi}}{\infty}^q\Big]^{1/q}.\] 

Notice that, due to assumption \ref{assumption:sublinear} and the H\"older inequality, 
\begin{equation}\label{ineq:Ham_finite}
	\big|\mathbb{H}(t,X,\Psi,u)\big|<+\infty
\end{equation} for every $t\in [0,T]$, $X\in L^p(\Omega,\mathcal{F},\pp;\rd)$, $\Psi\in L^q(\Omega,\mathcal{F},\pp;\rds)$, $u\in L^p(\Omega,\mathcal{F},\pp;U)$. Moreover, if $X\in\xp$, $u\in \up$ and $\Psi\in \mathcal{Y}^q$, then
\begin{equation}\label{ineq:Ham_int_finite}
	\begin{split}
		\Bigg|\int_0^T\mathbb{H}(t,&X(t),\Psi(t),u(t))dt\Bigg|\\&\leq \mathbb{E}\Bigg[\int_0^T|H(t,X(t),\Psi(t),X(t)\sharp \pp,u(t))|dt\Bigg]<+\infty.\end{split}
\end{equation}

\begin{theorem}\label{th:PMP_Lagrangian}
	Let $(\Omega,\mathcal{F},\pp)$ be a standard probability space, $X_0$ be an initial assignment, $(X^*,u^*)\in\mathcal{A}_L(X_0)$ be a  Pontryagin local $L^p$-minimizer. Then there exists a function $\Psi\in \mathcal{Y}^q$ such that the following conditions hold true:	
	\begin{itemize}
		\item
		costate equation: for $\pp$-a.e. $\omega\in \Omega$, $\Psi(\cdot,\omega)$ solves
		\begin{equation}\label{eq:Psi}
			\begin{split}
				\frac{d}{dt}&\Psi(t,\omega)=-\Psi(t,\omega)\nabla_xf(t,X^*(t,\omega),X^*(t)\sharp \pp,u^*(t,\omega))\\&+\nabla_xf_0(t,X^*(t,\omega),X^*(t)\sharp \pp,u^*(t))\\&-
				\int_\Omega \Psi(t,\omega')\nabla_m f(t,X^*(t,\omega'),X^*(t)\sharp \pp,X^*(t,\omega),u^*(t,\omega'))\pp(d\omega')\\&+
				\int_\Omega \nabla_m f_0(t,X^*(t,\omega'),X^*(t)\sharp \pp,X^*(t,\omega),u^*(t,\omega'))\pp(d\omega'),\end{split}
		\end{equation} 
		\item transversality condition: for $\pp$-a.e. $\omega\in \Omega$,
		\begin{equation}\label{eq:transversality}
			\begin{split}
				\Psi(T,\omega)=-\nabla_x&\sigma(X^*(T,\omega),X^*(T)\sharp \pp)\\&- \int_\Omega\nabla_m\sigma(X^*(T,\omega'),X^*(T)\sharp \pp,X^*(T,\omega))\pp(d\omega').\end{split}
		\end{equation} 
		\item maximization of the Hamiltonian condition: at almost every   point $s\in [0,T]$,
		\begin{equation}\label{condition:maximum_integral}
			\mathbb{H}(s,X^*(s),\Psi(s),u^*(s))=\max_{\nu\in L^p(\Omega,\mathcal{F},\pp;U)}\mathbb{H}(s,X^*(s),\Psi(s),\nu)
		\end{equation} or, equivalently,  
		\begin{equation}\label{condition:maximum_local}
			\begin{split}
				H(s,X^*(s),\Psi(s),&X^*(s)\sharp \pp,u^*(s))\\&=\max_{u\in U}H(s,X^*(s),\Psi(s),X^*(s)\sharp \pp,u)\ \ \pp\text{-a.s.}\end{split}
		\end{equation}
	\end{itemize}
\end{theorem}

Proposition~\ref{prop:minimizers} and Theorem~\ref{th:PMP_Lagrangian} imply the following. 

\begin{corollary}\label{cor:PMP_w_min} The conclusion of Theorem~\ref{th:PMP_Lagrangian} holds true in the cases when $(X^*,u^*)$ is a  Pontryagin local $W_p$-minimizer, a  strong local $L^p$-minimizer or  strong local $W_p$-minimizer.
\end{corollary}

\begin{remark}\label{rmk:Hamiltonian} Computing the derivatives according to the formulae given in Propositions~\ref{prop:computation:int_phi},~\ref{prop:derivative_push_forward}, we arrive at the following the system on state and costate variables in the Hamiltonian form:
	\[\begin{split}
		&\frac{d}{dt}X^*(t)=\nabla_\Psi\mathbb{H}(t,X^*(t),\Psi(t),u^*(t)),\ \ X(0)=X_0,\\ 
		&\frac{d}{dt}\Psi(t)=-\nabla_X\mathbb{H}(t,X^*(t),\Psi(t),u^*(t)),\ \ \Psi(T)=-\nabla_X\Sigma(X^*(T)).
	\end{split}\] Here, $\nabla_X\mathbb{H}$, $\nabla_\Psi\mathbb{H}$ stands for the derivatives w.r.t. $X\in L^p(\Omega,\mathcal{F},\pp;\rd)$ and $\Psi\in L^q(\Omega,\mathcal{F},\pp;\rd)$. Additionally, 
	\[\Sigma(X)\triangleq \expect \sigma(X,X\sharp\pp).\] 
	
	Notice that this representation looks like  a Pontryagin maximum principle for  processes defined on the Banach space $L^p(\Omega,\mathcal{F},\pp;\rd)$ in the case where the  controls are defined on $L^p(\Omega,\mathcal{F},\pp;U)$. In the paper, we do not rely on this reduction to a control problem in the Banach spaces due to the fact that this way requires conditions those are stronger than~\ref{assumption:Uclosed}--\ref{assumption:derivative_sigma} (see \cite{Kipka2014, Kipka2015, Krastanov2011, Krastanov2015}). In particular, these papers requires the uniform (or even Lipschitz) continuity of the Frechet derivative of the  functions $f(t,X,X\sharp\pp,u)$, $f_0(t,X,X\sharp\pp,u)$ w.r.t. $X\in L^p(\Omega,\mathcal{F},\pp;\rd)$. At the same time, in our setting, these functions are only continuous. Therefore, we provide a direct proof that essentially relies on the definition of derivative with respect to a probability measure and the tools of measure theory.
\end{remark}

\begin{remark}\label{remark:finite}
	To compare Theorem~\ref{th:PMP_Lagrangian} with the finite dimensional PMP, one can consider the system of $N$ identical agents assuming that
	\begin{itemize}
		\item the state of the system is described by a vector $\mathbbm{x}=(x_1,\ldots,x_N)$, where $x_i\in\rd$;
		\item the instantaneous control is given by a vector of controls $\mathbbm{u}=(u_1,\ldots,u_N)$;
		\item the dynamics of each agent is governed by the equation:
		\[\frac{d}{dt}x_i(t)=f\Bigg(t,x_i(t),\frac{1}{N}\sum_{j=1}^N\delta_{x_j(t)},u_i(t)\Bigg);\]
		\item the objective functional is equal to
		\[\begin{split}\sum_{i=1}^{N}\sigma\Bigg(x_i&(T),\frac{1}{N}\sum_{j=1}^N\delta_{x_j(T)}\Bigg)\\&+\int_0^T\sum_{i=1}^Nf_0\Bigg(t,x_i(t),\frac{1}{N}\sum_{j=1}^N\delta_{x_j(t)},u_i(t)\Bigg)dt.\end{split}\]
	\end{itemize}
	Choosing $\Omega=\{1,\ldots,N\}$, $\mathcal{F}$ to be the family of all subsets of $\{1,\ldots,N\}$ and let $\pp$ be such that $\pp(\{i\})=1/N$, one can reduce such system to the Lagrangian formulation of mean field type control problem. Moreover, applying Theorem~\ref{th:PMP_Lagrangian}, we derive the necessary condition on a Pontryagin minimizer in this finite agent control problem which  coincides with finite dimensional PMP for the Pontryagin function 
	\[{\mathbbl{H}}(t,\mathbbm{x},\bbpsi,\mathbbm{u})\triangleq \sum_{i=1}^{N}\psi_i f\Bigg(t,x_i,\frac{1}{N}\sum_{j=1}^N\delta_{x_j},u_i\Bigg)-\sum_{i=1}^{N}f_0\Bigg(t,x_i,\frac{1}{N}\sum_{j=1}^N\delta_{x_j},u_i\Bigg)\] and the terminal payoff
	\[\bbsigma(\mathbbm{x})\triangleq \sum_{i=1}^N\sigma\Bigg(x_i,\frac{1}{N}\sum_{j=1}^N\delta_{x_j}\Bigg).\] Above, we used the vector $\bbpsi=(\psi_1,\ldots,\psi_N)$ assuming that $\psi_i\in \rds$.  
\end{remark}

\section{Spike variations}\label{sect:Spike}
In this section, we introduce and discuss spike variations of the Lagrangian control processes which play a crucial role in the proof of Pontryagin maximum principle in the Lagrangian form. First, let us formulate the following property.

\begin{proposition}\label{prop:N_T}
	There exist  sets $\mathcal{N}\subset L^p(\Omega,\mathcal{F},\pp;U)$ and $\mathcal{T}\subset [0,T]$ such that 	 $\mathcal{N}$ is countable, dense in $L^p(\Omega,\mathcal{F},\pp;U)$, $\lambda([0,T]\setminus \mathcal{T})=0$ and, for every $s\in \mathcal{T}$ and $\nu\in\mathcal{N}$, the following properties hold true:
	\[     \norm{u^*(s)}{L^p}<+\infty,\] 
	\begin{equation}\label{equality:f_star}\begin{split}
			\lim_{h\downarrow 0}\frac{1}{h}\int_s^{s+h}\Big\| f( t,&X^*( t),X^*( t)\sharp \pp,u^*( t))\\&-f(s,X^*(s),X^*(s)\sharp \pp,u^*(s))\Big\|_{L^p}d t=0,\end{split}\end{equation}
	\begin{equation}\label{equality:f_0_star}\begin{split}	\lim_{h\downarrow 0}\frac{1}{h}\int_s^{s+h}\expect\Big| f_0( t,&X^*( t),X^*( t)\sharp \pp,u^*( t))\\&-f_0(s,X^*(s),X^*(s)\sharp \pp,u^*(s))\Big|d t=0;\end{split}  \end{equation}
	\begin{equation}\label{equality:f_nu}\begin{split}
			\lim_{h\downarrow 0}\frac{1}{h}\int_s^{s+h}\Big\| f( t,&X^*( t),X^*( t)\sharp \pp,\nu)\\&-f(s,X^*(s),X^*(s)\sharp \pp,\nu)\Big\|_{L^p}d t=0,\end{split}\end{equation}
	\begin{equation}\label{equality:f_0_nu}\begin{split}	\lim_{h\downarrow 0}\frac{1}{h}\int_s^{s+h}\expect\Big| f_0( t,&X^*( t),X^*( t)\sharp \pp,\nu)\\&-f_0(s,X^*(s),X^*(s)\sharp \pp,\nu)\Big|d t=0.\end{split}  \end{equation}
\end{proposition}
This statement is proved in \ref{appendix:Lebesgue}.

Let $s\in \mathcal{T}$, $\nu\in \mathcal{N}$. For $h\in [0,T-s]$, set 
\begin{equation*}\label{intro:variation}
	u^h_\nu(t,\omega)\triangleq\left\{ \begin{array}{ll}
		u^*(t,\omega), & t\in [0,s), \\
		\nu(\omega), & t\in [s,s+h), \\
		u^*(t,\omega) & t\in [s+h,T].
	\end{array}
	\right.
\end{equation*} Notice that $u^0_\nu\equiv u^*$.

Furthermore, let us consider the following system of ODEs:
\begin{equation}\label{eq:ODE_omega}
	\frac{d}{dt} Z_\nu^h(t,\omega)= f(t,Z^h_\nu(t,\omega),Z^h_\nu(t)\sharp \pp,u^h_\nu(t,\omega)),\ \ Z^h_\nu(0,\omega)=X_0(\omega).
\end{equation}
\begin{proposition}\label{prop:existence_h}
	For each $h\in [0,T-s]$, there exists a unique solution of~(\ref{eq:ODE_omega}) that lies in $\xp$.
\end{proposition}
The proof of this statement directly follows from \cite[Theorem A.5 and Proposition A.7]{Cavagnari_et_al_approaches} and assumptions \ref{assumption:sublinear},~\ref{assumption:derivative_f}. 

The very construction of $Z_\nu^h$   implies that
\begin{itemize}
	\item $Z^0_\nu(t)=X^*(t),$ $t\in [0,T];$
	\item  $Z^h_\nu(t)=X^*(t)$ when $t\in [0,s]$ $\pp$-a.s.
\end{itemize}

The following statement provides the estimates of the norm of $Z_\nu^h(t)$ as well as the distance between $Z_\nu^h(t)$ and $X^*(s)$. 
\begin{proposition}\label{prop:Z_h_bounds}
	There exist  constants $C_0$, $C_1$, $C_2$, $\bar{h}$ dependent on $(X^*,u^*)$ and $\nu$ such that
	\begin{enumerate}
		\item $\|Z_\nu^h(t)\|_{L^p}\leq C_0$ for $h\in [0,T-s]$, $t\in [s,T]$;
		\item $\|Z_\nu^h(t)-X^*(s)\|_{L^p}\leq C_1(t-s)$ for $t\in [s,s+h]$;
		\item $\|Z_\nu^h(t)-X^*(t)\|_{L^p}\leq C_2 h$ when $h<\bar{h}$, $t\in [0,T]$.
	\end{enumerate}
\end{proposition}

The proof follows the standard scheme proposed for the Pontryagin maximum principle for the finite-dimensional case. However, it contains some technical details. Thus, we put it in~\ref{appendix:sub:prior}. 

\begin{corollary}\label{corollary:almost everywhere}
	There exists a sequence $\{h_n\}_{n=1}^\infty$ such that
	\begin{enumerate}
		\item $\{Z^{h_n}_\nu(t,\omega)\}_{n=1}^\infty$ converges to $X^*(t,\omega)$ for $\lambda\otimes \pp$-a.e. $(t,\omega)\in [s,T]\times\Omega$;
		\item $\{Z^{h_n}_\nu(T,\omega)\}_{n=1}^\infty$ converges to $X^*(T,\omega)$ for $\pp$-a.e. $\omega\in \Omega$.
	\end{enumerate}
\end{corollary}
\begin{proof}
	Proposition~\ref{prop:Z_h_bounds} implies that
	\[\int_0^T\int_\Omega\|Z^h(t,\omega)-X^*(t,\omega)\|^p\pp(d\omega)dt\leq C_2^pTh^p.\] Therefore, due to  \cite[Theorem 4.5.4]{Bogachev}, the family $\{Z^h_\nu\}_{h\in (0,\bar{h}]}$ converges to the function $X^*$ in the measure $\lambda\otimes\pp$ as $h\rightarrow 0$. This and \cite[Theorem 2.2.5]{Bogachev} give that there exists a sequence $\{h_n\}_{n=1}^\infty$ converging to zero such that $\{Z^{h_n}_\nu(t,\omega)\}_{n=1}^\infty$ converges to $X^*(t,\omega)$ for $\lambda\otimes \pp$-a.e. $(t,\omega)$. This proves the first statement of the corollary. To prove the second statement, it suffices to consider the sequence $\{Z^{h_n}_\nu(T)\}$ that converges to $X^*(T)$ in $L^p$ and, thus, in probability $\pp$ and find the subsequence still denoted by $\{h_n\}$ such that $Z_\nu^{h_n}(T)\rightarrow X^*(T)$ $\pp$-a.s.
\end{proof}

Below, we fix the sequence $\{h_n\}_{n=1}^\infty$ satisfying the statements of Corollary~\ref{corollary:almost everywhere}. 

Now let us denote
\begin{equation}\label{intro:variations_f}
	\begin{split}
		\Delta^s_\nu f^*(\omega)\triangleq f(s,&X^*(s,\omega),X^*(s)\sharp \pp,\nu(\omega)) \\&- f(s,X^*(s,\omega),X^*(s)\sharp \pp,u^*(s,\omega)),\end{split}
\end{equation}
\begin{equation}\label{intro:variations_f_0}
	\begin{split}
		\Delta^s_\nu f_0^*(\omega)\triangleq f_0(s,&X^*(s,\omega),X^*(s)\sharp \pp,\nu(\omega)) \\&- f_0(s,X^*(s,\omega),X^*(s)\sharp \pp,u^*(s,\omega))\end{split}
\end{equation} and consider the following  system of ODEs on $[s,T]$:
\begin{equation}\label{eq:derivative:Y}
	\begin{split}
		\frac{d}{dt}Y_\nu(t,\omega)=\nabla_x &f(t,X^*(t,\omega),X^*(t)\sharp \pp,u^*(t,\omega))\cdot Y_\nu(t,\omega)\\+\int_{\Omega}\nabla_m &f(t,X^*(t,\omega),X^*(t)\sharp \pp,X^*(t,\omega'),u^*(t,\omega))Y_\nu(t,\omega') \pp(d\omega'), \\
		Y_\nu(s,\omega)=\Delta^s_\nu f^*&(\omega).
	\end{split}
\end{equation}
\begin{proposition}\label{prop:Y_nu} System~(\ref{eq:derivative:Y}) admits a unique solution $Y_\nu:[s,T]\times\Omega\rightarrow\rd$ such that the mapping assigning to $\omega\in\Omega$ the whole path $Y_\nu(\cdot,\omega)$ lies  in the space $L^p(\Omega,\mathcal{F},\pp;C([s,T];\rd))$. Moreover, there exists a constant $C_3$ such that, for all $t\in [s,T]$,
	\[\|Y_\nu(t)\|_{L^p}\leq C_3.\]
\end{proposition}
\begin{proof}
	The existence and uniqueness result for $Y_\nu(\cdot)$ directly follows from \cite[Theorem A.5 and Proposition A.7]{Cavagnari_et_al_approaches} and the boundness of $\nabla_x f$ and $\nabla_m f$. Furthermore, due to assumption~\ref{assumption:derivative_f}, we have that
	\[\|Y_\nu(t)\|_{L^p}\leq \|\Delta^s_\nu f^*\|_{L^p}+(C_x+C_m)\int_s^t\|Y_\nu(\tau)\|_{L^p}d\tau.\] Applying the Gronwall's inequality we obtain that $\|Y_\nu(t)\|_{L^p}$ is uniformly bounded.
\end{proof}

Below, if $\varrho:\Omega\times\Omega\rightarrow \rds$, $\xi:\Omega\rightarrow\rd$ are measurable, then we denote by $\varrho\diamond\xi$ their partial inner product that is a measurable function from $\Omega$ to $\mathbb{R}$ defined by the rule:
\begin{equation}\label{intro:partial_product}(\varrho\diamond\xi)(\omega)\triangleq \int_{\Omega}\varrho(\omega,\omega')\xi(\omega')\pp(d\omega').\end{equation} 
We will use the same notation if $\varrho:[0,T]\times\Omega\times\Omega\rightarrow \rds$, $\xi:[0,T]\times\Omega\rightarrow\rd$, i.e., in this case
\begin{equation}\label{intro:partial_t_product}(\varrho\diamond\xi)(t,\omega)\triangleq \int_{\Omega}\varrho(t,\omega,\omega')\xi(t,\omega')\pp(d\omega').\end{equation} 

To shorten the notation, we denote, for $t\in [0,T]$, $\omega,\omega'\in\Omega$,
\begin{equation}\label{intro:f_x}
	f^*_x(t,\omega)\triangleq \nabla_xf(t,X^*(t,\omega),X^*(t)\sharp \pp,u^*(t,\omega)),
\end{equation}
\begin{equation*}\label{intro:f_m}
	f^*_m(t,\omega,\omega')\triangleq \nabla_m f(t,X^*(t,\omega),X^*(t)\sharp \pp,X^*(t,\omega'),u^*(t,\omega)).
\end{equation*}
Furthermore, we use convention \eqref{intro:partial_t_product}:
\begin{equation}\label{intro:langle}
	(f_m^*\diamond Y_\nu) (t,\omega)\triangleq \int_{\Omega}f^*_m(t,\omega,\omega')Y_\nu(t,\omega')\pp(d\omega').
\end{equation}

\begin{proposition}\label{prop:derivative_y_h} The following convergence holds true:
	\[\frac{1}{h_n}\|Z^{h_n}_\nu(t)-X^*(t)-h_nY_\nu(t)\|_{L^p}\rightarrow 0\text{ as }n\rightarrow\infty\]  uniformly for $t\in (s,T]$.
\end{proposition} 
The proof is given in \ref{appendix:sub:process}.


Below we evaluate the variation of the running cost. For shortness, we will use the following notation:
\begin{equation}\label{intro:f_0_x}
	f_{0,x}^*(t,\omega)\triangleq\nabla_x f_0(t,X^*(t,\omega),X^*(t)\sharp \pp,u^*(t,\omega)),
\end{equation}
\begin{equation*}\label{intro:f_0_m}
	f_{0,m}^*(t,\omega,\omega')\triangleq\nabla_m f_0(t,X^*(t,\omega),X^*(t)\sharp \pp,X^*(t,\omega'),u^*(t,\omega)).
\end{equation*}

As above, due to  convention \eqref{intro:partial_t_product}, we denote
\begin{equation}\label{intro:langle_0}
	(f_{0,m}^*\diamond Y_\nu) (t,\omega)\triangleq \int_\Omega f_{0,m}^*(t,\omega,\omega') Y_\nu(t,\omega')\pp(d\omega').
\end{equation}


\begin{proposition}\label{prop:derivative_integral} The following equality holds true:
	\[\begin{split}
		\lim_{n\rightarrow\infty}\frac{1}{h_n}\bigg[\int_0^T\mathbb{E}f_0(t,&Z^{h_n}_\nu(t),Z^{h_n}_\nu(t)\sharp \pp,u^{h_n}_\nu(t))dt\\&-\int_0^T\mathbb{E}f_0(t,X^*(t),X^*(t)\sharp \pp,u^*(t))dt\bigg]\\=
		\mathbb{E}\Delta^s_\nu f^*_0&+\int_s^T\mathbb{E}[f_{0,x}^*(t)Y_\nu(t)+ (f_{0,m}^*\diamond Y_\nu)(t)]dt.
	\end{split}\]
\end{proposition}
We prove this statement in \ref{appendix:sub:integral}.

Now let us examine the limit behavior of the terminal payoff. To simplify notation, put
\begin{equation}\label{intro:sigma_x}
	\sigma_x^*(\omega)\triangleq\nabla_x \sigma(X^*(T,\omega),X^*(T)\sharp \pp),
\end{equation}
\begin{equation}\label{intro:sigma_m}
	\sigma_m^*(\omega,\omega')\triangleq\nabla_m \sigma(X^*(T,\omega),X^*(T)\sharp \pp,X^*(T,\omega')).
\end{equation}
Recall designation \eqref{intro:partial_product}. In this case, we have that
\begin{equation}\label{intro:langle_sigma}
	(\sigma_m^*\diamond Y_\nu) (\omega)\triangleq \int_\Omega \sigma_m^*(\omega,\omega') Y_\nu(\omega')\pp(d\omega').
\end{equation}

\begin{proposition}\label{prop:derivative_terminal} 
	\[\begin{split}
		\lim_{n\rightarrow\infty}\frac{1}{h_n}\mathbb{E}|\sigma(Z^{h_n}_\nu(T)&,Z^{h_n}_\nu(T)\sharp \pp)\\&-\sigma(X^*(T),X^*(T)\sharp \pp)-h_n[\sigma_x^*+ \sigma_m^*\diamond Y_\nu]|=0.\end{split}\]
\end{proposition}
We omit the proof of this proposition since it mimics   Steps 4 and 5 in the proof of Proposition~\ref{prop:derivative_integral} (see \ref{appendix:sub:integral}) and relies on the fact that $\{h_n\}$ is such that $\{Z^{h_n}_\nu(T)\}$ converges to $X^*(T)$ $\pp$-a.e. (see Corollary~\ref{corollary:almost everywhere}).

\section{Proof of the Pontryagin maximum principle in the Lagrangian form}\label{sect:proof_PMP_Lagrangian}
\begin{proof}[Proof of Theorem~\ref{th:PMP_Lagrangian}]
	In the proof we use notation introduced in~\eqref{intro:variations_f}--\eqref{eq:derivative:Y} and (\ref{intro:f_x})--(\ref{intro:langle_sigma}). Moreover, we assume that $s\in \mathcal{T}$ satisfies conditions of Proposition \ref{prop:N_T}, while the sequence $\{h_n\}_{n=1}^\infty$ is chosen such that conditions of Corollary \ref{corollary:almost everywhere} holds true. 
	
	First, we consider that case where $\nu \in\mathcal{N}$ that was also introduced in Proposition~\ref{prop:N_T}.
	
	By the third statement of Proposition~\ref{prop:Z_h_bounds},
	\[\|Z^{h_n}_\nu(t)-X^*(t)\|_{L^p}\leq C_1h_n,\] while  $Z^{h_n}(0)=X^*(0)$. Moreover, 
	\[(\lambda\otimes\pp)(u^{h_n}_\nu\neq u^*)\leq(\lambda\otimes\pp)([s,s+h_n]\times \Omega)=h_n.\]
	Thus, the assumption that $(X^*,u^*)$ is a  Pontryagin local $L^p$-minimizer at $X_0$ implies that, for sufficiently large $n$,
	\[J_L(X^*,u^*)\leq J_L(Z^{h_n}_\nu,u^{h_n}_\nu).\] This  yields the inequality
	\begin{equation}\label{ineq:0_leq_J_L}
		0\leq \lim_{n\rightarrow\infty}\frac{1}{h_n}[J_L(Z^{h_n}_\nu,u^{h_n}_\nu)-J_L(X^*,u^*)].
	\end{equation} The existence of  the limit above is due to Propositions ~\ref{prop:derivative_integral},~\ref{prop:derivative_terminal}. Using them and definition of the functional $J_L$ (see~\eqref{intro:def_J_L}), we compute 
	\begin{equation}\label{equality:limit_h_n_Z_nu}\begin{split}
			\lim_{n\rightarrow\infty}\frac{1}{h_n}\big[J_L(Z^{h_n}_\nu&,u^{h_n}_\nu)-J_L(X^*,u^*)\big]\\=\int_\Omega\Delta^s_\nu &f^*_0(\omega)\pp(d\omega)+\int_s^T\int_\Omega f_{0,x}^*(t,\omega)Y_\nu(\omega)\pp(d\omega)dt\\&+ \int_s^T\int_\Omega\int_\Omega f_{0,m}^*(t,\omega,\omega')Y_\nu(t,\omega')\pp(d\omega')\pp(d\omega)dt\\&+
			\int_\Omega\sigma_x(\omega)Y_\nu(T,\omega)\pp(d\omega)\\&+\int_{\Omega}\int_\Omega\sigma^*_m(\omega,\omega')Y_\nu(T,\omega')\pp(d\omega')\pp(d\omega).
		\end{split} 
	\end{equation}
	Using the Fubini theorem and renaming variables, we have, for each $t\in [s,T]$, 
	\[\begin{split}\int_\Omega\int_\Omega f_{0,m}^*&(t,\omega,\omega')Y_\nu(t,\omega')\pp(d\omega')\pp(d\omega)\\&=\int_\Omega\int_\Omega f_{0,m}^*(t,\omega,\omega')Y_\nu(t,\omega')\pp(d\omega)\pp(d\omega')\\
		&=\int_\Omega\Bigg[\int_\Omega f_{0,m}^*(t,\omega',\omega)\pp(d\omega')\Bigg]Y_\nu(t,\omega)\pp(d\omega).\end{split}\] Similarly, 
	\[
	\begin{split}
		\int_{\Omega}\int_\Omega\sigma^*_m&(\omega,\omega')Y_\nu(T,\omega')\pp(d\omega')\pp(d\omega)\\&=\int_{\Omega}\int_\Omega\sigma^*_m(\omega,\omega')Y_\nu(T,\omega')\pp(d\omega)\pp(d\omega')\\&=
		\int_{\Omega}\Bigg[\int_\Omega\sigma^*_m(\omega',\omega)\pp(d\omega')\Bigg]Y_\nu(T,\omega)\pp(d\omega).
	\end{split}
	\]
	Substituting these two equalities into \eqref{equality:limit_h_n_Z_nu}, we obtain
	\[\begin{split}
		\lim_{n\rightarrow\infty}\frac{1}{h_n}[J_L(Z^{h_n}_\nu,u^{h_n}_\nu)-J_L(X^*,u^*)]&=\int_\Omega\Delta^s_\nu f^*_0(\omega)\pp(d\omega)\\+\int_s^T\int_\Omega \bigg[f_{0,x}^*(t,\omega)+\int_\Omega f_{0,m}^*&(t,\omega',\omega)\pp(d\omega')\bigg]Y_\nu(t,\omega)\pp(d\omega)dt\\+
		\int_\Omega\bigg[\sigma_x(\omega)+&\int_\Omega\sigma^*_m(\omega',\omega)\pp(d\omega')\bigg]Y_\nu(T,\omega)\pp(d\omega).
	\end{split} \]
	We define $\Psi$ as the solution of the following boundary value problem:
	\begin{equation}\label{intro:Psi}\begin{split}
			\frac{d}{dt}\Psi(t,\omega)=-\Psi(t,\omega) f_{x}^*(t,&\omega)-\int_\Omega \Psi(t,\omega')f_{m}^*(t,\omega',\omega)\pp(d\omega')\\&+ f_{0,x}^*(t,\omega)+\int_\Omega f_{0,m}^*(t,\omega',\omega)\pp(d\omega'),\\ 
			\Psi(T,\omega)=-\sigma_x(\omega)-\int_\Omega\sigma&{}^*_m(\omega',\omega)\pp(d\omega').
	\end{split}\end{equation} 
	
	The existence and uniqueness of a function $\Psi\in \mathcal{Y}^q$ solving~\eqref{intro:Psi} can be obtained from \cite[Theorem A.5 and Proposition A.7]{Cavagnari_et_al_approaches} due to the inclusion that $X^*\in\xp$ and assumptions \ref{assumption:derivative_f_0}, \ref{assumption:derivative_sigma} those imply the fulfillment of conditions of \cite[Theorem A.5 and Proposition A.7]{Cavagnari_et_al_approaches} for the exponent dual to $p$.  Notice that the choice of $\Psi$ gives that the costate equation and the transversality condition hold true.
	
	Now, let us consider the maximization of the Hamiltonian condition.
	Expressing $f_{0,x}^*(t,\omega)+\int_\Omega f_{0,m}^*(t,\omega',\omega)\pp(d\omega')$ from~\eqref{intro:Psi} and changing the order of integration once more, we have
	\[\begin{split}
		\lim_{n\rightarrow\infty}\frac{1}{h_n}[J_L&(Z^{h_n}_\nu,u^{h_n}_\nu)-J_L(X^*,u^*)]\\=\int_\Omega\Delta^s_\nu &f^*_0(\omega)\pp(d\omega)+\int_\Omega\int_s^T\frac{d}{dt}\Psi(t,\omega)Y_\nu(t,\omega)dt\pp(d\omega)
		\\&+\int_\Omega\int_s^T 
		\Psi(t,\omega) f_{x}^*(t,\omega)Y_\nu(t,\omega)dt\pp(d\omega)\\
		&+\int_\Omega\int_s^T\Psi(t,\omega)\int_\Omega f_{m}^*(t,\omega,\omega')Y_\nu(t,\omega')\pp(d\omega')dt\pp(d\omega)\\&+
		\int_\Omega\bigg[\sigma_x(\omega)+\int_\Omega\sigma^*_m(\omega',\omega)\pp(d\omega')\bigg]Y_\nu(T,\omega)\pp(d\omega).
	\end{split} \] Taking into account the fact that
	\[\frac{d}{dt}Y_\nu(t,\omega)=f^*_x(t,\omega)Y_\nu(t,\omega)+\int_\Omega f_m^*(t,\omega,\omega')Y_\nu(t,\omega')\pp(d\omega'),\] we arrive at the following equality
	\[\begin{split}
		\lim_{n\rightarrow\infty}\frac{1}{h_n}[J_L&(Z^{h_n}_\nu,u^{h_n}_\nu)-J_L(X^*,u^*)]\\=\int_\Omega\Delta^s_\nu &f^*_0(\omega)\pp(d\omega)+\int_\Omega\int_s^T\frac{d}{dt}\Psi(t,\omega)Y_\nu(t,\omega)dt\pp(d\omega)
		\\&+\int_\Omega\int_s^T\Psi(t,\omega)\frac{d}{dt}Y_\nu(t,\omega)dt\pp(d\omega)\\&+
		\int_\Omega\bigg[\sigma_x(\omega)+\int_\Omega\sigma^*_m(\omega',\omega)\pp(d\omega')\bigg]Y_\nu(T,\omega)\pp(d\omega).
	\end{split} \]
	Since $\Psi(T,\omega)=-\sigma_x(\omega)-\int_\Omega\sigma^*_m(\omega',\omega)\pp(d\omega')$, $Y_\nu(s,\omega)=\Delta_\nu^s f^*(\omega)$,  the integration by part formula yields that
	\[\lim_{n\rightarrow\infty}\frac{1}{h_n}[J_L(Z^h_\nu,u^h_\nu)-J_L(X^*,u^*)]=\expect[\Delta^s_\nu f^*_0-\Psi(s)\Delta^s_\nu f^*(s)].\] Recall that (see~(\ref{ineq:0_leq_J_L})) this limit is nonnegative, while (see~(\ref{intro:hamiltinian}),~(\ref{intro:variations_f}),~(\ref{intro:variations_f_0}))
	\[\begin{split}\Delta^s_\nu f^*_0(\omega)-\Psi(s,\omega)&\Delta^s_\nu f^*(\omega)\\=H(s,X^*(&s,\omega),\Psi(s,\omega),X^*(s)\sharp \pp,u^*(s,\omega))\\&-H(s,X^*(s,\omega),\Psi(s,\omega),X^*(s)\sharp \pp,\nu(\omega))\end{split}\] Hence, for each $\nu\in\mathcal{N}$,
	\begin{equation}\label{ineq:Hamiltonian_integral_N}
		\expect H(s,X^*(s),\Psi(s),X^*(s)\sharp \pp,u^*(s))\geq \expect H(s,X^*(s),\Psi(s),X^*(s)\sharp \pp,\nu).
	\end{equation} 
	
	Now let us derive the integral form of  maximization condition (see~\eqref{condition:maximum_integral}). 
	
	If $\nu$ is an arbitrary element of $L^p(\Omega,\mathcal{F},\pp;U)$, then, by construction of the set $\mathcal{N}$, there exists a sequence $\{\nu_k\}_{k=1}^\infty\subset \mathcal{N}$ that converges to $\nu$ in $L^p(\Omega,\mathcal{F},\pp;U)$. From \cite[Theorems 4.5.4, Theorem 2.2.5(i)]{Bogachev}, without loss of generality, we can assume that $\{\nu_k\}_{k=1}^\infty$ converge to $\nu$ $\pp$-a.s. This implies that 
	\begin{equation}\label{convergence:H_k}
		\begin{split}
			H(s,X^*(s),\Psi(s),&X^*(s)\sharp \pp,\nu_k)\rightarrow\\ &H(s,X^*(s),\Psi(s),X^*(s)\sharp \pp,\nu)\ \ \text{ as }n\rightarrow\infty,\ \  \pp\text{-a.s.}\end{split}
	\end{equation}
	Furthermore, denote 
	\[\begin{split} \overline{H}_k\triangleq C_\infty^0&(1+\|X^*\|_{\xp}^p)\\&+ C_\infty(1+\|X^*\|_{\xp})\|\Psi(s)\|+2q^{-1}C_\infty\|\Psi(s)\|^q \\&+(C_\infty p^{-1}+C_\infty^0) \|X^*(s)\|^p+(C_\infty p^{-1}+C_\infty^0)\|\nu_k\|.\end{split}\] 
	Notice that \cite[Theorem 4.5.4]{Bogachev} implies that the sequence of random variables $\{\overline{H}_k\}_{k=1}^\infty$ is uniformly integrable.
	Due to assumption \ref{assumption:sublinear}, inequality~\eqref{ineq:M_X_t_X_xp} and the Young's inequality, we have that
	\begin{equation}\label{ineq:H_b}
		\Big|H(s,X^*(s),\Psi(s),X^*(s)\sharp \pp,\nu_k)\Big| \leq \overline{H}_k.\end{equation}  Recall that the uniform integrability of a sequence of functions is equivalent to the fact that the $L^1$-norm of  function from this sequence are uniformly bounded while the integrals are uniformly absolutely continuous \cite[Proposition 4.5.3]{Bogachev}. Using this fact and \eqref{ineq:H_b}, we have that the sequence $\{H(s,X^*(s),\Psi(s),X^*(s)\sharp \pp,\nu_k)\}_{k=1}^\infty$ is uniformly integrable. Therefore, the convergence of the sequence $\{H(s,X^*(s),\Psi(s),X^*(s)\sharp \pp,\nu_k)\}_{k=1}^\infty$ to $H(s,X^*(s),\Psi(s),X^*(s)\sharp \pp,\nu)$ $\pp$-a.s. yields (see \cite[Theorem 4.5.4]{Bogachev}) that
	\[\begin{split}\expect H(s,X^*(s),\Psi(s),&X^*(s)\sharp \pp,\nu_k)\rightarrow\\ &\expect H(s,X^*(s),\Psi(s),X^*(s)\sharp \pp,\nu)\ \ \text{ as }n\rightarrow\infty .\end{split}\] This and \eqref{ineq:Hamiltonian_integral_N} imply that, for each $\nu\in L^p(\Omega,\mathcal{F},\pp;U)$,
	\begin{equation}\label{ineq:Hamiltonian_integral}
		\expect H(s,X^*(s),\Psi(s),X^*(s)\sharp \pp,u^*(s))\geq \expect H(s,X^*(s),\Psi(s),X^*(s)\sharp \pp,\nu).
	\end{equation}
	
	This is integral maximization condition~\eqref{condition:maximum_integral}.
	
	It remains to show that it is equivalent to  local maximization condition~(\ref{condition:maximum_local}). First notice that~(\ref{condition:maximum_local})  obviously implies~\eqref{condition:maximum_integral}. To prove the converse implication~\eqref{condition:maximum_integral}$\Rightarrow$(\ref{condition:maximum_local}), we assume that~\eqref{condition:maximum_integral} is fulfilled, while~(\ref{condition:maximum_local}) is violated.
	Given natural numbers $N$ and $M$, let  $\Xi_{N,M}\in\mathcal{F}$ be  such that, for each $\omega\in\Xi_{N,M}$,
	\begin{equation*}\label{ineq:Hamiltonian_1_k}
		\begin{split}
			H(s,X^*(s,\omega),\Psi(s,\omega),X^*(s)\sharp \pp,u^*(s,\omega)&)+2N^{-1}\\\leq
			\sup\Big\{H(s,X^*(s,\omega),\Psi(s,\omega),&X^*(s)\sharp \pp,u): \\  u\in U,\, &\|u\|^p\leq \|u^*(s,\omega)\|^p+ M\Big\}.
	\end{split}\end{equation*} Since, we assumed that condition \eqref{condition:maximum_local} is violated, it holds that \[\pp\Bigg[\bigcup_{N=1}^\infty\bigcup_{M=1}^\infty \Xi_{N,M}\Bigg]>0.\] This, in particular, means that, for some $N$ and $M$, 
	\[\pp(\Xi_{N,M})>0.\] From now, we fix $N$ and $M$ satisfying this condition. Thus, the  multivalued mapping $\mathcal{G}:\Xi_{N,M}\rightrightarrows U$ that assigns to each $\omega\in \Xi_{N,M}$ the set 
	\[\begin{split}
		\mathcal{G}(\omega)\triangleq \bigg\{u\in U:H(s,X^*(s,\omega)&,X^*(s)\sharp \pp,\Psi(s,\omega),u^*(s,\omega))+N^{-1}\\&\leq  H(s,X^*(s,\omega),X^*(s)\sharp \pp,\Psi(s,\omega),u),\\ \|u\|^p\leq \|u^*&(s,\omega)\|^p+ M\bigg\}\end{split}\] has nonempty images. Moreover, since the mappings those assign to a pair $(\omega,u)\in\Xi_{N,M}\times U$ the values  
	\begin{itemize}
		\item $ H(s,X^*(s,\omega),X^*(s)\sharp \pp,\Psi(s,\omega),u^*(s,\omega))$,
		\item $\|u^*(s,\omega)\|^p$,
		\item  $ H(s,X^*(s,\omega),X^*(s)\sharp \pp,\Psi(s,\omega),u)$
	\end{itemize} respectively are $\mathcal{F}\otimes \mathcal{B}(U)/\mathcal{B}(\mathbb{R})$-measurable, the graph of $\mathcal{G}$ belongs to $\mathcal{F}\otimes \mathcal{B}(U)$.  By the Aumann selection theorem
	\cite[Corollary 18.27]{Infinite_dimensional}, one can find  a function $\hat{u}:\Xi_{N,M}\rightarrow U$ that is $\mathcal{F}/\mathcal{B}(U)$-measurable  for $\pp$-a.e. $\omega\in\Xi_{N,M}$ satisfies the inclusion $\hat{u}(\omega)\in\mathcal{G}(\omega)$.  Another way to find this function is to use \cite[Theorem 6.9.13]{Bogachev} that gives a $\mathcal{F}_{\pp}/\mathcal{B}(U)$-measurable function $\tilde{u}:\Xi_{N,M}\rightarrow U$ that is a selector of $\mathcal{G}$. Recall that  $\mathcal{F}_{\pp}$ stands for the completion of $\mathcal{F}$ w.r.t.  the probability $\pp$. The desired $\mathcal{F}/\mathcal{B}(U)$-measurable function $\hat{u}:\Xi_{N,M}\rightarrow U$ such that, for $\pp$-a.e. $\omega\in\Xi_{N,M}$, $\hat{u}(\omega)=\tilde{u}(\omega)\in\mathcal{G}(\omega)$ exists due to  \cite[Corollary 6.5.6]{Bogachev} and the fact $\mathcal{B}(U)$ is countably generated. The latter directly follows from assumption~\ref{assumption:Uclosed} (see \cite[Example 6.5.2]{Bogachev}). 
	
	Put
	\[\hat{\nu}\triangleq\left\{\begin{array}{cc}
		u^*(s,\omega), & \omega\in \Omega\setminus \Xi_{N,M},  \\
		\hat{u}(\omega), & \omega\in\Xi_{N,M}. 
	\end{array}\right.\] First, notice that $\nu\in L^p(\Omega,\mathcal{F},\pp;U)$. Indeed,
	\[\begin{split}
		\|\hat\nu\|_{L^p}^p=\expect\big[\|u^*(s)\|^p&\mathbbm{1}_{\Omega\setminus\Xi_N}\big] +\expect\Big[\|\hat{u}\|^p\mathbbm{1}_{\Xi_{N,M}}\Big]\\\leq \|&u^*(s)\|_{L^p}^p+M\pp(\Xi_{N,M})<\infty.\end{split}\]
	Furthermore, by construction, we have that
	\[\begin{split}
		H(s,X^*(s,\omega),\Psi(s,&\omega),X^*(s)\sharp \pp,\hat{\nu}(\omega))\\&=H(s,X^*(s,\omega),\Psi(s,\omega),X^*(s)\sharp \pp,u^*(s,\omega))\end{split}\] for $\omega\in \Omega\setminus \Xi_{N,M}$, and
	\[\begin{split}H(s,X^*(s,\omega),\Psi(s,&\omega),X^*(s)\sharp \pp,\hat{\nu}(\omega))\\&\geq H(s,X^*(s,\omega),\Psi(s,\omega),X^*(s)\sharp \pp,u^*(s,\omega))+N^{-1}\end{split}\] if $ \omega\in\Xi_{N,M}$. Hence,
	\[   \begin{split}
		\int_\Omega H(s,&X^*(s,\omega),\Psi(s,\omega),X^*(s)\sharp \pp,u^*(s,\omega))\pp(d\omega)+N^{-1} \pp(\Xi_{N,M})\\\leq &\int_\Omega H(s,X^*(s,\omega),\Psi(s,\omega),X^*(s)\sharp \pp,\hat{\nu}(\omega))\pp(d\omega).
	\end{split} \] Since  $P(\Xi_{N,M})>0$, this contradicts~(\ref{ineq:Hamiltonian_integral}). 
	
	Therefore,~\eqref{condition:maximum_integral} yields that~(\ref{condition:maximum_local}) holds true $\pp$-a.s. This completes the proof.
	
\end{proof}


\section{Kantorovich approach} \label{sect:Kantorovich}
In this section, we introduce the concept of local minima within the Kantorovich formulation of the mean field type control problem (see Definition~\ref{def:Kantorovich_minimizer}), examines its link with the Lagrangian approach (see Theorem~\ref{th:Kantorovich_min_Lagrangian}) and derive the Pontryagin maximum principle for the Kantorovich formalization (see Theorem~\ref{th:PMP_Kantorovich}). Certainly, within this section, we assume that conditions~\ref{assumption:Uclosed}--\ref{assumption:derivative_sigma} are in force. 

\subsection{Kantorovich admissible processes}
\begin{definition}\label{def:Kantorovich_process}
	We say that a pair $(\eta,u_K)$, where $\eta\in \mathcal{P}^p(\Gamma)$, $u_K\in L^p([0,T]\times \Gamma,\mathscr{B}_T\otimes\mathcal{B}(\Gamma),\lambda\otimes\eta;U)$,
	is a Kantorovich control process if 
	\begin{itemize}
		\item $\eta$ is concentrated on the set of absolutely continuous curves;
		\item  $\eta$-a.e. $\gamma\in \Gamma$ satisfies the differential equation
		\begin{equation}\label{eq:gamma}
			\frac{d}{dt}{\gamma}(t)=f(t,\gamma(t),e_t\sharp \eta,u_K(t,\gamma)).
	\end{equation}\end{itemize}
\end{definition}

The outcome of the Kantorovich process $(\eta,u_K)$ is evaluated by the quantity
\[\begin{split}
	J_K(\eta,u_K)\triangleq \int_{\Gamma}\sigma&(e_T(\gamma),e_T\sharp \eta)\eta(d\gamma)\\&+\int_\Gamma\int_0^T f_0(t,e_t(\gamma),e_t\sharp \eta,u_K(t,\gamma))  dt \,\eta(d\gamma)<+\infty.\end{split}\]

\begin{definition}\label{def:Kantorovich_adm}
	Given an initial distribution $m_0\in \mathcal{P}_p(\rd)$, we denote the set of  Kantorovich control processes $(\eta,u_K)$ satisfying the initial condition  $e_0\sharp \eta=m_0$ by $\operatorname{Adm}_K(m_0)$.
\end{definition}

Let us formulate the following concept that provides the link between Kantorovich and Lagrangian approaches. It will play a crucial role in the derivation of the Pontryagin maximum principle within the Kantorovich framework. To introduce it, recall that, when $X\in\xp$,  $\widehat{X}$ stands for the operator that assigns to $\omega\in\Omega$ the whole path $X(\cdot,\omega)$.

\begin{definition}\label{def:Lagrangian_Kantorovich}
	Let $(\eta,u_K)$ be an admissible Kantorovich control process and let $(\Omega,\mathcal{F},\pp)$ be a standard probability space. We say that a Lagrangian control process $(X,u_L)$ defined on $(\Omega,\mathcal{F},\pp)$ realizes $(\eta,u_K)$ if
	\begin{equation}\label{eq:L-Kmu}
		\eta=\widehat{X}\sharp \pp,\end{equation}
	and, for $\pp\text{-a.e. }\omega\in\Omega$ and a.e. $t\in [0,T]$,
	\begin{equation}\label{eq:L-Ku}
		u_L(t,\omega)=u_K(t,\widehat{X}(\omega)).
	\end{equation}
\end{definition}

\begin{proposition}\label{prop:Kantorovich_correspondence}
	Let  $(\eta,u_K)$  be a Kantorovich control process. Assume also that $(\Omega,\mathcal{F},\pp)$ is  a standard probability space such that  at least one of the following conditions is satisfied:
	\begin{itemize}
		\item  the probability  $\pp$ has no atoms, 
		\item   $(\Omega,\mathcal{F},\pp)=(\Gamma,\mathcal{B}(\Gamma),\eta)$. 
	\end{itemize}
	Then, there exists a Lagrangian process $(X,u_L)$ defined on $(\Omega,\mathcal{F},\pp)$ that realizes $(\eta,u_K)$.
	Furthermore, if $(\Omega,\mathcal{F},\pp)=(\Gamma,\mathcal{B}(\Gamma),\eta)$, one can put $X=\operatorname{id}_{\Omega}$ and $u_K=u_L$.
\end{proposition}
\begin{proof}
	In the case where $(\Omega,\mathcal{F},\pp)=(\Gamma,\mathcal{B}(\Gamma),\eta)$, set $\widehat{X}=\operatorname{id}_{\Gamma}$. 
	In the other case, i.e., when   the probability $\pp$ has no atoms, we first claim that the measure $\eta$ is tight. This is due \cite[Theorem 12.7 and Definition 12.2]{Infinite_dimensional} and the fact that
	$\Gamma$ is a Polish space. Thus,   \cite[Theorem 3.1(i)]{Skorokhod_Representation}  gives the existence of a measurable map $\widehat{X}\in B(\Omega,\mathcal{F};\Gamma)$ such that $\widehat{X}\sharp\pp=\eta.$ 
	In both cases,  we obtain
	$e_t\sharp\eta=e_t\sharp (\widehat{X}\sharp \pp)$ and  equality~\eqref{eq:L-Kmu} holds. Furthermore, by construction, we have that $\|\widehat{X}\|_{\infty}$ is finite for $\pp$-a.s. Moreover,
	\begin{equation}\label{ineq:X_contr_eta_L_p}
		\mathbb{E}\|\widehat{X}\|_{\infty}^p=\int_\Gamma \|\gamma\|^p_{\infty}\eta(d\gamma)<\infty.
	\end{equation} Here the last inequality is due to the assumption that $\eta\in \mathcal{P}^p(\Gamma)$. Letting $X(t,\omega)\triangleq \widehat{X}(\omega)(t)$, we construct the desired process $X\in \xp.$
	
	Now, for every $t\in [0,T]$ and $\omega\in \Omega$, set
	\[u_L(t,\omega)\triangleq u_K(t,\widehat{X}(\omega)).\]
	
	Obviously, this control satisfies   equality~\eqref{eq:L-Ku}. Furthermore, from the inclusion
	$ u_K\in L^p([0,T]\times \Gamma,\mathscr{B}_T\otimes\mathcal{B}( \Gamma),\lambda\otimes\eta;U)$  and the equality $\eta=\widehat{X}\sharp\pp$, it follows that
	\[\norm{u_L}{\up}^p
	=\int_{0}^T\int_{\Gamma}\|u_K(\cdot,\gamma)\|^p_{L^p([0,T];U)}\eta(d\gamma)dt
	<+\infty.\]
	Therefore,
	$u_L$ lies in $\up$. 
	
	Finally, let us show that, for $\pp$-a.e. $\omega$, $X(\cdot,\omega)$ satisfies the equation
	\[\frac{d}{dt}X(t,\omega)=f(t,X(t,\omega),X(t)\sharp\pp,u_L(t,\omega)),\] or, equivalently,
	\begin{equation}\label{eq:X_Kantorovich}X(t,\omega)=X(0,\omega)+\int_{0}^t f(\tau,X(\tau,\omega),e_\tau\sharp \eta,u_L(\tau,\omega))d\tau.\end{equation}  The latter follows from the assumption that, for $\eta$-a.e. $\gamma\in\Gamma$ and every $t\in [0,T]$,
	\[\gamma(t)=\gamma(0)+\int_0^tf(\tau,\gamma(\tau),e_\tau\sharp \eta,u_K(\tau,\gamma))d\tau.\]  The inclusions $X\in \xp$, $u_L\in \up$ and the fact that~(\ref{eq:X_Kantorovich}) is fulfilled for $\pp$-a.e. $\omega$ imply that 
	$(X,u_L)$ is an admissible Lagrangian process. By construction, it realizes $(\eta,u_K)$. 
\end{proof}

\subsection{Local minimizers within the Kantorovich approach}
\begin{definition}\label{def:Kantorovich_minimizer}
	A Kantorovich control process $(\eta^*,u^*_K)\in \operatorname{Adm}_K(m_0)$ is called a  strong local minimizer at $m_0$ within the Kantorovich approach if there exists   $\varepsilon>0$ such that $J_K(\eta,u_K)\geq J_K(\eta^*,u^*_K)$ for all  processes $(\eta,u_K)\in \operatorname{Adm}_K(m_0)$ satisfying  $W_p(e_t\sharp\eta,e_t\sharp\eta^*)\leq \varepsilon$ when $t\in [0,T]$.
\end{definition}

The next theorem states the link between local minimizers in the Kantorovich and Lagrangian approaches.

\begin{theorem}\label{th:Kantorovich_min_Lagrangian} Assume that $(\eta^*,u^*_K)$ is a  strong local minimizer in the framework of the Kantorovich approach at $m_0=e_0\sharp\eta^*$.
	Let $(X^*,u^*_L)$ be an admissible Lagrangian process that realizes the Kantorovich process $(\eta^*,u^*_K)$. Then,
	$(X^*,u^*_L)$ is a  strong local $W_p$-minimizer at $m_0$ in the framework of the Lagrangian approach.
\end{theorem}

The proof of this statement relies on Lemma \ref{lm:Kantorovich2_correspondence} and the following definition.

\begin{definition}\label{def:Lagrangian_Kantorovich2}
	Let $(X,u_L)$ be a Lagrangian control process defined on some  standard probability space $(\Omega,\mathcal{F},\pp)$. We say that a Kantorovich control process $(\eta,u_K)$  improves $(X,u_L)$ if it satisfies ~\eqref{eq:L-Kmu} and 
	$J_L(X,u_L)\geq J_K(\eta,u_K)$.
\end{definition}

\begin{lemma}\label{lm:Kantorovich2_correspondence}
	Let   $(X,u_L)$ 
	be  a Lagrangian control process  defined on some standard probability space $(\Omega,\mathcal{F},\pp)$.
	Then, there exists a Kantorovich control process $(\eta,u_K)$  that improves $(X,u_L)$.
\end{lemma}
\begin{proof} We split the proof into four steps. First, we define a distribution on the set of curves. Next, steps 2 and 3 are concerned with constructions of functions $s$ and $v$ those are a.e. on $\Gamma$ and take values in $\Omega$. The function~$s$ will provide pathwise improvement of the strategy, while the function $v$ will be used to control its norm. Finally, on step 4, we combine the Borel modifications of these functions and define a Kantorovich strategy that is admissible and improves the original Lagrangian strategy.
	
	\textit{Step 1.}
	Define the
	probability $\eta\in\mathcal{P}(\Gamma)$ by the rule 
	\[\eta\triangleq \widehat{X}\sharp \pp.\] Since $X\in \xp$, we have that
	\[\eta\in \mathcal{P}^p(\Gamma).\]
	Furthermore, for $t\in [0,T]$, set \[m(t)\triangleq X(t)\sharp\pp.\] 
	By construction,~\eqref{eq:L-Kmu} holds true. 
	
	Notice that 
	\begin{equation}\label{equality:sigma_gamma}
		\expect\sigma(X(T),X(T)\sharp\pp)=\int_{\Gamma} \sigma(e_T(\gamma),m(T))\eta(d\gamma)
	\end{equation}
	Thus, we consider only the running cost below.

	\textit{Step 2.} Recall that  $\mathcal{B}_\eta(\Gamma)$ denotes  the $\eta$-completion of $\mathcal{B}(\Gamma)$. The extension of the measure $\eta$ on $\mathcal{B}_\eta(\Gamma)$ is still denoted by~$\eta$.

	By the disintegration theorem (see \cite[Theorem 5.3.1]{Ambrosio} or \cite[III-70]{Dellacherie_Meyer}), there exists a system of probability measures $\{\pp_\gamma\}_{\gamma\in \Gamma}$ such that, for $\eta$-a.e. $\gamma\in\Gamma$, the probability $\pp_\gamma$ is concentrated on the set $\widehat{X}^{-1}(\gamma)$ and, given a Borel map $\phi:\Omega\rightarrow [0,+\infty]$,
	\begin{equation}\label{equality:disintergration_P}
		\mathbb{E}\phi=\int_{\Gamma}\int_{\widehat{X}^{-1}(\gamma)}\phi(\omega)\pp_\gamma(d\omega)\eta(d\gamma).
	\end{equation}

	Now, for each $\omega\in \Omega$, denote
	\begin{equation}\label{intro:g_omega}
		g(\omega)\triangleq \int_{0}^Tf_0(t,X(t,\omega),m(t),u_L(t,\omega)) dt.
	\end{equation}
	Moreover, put, for   $\gamma\in\Gamma$,
	\begin{equation}\label{intro:g_bar_omega}\bar{g}(\gamma)\triangleq \int_{\widehat{X}^{-1}(\gamma)}g(\omega)\pp_\gamma(d\omega), \end{equation}  
	\begin{equation}\label{intro:l_omega}\mathscr{l}(\gamma)\triangleq \int_{\widehat{X}^{-1}(\gamma)}\|u_L(\cdot,\omega)\|^p_{L^p} \pp_\gamma(d\omega).\end{equation}
	Recall that in the formula above
	\[\|u_L(\cdot,\omega)\|^p_{L^p}=\int_{0}^{T}\|u_L(t,\omega)\|^pdt.\]
	Notice that $\mathscr{l}$ is a Borel measurable map from $\Gamma$ to $[0,+\infty]$ such that $\int_{\Gamma}\mathscr{l}(\gamma)\eta(d\gamma)<+\infty$.
	Let us consider the outcome corresponding to the process $(X,u_L)$
	\[J(X,u_L)= \expect\int_0^T  f_0(t,X(t,\omega),m(t),u_L(t,\omega))dt.\]
	Due to the construction of  the system of measures $\{\pp_\gamma\}_{\gamma\in\Gamma}$ and the definitions of the functions $g$, $\bar{g}$ (see \eqref{intro:g_omega}, \eqref{intro:g_bar_omega}), 
	we have that
	\begin{equation}\label{eq:constB}		J(X,u_L)=	\int_{\Gamma}\int_{\widehat{X}^{-1}(\gamma)}g(\omega)\pp_\gamma(d\omega)\eta(d\gamma)=	\int_{\Gamma}\bar{g}(\gamma)\eta(d\gamma).
	\end{equation}
	Notice that the mappings $\bar{g}$ and $\mathscr{l}$ are defined using the averaging of the functions $g$ and $\|u_L(\cdot)\|^p_{L^p}$ over the set $\widehat{X}^{-1}(\gamma)$ respectively (see \eqref{intro:g_bar_omega}, \eqref{intro:l_omega}). Hence, we have that, for $\eta$-a.e. $\gamma\in\Gamma$, 
	\begin{equation}\label{ineq:g_bar_g}
		\bar{g}(\gamma)\geq \inf_{\omega\in \widehat{X}^{-1}(\gamma)}g(\omega),\qquad 
		\mathscr{l}(\gamma)\geq  \inf_{\omega\in \widehat{X}^{-1}(\gamma)} \|u_L(\cdot,\omega)\|^p_{L^p}.
	\end{equation}
	There exists a Borel set  $\Gamma_+$ such that 
	\begin{itemize}
		\item inequalities~\eqref{ineq:g_bar_g} hold true on it;
		\item the probability $\pp_\gamma$ is concentrated on $\widehat{X}^{-1}(\gamma)$ whenever $\gamma\in\Gamma_+$;
		\item $\eta(\Gamma_+)=\eta(\Gamma)=1$.
	\end{itemize}
	
	Let us introduce  multivalued mappings $V:\Gamma_+\rightrightarrows \Omega$, $S_0:\Gamma_+\rightrightarrows \Omega$ by the following rules: 
	\begin{equation}\label{intro:V}
		V(\gamma)\triangleq\{\omega\in \widehat{X}^{-1}(\gamma):\,  \mathscr{l}(\gamma)\geq  \|u_L(\cdot,\omega)\|^p_{L^p}\},\end{equation}
	\begin{equation*}	S_0(\gamma)\triangleq
		\{\omega\in\widehat{X}^{-1}(\gamma):\, \bar{g}(\gamma)> g(\omega)\}.
	\end{equation*}
	
	Informally, elements of $V(\gamma)$ are labels $\omega$ those generate the curve $\gamma$ and with norms of controls not greater than the averaged norm of controls producing $\gamma$. Simultaneously, $S_0(\gamma)$ contains labels $\omega$ those give outcomes strictly less than the averaged outcome on the labels producing the curve~$\gamma$. It looks that, if one choose a selector $s'(\gamma)\in S_0(\gamma)$ and consider the strategy $(t,\gamma)\mapsto u_L(t,s'(\gamma))$, the corresponding Kantorovich process will improve $(X,u_L)$. The main issues here are that  $S_0(\gamma)$ can be empty  on a set of positive measure, whilst the strategy  $(t,\gamma)\mapsto u_L(t,s'(\gamma))$ may have an infinite norm. Thus, we need some extra constructions.  
	
	The graphs of the mappings $V$ and $S_0$ lie in $\mathcal{B}(\Gamma_+)\otimes\mathcal{F}$. Indeed,
	\[\begin{split}
		\operatorname{gr}(V)&=\Big\{(\gamma,\omega)\in\Gamma_+\times \Omega:\, \gamma=\widehat{X}(\omega),\, \mathscr{l}(\gamma)\geq \|u_L(\cdot,\omega)\|_{L^p}^p\Big\}\\&=\Big\{(\gamma,\omega)\in\Gamma_+\times \Omega:\, \gamma=\widehat{X}(\omega)\Big\}\bigcap \\&{}\hspace{70pt}\Big\{(\gamma,\omega)\in\Gamma_+\times \Omega:\, \mathscr{l}(\gamma)-\|u_L(\cdot,\omega)\|_{L^p}^p\geq 0\Big\}. \end{split}\]
	Both sets in the right-hand side of this equality are from $\mathcal{B}(\Gamma_+)\otimes\mathcal{F}$ due to the $\mathcal{F}/\mathcal{B}(\Gamma)$-measurability of mappings $\omega\mapsto \widehat{X}(\omega)$, $\omega\mapsto \|u_L(\cdot,\omega)\|_{L^p}^p$ and the fact that the function $\gamma\mapsto\mathscr{l}(\gamma)$ is Borel. The inclusion $\operatorname{gr}(S_0)\in \mathcal{B}(\Gamma_+)\otimes\mathcal{F}$ is derived in the same way.  Moreover, the very definitions of the set $\Gamma_+$ and the function $\mathscr{l}$ (see~\eqref{intro:g_bar_omega}) imply that $V(\gamma)$ is nonempty for each $\gamma\in\Gamma_+$. 
	
	Due to \cite[Theorem 6.7.3]{Bogachev}, the sets \begin{equation}\label{intro:Gamma:1}\Gamma_1\triangleq \{\gamma\in\Gamma_+:S_0(\gamma)\neq\varnothing\},\end{equation} \[\Gamma_2\triangleq \{\gamma\in\Gamma_+: V(\gamma)\cap S_0(\gamma)\neq \varnothing\}\] are Souslin, and, thus, (see \cite[Theorem 1.10.5]{Bogachev}) lie in  $\mathcal{B}_\eta(\Gamma)$. 
	By construction, $\Gamma_2\subset \Gamma_1$. Now, we define a multivalued mapping $S:\Gamma_+\rightrightarrows \Omega$ by the rule: 
	\[S(\gamma)\triangleq \left\{ \begin{array}{ll}
		V(\gamma)\cap S_0(\gamma), & \gamma\in\Gamma_2, \\
		S_0(\gamma), & \gamma\in\Gamma_1\setminus\Gamma_2, \\
		V(\gamma), & \textrm{otherwise}.
	\end{array}
	\right.\]	
By the choice of $\Gamma_+$, $\Gamma_1$, and $\Gamma_2$, $S(\gamma)$ is nonempty for every $\gamma\in\Gamma_+$. Furthermore, the graph of $S$ is equal to 
\[\begin{split}
	\operatorname{gr}(S)=\Big[\operatorname{gr}(V)\cap\operatorname{gr}(S_0)\cap(\Gamma_2\times\Omega)\Big]\bigcup \Big[&\operatorname{gr}(S_0)\cap ((\Gamma_1\setminus\Gamma_2)\times\Omega)\\&\bigcup\Big[\operatorname{gr}(V)\cap((\Gamma_+\setminus\Gamma_1)\times\Omega)\Big]\end{split}\] and, obviously, belongs to $\mathcal{B}_\eta(\Gamma_+)\otimes\mathcal{F}$. Since the probability space $\probspace$ is standard, the Aumann selection theorem~\cite[Theorem 6.9.13]{Bogachev} gives that there exists a $\mathcal{B}_\eta(\Gamma_+)/\mathcal{F}$-measurable function $s:\Gamma_+\rightarrow \Omega$ such that $s(\gamma)\in S(\gamma)$ for all $\gamma\in\Gamma_+$. 

The function $s$ is the key ingredient of  our way to improve the Lagrangian strategy $(X,u_L)$. In fact, the function $s$ can be considered as a pathwise improvement of the outcome, i.e., \begin{equation}\label{ineq:gs_barg}
	g(s(\gamma))\leq \bar{g}(\gamma)\ \ \text{whenever }\gamma\in\Gamma_+.
\end{equation} To show this, notice by~\eqref{ineq:g_bar_g}, given $\gamma\in\Gamma_+$,
\begin{itemize}
	\item either $s(\gamma)\in S_0(\gamma)$; this means that $\bar{g}(\gamma)>g(s(\gamma))$; 
	\item or $S_0(\gamma)=\varnothing$; in this case $\bar{g}(\gamma)=g(\omega)$
	for every $\omega\in\widehat{X}^{-1}(\gamma)$; in particular, $\bar{g}(\gamma)=g(s(\gamma))$.
\end{itemize} 

As we mentioned above, the norm of the strategy  $(t,\gamma)\mapsto u_L(t,s(\gamma))$ can  be infinite. We will revise it using the function $v$ defined below.

\textit{Step 3.}
By the definition of the multifunction $V$ (see \eqref{intro:l_omega}, \eqref{intro:V}), $V(\gamma)$
is nonempty for all $\gamma\in\Gamma_+$ and the graph of $V$ belongs to $\mathcal{B}(\Gamma)\otimes\mathcal{F}$. Applying once again the Aumann selection theorem (see~\cite[Theorem 6.9.13]{Bogachev}) to the restriction of $V$ on the set $\Gamma_1\setminus\Gamma_2$ that lies in $\mathcal{B}_\eta(\Gamma)$,  we construct  a selector \[\Gamma_1\setminus\Gamma_2\ni \gamma\mapsto v(\gamma)\in V(\gamma)\] that is $\mathcal{B}_\eta(\Gamma_1\setminus\Gamma_2)/\mathcal{F}$-measurable. On the set $\Gamma_+\setminus(\Gamma_1\setminus\Gamma_2)$, we put  $v(\gamma)\triangleq s(\gamma)$.

Furthermore, let us show that
\begin{equation}\label{ineq:gs_gv}
	g(s(\gamma))\leq g(v(\gamma))\qquad\text{for each } \gamma\in\Gamma_+.
\end{equation}
Indeed, 
the fact that $g(s(\gamma))\neq g(v(\gamma))$ implies that $s(\gamma)\neq v(\gamma)$. The latter, due to the construction of the function $v(\cdot)$, can take place only when $\gamma\in\Gamma_1\setminus\Gamma_2$. In other words, the curve $\gamma$ is such that $V(\gamma)\cap S_0(\gamma)=\varnothing$ and $S_0(\gamma)\neq\varnothing$. This and the construction of the selectors mean that \[v(\gamma)\notin S_0(\gamma) \text{ while }s(\gamma)\in S_0(\gamma).\] Hence,  we deduce the estimates \[g(s(\gamma))<\bar{g}(\gamma)\leq g(v(\gamma))\text{ whenever }\gamma\in \Gamma_1\setminus\Gamma_2.\] Thus, \eqref{ineq:gs_gv} holds true.

\textit{Step 4.} We use \cite[Corollary 6.5.6]{Bogachev} and construct  functions $s^\natural,v^\natural:\Gamma\rightarrow\Omega$ those are $\mathcal{B}(\Gamma)/\mathcal{F}$-measurable and satisfy the equalities: $s=s^\natural$ and $v=v^\natural$ $\eta$-a.e. on $\Gamma_+$. 

For each natural $k$, define the strategy $u_K^k$ on $\Gamma$ by the following rule: 
\begin{equation}\label{intro:u_k_strategy}
	u_K^k(\cdot,\gamma)\triangleq \left\{ \begin{array}{ll}
		u_L(\cdot,s^\natural(\gamma)), & \|u_L(\cdot,s^\natural(\gamma))\|^p_{L^p}\leq k+k\|u_L(\cdot,v^\natural(\gamma))\|^p_{L^p}, \\
		u_L(\cdot,v^\natural(\gamma)), & \textrm{otherwise}.
	\end{array}
	\right.\end{equation} 

Now let us show that $(\eta,u_K^k)$ is an admissible Kantorovich process for each natural $k$. First, $u_K^k$ is $(\mathscr{B}_T\otimes\mathcal{B}(\Gamma))/\mathcal{B}(U)$-measurable.
Furthermore, by construction, we have that $\eta=\widehat{X}\sharp \pp\in \mathcal{P}^p(\Gamma)$.
To see that $\eta$-a.e. $\gamma$ satisfies~(\ref{eq:gamma}), it suffices to recall that, for $\eta$-a.e. $\gamma\in\Gamma$, $s^\natural(\gamma)=s(\gamma)\in \widehat{X}^{-1}(\gamma)$ and $v^\natural(\gamma)=v(\gamma)\in \widehat{X}^{-1}(\gamma)$.
Finally, by the definitions of the functions $u_K^k$, $s$, $s^\natural$, $v$ and $v^\natural$, we also obtain
\begin{equation*}\label{ineq:norm_u_k}\begin{split}	
		\int_{\Gamma} \|u_K^k(&\cdot,\gamma)\|^p_{L^p}\eta(d\gamma)\\&\leq k+k
		\int_{\Gamma_+} \|u_L(\cdot,v(\gamma))\|^p_{L^p}\eta(d\gamma)\leq
		k+k
		\int_{\Gamma_+} \mathscr{l}(\gamma)\eta(d\gamma)\\
		&=
		k+k
		\int_{\Gamma_+}
		\int_{\widehat{X}^{-1}(\gamma)}\|u_L(\cdot,\omega)\|^p_{L^p} \pp_\gamma(d\omega)\eta(d\gamma)
		\\
		&=k+k\|u_L\|^p_{\up}<+\infty.
\end{split}\end{equation*}
So, each $(\eta,u_K^k)$ is an admissible Kantorovich process. In particular, each integral 
\[ \begin{split}
	a_k\triangleq\int_{\Gamma}{g}(u_K^k(\gamma))\eta(d\gamma) =
	\int_{\Gamma}\int_0^T f_0(t,e_t(\gamma),m(t),u_K^k(t,\gamma))dt\,\eta(d\gamma)
\end{split} 
\]
is finite.

The definition of $u_K^k$ (see~\eqref{intro:u_k_strategy}) gives that,  for $\eta$-a.e. curve $\gamma\in\Gamma$, the equality $u_K^k(\cdot,\gamma)=u_K^s(\cdot,\gamma)=u_L(\cdot,s^\natural(\gamma))$ is fulfilled whenever the number $k$ is large enough. Hence,  since $s^\natural$ is a modification of the function $s$, the sequence $\{u_K^k(\cdot,\gamma)\}_{k=1}^\infty$ converges to $u_L(\cdot,s(\gamma))$ for $\eta$-a.e. $\gamma\in\Gamma_+$. On the other hand, from the fact that $g(s(\gamma))\leq g(v(\gamma))$ on $\Gamma_+$ (see \eqref{ineq:gs_gv}), it follows that the sequence $\{g(u_K^k(\gamma))\}_{k=1}^\infty$ is  non-increasing and converges to $g(s(\gamma))$ for $\eta$-a.e. $\gamma\in\Gamma_+$. Therefore, the sequence  $\{a_k\}_{k=1}^\infty$
is non-increasing and converges to 
\[A\triangleq\int_{\Gamma_+}{g}(s(\gamma))\eta(d\gamma)\in\mathbb{R}\cup\{-\infty\}\]
The latter is due to the Beppo Levi's theorem. 
Furthermore,  equality \eqref{eq:constB}, the fact that $\eta(\Gamma_+)=1$ and inequality~\eqref{ineq:gs_gv} imply that
\[\int_{\Gamma_+}\bar{g}(\gamma)\eta(d\gamma)=	\int_{\Gamma}\int_{\widehat{X}^{-1}(\gamma)}g(\omega)\pp_\gamma(d\omega)\eta(d\gamma)=J_L(X,u_L)\geq A.
\]

Now, let us show that
\begin{equation}\label{ineq:integral_f_0_aim}
	\begin{split}
		J_L(X,u_L)\geq
		a_k=\int_{\Gamma}\int_0^T f_0(t,e_t(\gamma),m(t),u_K^k(t,\gamma))dt\eta(d\gamma)
	\end{split}
\end{equation}
whenever $k$ is large enough.

First, we assume that   $J_L(X,u_L)>A$. Since $A$ is the limit of the sequence $\{a_k\}_{k=1}^\infty$, we have that  $J_L(X,u_L)>a_k$ when $k$ is greater than some natural number. 
Now, we consider the case $J_L(X,u_L)=A$, i.e., we assume that
\begin{equation*}
	\begin{split}
		\int_{\Gamma_+}\bar{g}(\gamma)\eta(d\gamma)=\int_{\Gamma_+}{g}(s(\gamma))\eta(d\gamma).
	\end{split}
\end{equation*}
Since  $\bar{g}(\gamma)\geq g(s(\gamma))$, using \eqref{ineq:gs_barg}, we obtain that  the equality $\bar{g}(\gamma)={g}(s(\gamma))$ holds true for $\eta$-a.e. $\gamma\in \Gamma_+$. By the definition of the set $\Gamma_1$ (see~\eqref{intro:Gamma:1}) and the construction of the selector $v(\cdot)$, it follows that $\eta(\Gamma_1\setminus \Gamma_2)=0$ and $v(\cdot)=s(\cdot)$ $\eta$-a.e. Thus, all strategies $u_K^k(\cdot,\gamma)$ coincide with $u_L(\cdot,s(\gamma))$ for $\eta$-a.e. $\gamma\in\Gamma_+$. This gives $a_k=J_L(X,u_L)$ as well as $a_k=A$ for all natural $k$. Thus, in the case where $A=J_L(X,u_L)$, \eqref{ineq:integral_f_0_aim} holds with every $k$.


From now, we fix a number $k$ such that \eqref{ineq:integral_f_0_aim} holds true. Recall that we already proved that  $(\eta,u_K^k)$ is an admissible Kantorovich process. The fact that it improves  $(X,u_L)$ (i.e., $J_L(X,u_L)\geq J_K(\eta,u_K^k)$) follows from  \eqref{ineq:integral_f_0_aim} and  equality~\eqref{equality:sigma_gamma}.
\end{proof}

Now let us prove Theorem~\ref{th:Kantorovich_min_Lagrangian} which state that if a Lagrangian process $(X^*,u_L^*)$ realizes a local minimizer within the Kantorovich framework $(\eta^*,u_K^*)$, then it is a strong local $W_p$-minimizer.

\begin{proof}[Proof of Theorem~\ref{th:Kantorovich_min_Lagrangian}]
Let $\varepsilon>0$ be such that $J_K(\eta,u_K)\geq J_K(\eta^*,u_K^*)$ for every Kantorovich control process $(\eta,u_K)\in \operatorname{Adm}_K(m_0)$ satisfying $W_p(e_t\sharp\eta,e_t\sharp \eta_*)\leq \varepsilon$.

Consider
an admissible  Lagrangian process
$(X,{u}_L)\in \mathcal{A}_L(X_0)$ such that  $\|X-X^*\|_{\xp}\leq \varepsilon$. This implies that $W_p(X(t)\sharp\pp,X^*(t)\sharp\pp)\leq \varepsilon$ for every $t\in [0,T]$. By Lemma~\ref{lm:Kantorovich2_correspondence},  there exists a Kantorovich process $(\eta,u_K)\in \operatorname{Adm}_K(m_0)$ that improves $(X,u_L)$. In particular, $\eta=\widehat{X}\sharp \pp$. Moreover, since $(X^*,u^*_L)$ realizes $(\eta^*,u^*_K)$ we have that 
$e_t\sharp\eta^*={X}^*(t)\sharp\pp$. Therefore,  
$m_0=e_0\sharp \eta$ and $W_p(e_t\sharp \eta, e_t\sharp \eta^*)\leq\varepsilon$. By the definition of the local minimizer in the framework of the Kantorovich approach, we have $J_K(\eta,u_K) \geq J_K(\eta^*,u^*_K)$. On the other hand, since 
$(X^*,u^*_L)$ realizes $(\eta^*,u^*_K)$ and $(\eta,u_K)$  improves $(X,u_L)$, we also obtain 
\[ J_L(X,u_L)\geq J_K(\eta,u_K) \geq J_K(\eta^*,u^*_K)=J_L(X^*,u^*_L).\]	
Thus, $(X^*,u^*_L)$ is a strong local $W_p$-minimizer  at $m_0$ in the framework of the Lagrangian approach.
\end{proof}

\subsection{PMP in the Kantorovich form}

In the following $\mathscr{Y}_K^q$ stands for the set of functions $\Psi:[0,T]\times \Gamma\rightarrow\rds$ such that $\widehat{\Psi}\in L^q(\Gamma,\mathcal{B}(\Gamma),\eta^*;\Gamma^\star)$. Recall that $\widehat{\Psi}$ is the mapping assigning to $\gamma\in\Gamma$ the whole path $\psi(\cdot,\gamma)$.

\begin{theorem}\label{th:PMP_Kantorovich}
Let  $(\eta^*,u^*_K)\in \operatorname{Adm}_K(m_0)$ be  a  strong local minimizer in the framework of the Kantorovich approach. 

Then, 
there exists a function $\psi\in\mathscr{Y}_K^q$ such that the following conditions holds true:
\begin{itemize}
	\item costate equation:   for $\eta^*$-a.e. $\gamma\in\Gamma$, $\psi(\cdot,\gamma)$ solves 
	\begin{equation}\label{eq:Psi_K}
		\begin{split}
			\frac{d}{dt}\psi(t,\gamma&)\\=-\psi&(t,\gamma)\nabla_xf(t,\gamma(t),e_t\sharp \eta^*,u^*_K(t,\gamma))\\&+\nabla_xf_0(t,\gamma(t),e_t\sharp \eta^*,u^*_K(t,\gamma))\\&-
			\int_\Gamma \psi(t,\gamma')\nabla_m f(t,\gamma'(t),e_t\sharp \eta^*,\gamma(t),u^*_K(t,\gamma'))\eta^*(d\gamma')\\&+
			\int_\Gamma \nabla_m f_0(t,\gamma'(t),e_t\sharp \eta^*,\gamma(t),u^*_K(t,\gamma'))\eta^*(d\gamma');
		\end{split}
	\end{equation} 
	\item  transversality condition:
	\begin{equation}\label{condition:trans_K}
		\begin{split}
			\psi(T,\gamma)\sharp \eta^*=-\nabla_x \sigma (&\gamma(T),e_T\sharp \eta^*)\\&-\int_{\Gamma}\nabla_m 
			\sigma (\gamma'(t),e_T\sharp \eta^*,\gamma(T))\eta^*(d\gamma')
	\end{split}\end{equation}  for $\eta^*$-a.e. $\gamma\in\Gamma$;
	\item  maximization of the Hamiltonian condition:  for a.e.  $s\in [0,T]$, and $\eta^*$-a.e. $\gamma\in\Gamma$,
	\begin{equation}\label{condition:maximum_K}
		\begin{split}
			H(s,\gamma(s),\psi(s,&\gamma),e_s\sharp\eta^*,u^*_K(s,\gamma))\\&=\max_{u\in U}H(s,\gamma(s),\psi(s,\gamma),e_s\sharp\eta^*,u)
	\end{split}\end{equation}  or, equivalently, for  a.e.  $s\in [0,T]$,
	\begin{equation}\label{condition:maximum_K_int}
		\begin{split}
			\int_\Gamma	H(&s,\gamma(s),\psi(s,\gamma),e_s\sharp\eta^*,u^*_K(s,\gamma)) \eta^*(d\gamma)\\&=\max_{\nu\in L^p(\Gamma,\mathcal{B}(\Gamma),\eta^*;U)}\int_{\Gamma}H(s,\gamma(s),\psi(s,\gamma),e_s\sharp\eta^*,\nu(\gamma))\eta^*(d\gamma).
	\end{split}\end{equation}
\end{itemize}

\end{theorem}
\begin{proof}
We choose $(\Omega,\mathcal{F},\pp)\triangleq(\Gamma,\mathcal{B}(\Gamma),\eta^*)$.
By Proposition~\ref{prop:Kantorovich_correspondence},  the Lagrangian process $(X^*,u^*_K)$, where $\widehat{X^*}$ is equal to $\operatorname{id}_\Gamma$, realizes $(\eta^*,u^*_K)$.
Theorem~\ref{th:Kantorovich_min_Lagrangian} gives that the process $(X^*,u^*_K)$ is a strong local $W_p$-minimizer  at $m_0$ in the framework of the Lagrangian approach.
Applying Theorem~\ref{th:PMP_Lagrangian} for the Lagrangian control process $(X^*,u^*_K)$
and $\pp=\eta^*$, we have that now equation~\eqref{eq:Psi} is~\eqref{eq:Psi_K} while conditions~\eqref{eq:transversality},~\eqref{condition:maximum_local}  take the forms of conditions
~\eqref{condition:trans_K},~\eqref{condition:maximum_K}  respectively. The equivalence of~\eqref{condition:maximum_K} and~\eqref{condition:maximum_K_int} is a particular case of the equivalence between~\eqref{condition:maximum_integral} and~\eqref{condition:maximum_local} proved in Theorem~\ref{th:PMP_Lagrangian}.
\end{proof}

\section{Eulerian approach}\label{sect:Euler}

This section is concerned with the Eulearian formulation of the mean field type control problems. Below, we study the links between local minimizers within  the Eulerian and Lagrangian approaches. Using this, we deduce the Pontryagin maximum principle for  the Eulerian formulation of the mean field type control problem. 

In this section, we assume condition~\ref{assumption:Uclosed}--\ref{assumption:derivative_sigma} and, additionally, we impose the following convexity assumption borrowed from \cite{Cavagnari_et_al_approaches}:
\begin{enumerate}[label=(C\arabic*)]
\item\label{assumption:convU} the set $U$ is a closed convex subset of a Banach space;
\item\label{assumption:epi}   the mapping $U\ni u\mapsto f(t,x, m,u)$ is affine in $u$, i.e., for $t\in [0,T]$, $x\in\rd$, $m\in\prd$, $u_1,u_2\in U$, $\alpha\in [0,1]$,
\[f(t,x,m,\alpha u_1+(1-\alpha )u_2)=
\alpha f(t,x,m, u_1)+(1-\alpha )f(t,x,m,u_2);\]
\item\label{assumption:conv0}  
the function $f_0$ is convex in $u$, i.e., for every $t\in [0,T]$, $x\in\rd$, $m\in\prd$, $u_1,u_2\in U$,  $\alpha\in [0,1]$,
\[f_0(t,x,m,\alpha u_1+(1-\alpha )u_2)\leq
\alpha f_0(t,x,m, u_1)+(1-\alpha )f_0(t,x,m,u_2).\]
\end{enumerate}

Notice that this condition is always fulfilled if one uses  relaxed controls~\cite{Cavagnari_et_al_approaches}.

\subsection{Control processes within the Eulerian formulation}
To simplify notation, we put, given a measure-valued function $m(\cdot)\in C([0,T];\mathcal{P}(\rd))$, 
\[\eup{m(\cdot)}\triangleq L^p\big([0,T]\times\rd,\mathscr{B}_T\otimes\mathcal{B}(\rd),\lambda\otimes (m(t))_{t\in [0,T]};U\big).\] Analogously, let
\[\erp{m(\cdot)}\triangleq L^p\big([0,T]\times\rd,\mathscr{B}_T\otimes\mathcal{B}(\rd),\lambda\otimes (m(t))_{t\in [0,T]};\rd\big).\] The norms on $\eup{m(\cdot)}$ and $\erp{m(\cdot)}$ are still denoted by $\norm{\cdot}{L^p}$.

\begin{definition}
We say that a pair $(m(\cdot),u_E)$, where 
\begin{itemize}
	\item $m(\cdot)\in \operatorname{AC}^p([0,T];\mathcal{P}^p(\rd))$, 
	\item $u_E\in \eup{m(\cdot)}$, 
\end{itemize}
is an Eulerian control process if $m(\cdot)$ and the velocity field
$v_E:[0,T]\times\rd\rightarrow\rd$ defined by the rule 
\begin{equation}\label{intro:vector_field}
	v_E(t,x)\triangleq f(t,x,m(t),u_E(t,x))
\end{equation}
satisfy    the following continuity equation:
\[\partial_tm(t)+\operatorname{div} (v_E(t,x)m(t))=0\]
in the sense of distribution, i.e., for every $\varphi\in C_c^\infty((0,T)\times\rd)$, 
\[\int_0^T\int_{\rd}[\partial_t\varphi(t,x)+\nabla_x\varphi(t,x) v_E(t,x)]m(t,dx)dt=0.\]
\end{definition} 

Notice that due to assumption~\ref{assumption:function} the vector field $v_E$ defined by~\eqref{intro:vector_field} for each Eulerian control process $(m(\cdot),u_E)$ lies in $\erp{m(\cdot)}$.

The outcome of the Eulerian control process $(m(\cdot),u_E)$  is evaluated by the formula:
\[\begin{split}
J_E(m(\cdot),u_E)\triangleq \int_{\rd}\sigma(x,&m(T))m(T,dx)\\&+ \int_0^T\int_{\rd} f_0(t,x,m(t),u_E(t,x)) m(t,dx)dt.\end{split}\]

\begin{definition} Let $m_0\in\prd$. We denote the set of  Eulerian control processes $(m(\cdot),u)$ satisfying the initial condition $m(0)=m_0$ by $\operatorname{Adm}_E(m_0)$.
\end{definition}

To study the link between the Eulerian and Langrangian approaches, let us introduce the following notions.
\begin{definition}\label{def:Lagrangian_Euler1}
Let $(m(\cdot),u_E)$ be an  Eulerian control process.
A Lagrangian control process $(X,u_L)$ defined on a standard probability space $(\Omega,\mathcal{F},\pp)$ realizes $(m(\cdot),u_E)$ provided that
\begin{itemize}
	\item for every $t\in [0,T]$,
	\begin{equation}\label{eq:L-Emu}
		m(t)= X(t)\sharp \pp;\end{equation}
	\item for a.e. $t\in [0,T]$ and $\pp$-a.e. $\omega\in\Omega$,
	\begin{equation}\label{eq:L-Eu}
		u_L(t,\omega)=u_E(t,X(t,\omega)).
	\end{equation}
\end{itemize}

\end{definition}
Notice that these conditions   yield the equality
\[J_E(m(\cdot),u_E)=J_L(X,u_L).\]

The next proposition states that each Eulerian process can be realized by a Lagrangian one.

\begin{proposition}\label{prop:Euler_correspondence}
Assume that  $(m(\cdot),u_E)$  is an Eulerian control process.
Furthermore, let $(\Omega,\mathcal{F},\pp)$ be a standard probability space  such that at least one the following conditions satisfies:
\begin{enumerate}[label=($\Omega$\arabic*)]
	\item\label{Omega:cond:no_atom}  the probability  $\pp$ has no atoms,
	\item\label{Omega:cond:eta}   $\Omega=\Gamma$, $\mathcal{F}=\mathcal{B}(\Gamma)$, $\pp=\eta\in\mathcal{P}(\Gamma)$, while $\eta$-a.e. $\gamma$ solves the equation 
	\begin{equation}\label{eq:ode_gamma_u_E}
		\frac{d}{dt}\gamma(t)=f(t,\gamma(t),m(t),u_E(t,\gamma(t)))
	\end{equation} and $m(t)=e_t\sharp \eta$. 
\end{enumerate}

Then, there exists a Lagrangian process $(X,u_L)$ defined on $(\Omega,\mathcal{F},\pp)$ that realizes $(m(\cdot),u_E)$.
Furthermore, in  case ~\ref{Omega:cond:eta}, we can put $\widehat{X}=\operatorname{id}_{\Gamma}$ and $u_E(t,\gamma(t))=u_L(t,\gamma)$.
\end{proposition}
\begin{proof} First, let us construct a process $X$. If $(\Omega,\mathcal{F},\pp)$ satisfies condition~\ref{Omega:cond:eta}, we simply put $\widehat{X}(\gamma)\triangleq\gamma$. Hence,  ${X}(t,\gamma)\triangleq \gamma(t)$. The  case when $(\Omega,\mathcal{F},\pp)$ satisfies condition~\ref{Omega:cond:no_atom} is reduced to the previous one in the following way. Since $v_E$ defined by~(\ref{intro:vector_field}) lies in $\erp{m(\cdot)}$, one can apply
\cite[Theorem 8.2.1]{Ambrosio} and construct a probability measure $\eta'\in\mathcal{P}^p(\Gamma)$ such that $m(t)=e_t\sharp\eta'$ and $\eta'$-a.e. $\gamma\in \Gamma$ satisfy~\eqref{eq:ode_gamma_u_E}. Furthermore, since $\Gamma$ is Polish space, due to \cite[Theorem~3.1(i)]{Skorokhod_Representation}, there exists a $\mathcal{F}$/$\mathcal{B}$-measurable map $\widehat{X}:\Omega\rightarrow\Gamma$ such that $\widehat{X}\sharp\pp=\eta'.$ Letting $X(t,\omega)\triangleq \widehat{X}(\omega)(t)$, we construct the desired process $X$ for the case where $(\Omega,\mathcal{F},\pp)$ satisfies condition~\ref{Omega:cond:no_atom}. 

Notice that  condition~\eqref{eq:L-Emu} holds true for the process $X$ in both cases. Furthermore, we define \[u_L(t,\omega)\triangleq u_E(t,X(t,\omega)). \]  Therefore,~\eqref{eq:L-Eu} is fulfilled. 

Now, let us show that the process $(X,u_L)$ is admissible. First, we claim that, for $\pp$-a.e. $\omega\in\Omega$, $X(\cdot,\omega)$ solves
\begin{equation}\label{eq:X_v}
	\frac{d}{dt}X(t,\omega)=v_E(t,X(t,\omega)).\end{equation} Here $v_E(t,x)$ is defined by~\eqref{intro:vector_field}.
Indeed, if $\probspace$ satisfies condition~\ref{Omega:cond:eta}, this follows from the equality $X(\cdot,\gamma)\triangleq \gamma$. In  case~\ref{Omega:cond:no_atom}, we use the construction of the probability $\eta'\in\mathcal{P}^p(\Gamma)$ that is concentrated on curves satisfying~\eqref{eq:ode_gamma_u_E}. Equality~(\ref{intro:vector_field}) and the construction of $u_L$ implies that
\[v_E(t,X(t,\omega))=f(t,X(t,\omega),m(t),u_L(t,\omega)).\] This and~\eqref{eq:X_v} yield that, for $\pp$-a.e. $\omega\in\Omega$, $X(\cdot,\omega)$ is a solution of  the ODE
\[\frac{d}{dt}X(t,\omega)=f(t,X(t,\omega),m(t),u_L(t,\omega)).\]
Moreover, we have that $\pp$-a.s. 
\[\|\widehat{X}\|_\infty\leq \|X(0)\|+\int_{0}^T\|v_E(t,X(t))\|dt.\] This and  the construction of $X$ imply that
\begin{equation}\label{ineq:norm_X_v_E}
	\begin{split}
		\norm{X}{\xp}&\leq \mathcal{M}_p(m(0))+\int_0^T\int_\Omega\|v_E(t,X(t,\omega))\|\pp(d\omega)dt\\&=\mathcal{M}_p(m(0))+ \int_0^T\int_{\rd}\|v_E(t,x)\|m(t,dx)dt.\end{split}
\end{equation} Notice that 
\[\int_0^T\expect\|v_E(t,X(t))\|^pdt=\int_0^T\int_{\rd}\|v_E(t,x)\|^pm(t,dx)dt\]
Due to assumption~\ref{assumption:sublinear} and inclusion $u_E\in \eup{m(\cdot)}$, we have that $v_E\in\erp{m(\cdot)}$. Using this,~\eqref{ineq:norm_X_v_E} and the H\"older's inequality, we conclude that $X$ belongs to $\xp$.

To complete the proof, let us show that $u_L\in \up$. Indeed,
\[\begin{split}\norm{u_L}{\up}^p&=\int_0^T\expect\|u_L(t)\|^pdt=\int_0^T\expect\|u_E(t,X(t))\|^pdt\\&=
	\int_0^T\int_{\rd}\|u_E(t,x)\|^pm(t,dx)dt<+\infty.\end{split}\] The latter inequality is due to the fact that each Eulerian process $(m(\cdot),u_E)$ satisfies $u_E\in \eup{m(\cdot)}$.
	\end{proof}
	\begin{remark}
Notice that the previous proposition does not rely on the convexity assumption.
\end{remark}

\subsection{Local minimizers within the Eulerian formulation}
\begin{definition} 
An Eulerian control process $(m^*(\cdot),u^*_E)\in \operatorname{Adm}_E(m_0)$ is called a  strong local minimizer if there exists $\varepsilon>0$ such that $J_E(m(\cdot),u)\geq J_E(m^*(\cdot),u^*)$ for all admissible Eulerian processes $(m(\cdot),u)\in \operatorname{Adm}_E(m_0)$ satisfying  $W_p(m(t),m^*(t))\leq \varepsilon$
when $t\in [0,T]$.
\end{definition}

The following theorem states that each Eulerian strong minimizer corresponds to a minimizer within the Lagrangian approach.

\begin{theorem}\label{th:Euler_min_Lagrangian}
Let $(m^*(\cdot),u^*_E)$ be a strong local minimizer in the  Eulerian framework and let $(X^*,u^*_L)$ be an admissible Lagrangian process defined on some standard probability space $(\Omega,\mathcal{F},\pp)$ that realizes  $(m^*(\cdot),u^*_E)$. Then,
$(X^*,u^*_L)$ is a strong local $W_p$-minimizer at $m_0$ within the Lagrangian framework.
\end{theorem}

The proof of this statement involves the notion of improvement of a Lagrangian process by an Eulerian one and the fact that such improvement always exists.

\begin{definition}\label{def:Lagrangian_Euler2}
Let $(X,u_L)$ be a Lagrangian control process defined on a standard probability space $(\Omega,\mathcal{F},\pp)$. We say that an Eulerian control process $(m(\cdot),u_E)$  improves $(X,u_L)$ if it satisfies ~\eqref{eq:L-Emu} and 
$J_L(X,u_L)\geq J_E(m(\cdot),u_E)$.
\end{definition}

\begin{lemma}\label{lm:Euler2_correspondence}
Let   $(X,u_L)$ 
be a Lagrangian control process  defined on a standard probability space $(\Omega,\mathcal{F},\pp)$.
Then, there exists an Eulerian process $(m(\cdot),u_E)$  that improves $(X,u_L)$.
\end{lemma}
\begin{proof}

We define the  flow of probabilities $m(\cdot)\in C([0,T];\prd)$ and the velocity field $v_L\in B([0,T]\times\Omega,\mathscr{B}_T\otimes\mathcal{F};\mathbb{R}^d)$ by the following rules: for all $t\in [0,T]$ and $\omega\in\Omega$,
\[m(t)\triangleq X(t)\sharp\pp,\quad v_L(t,\omega)\triangleq f(t,X(t,\omega),m(t),u_L(t,\omega)).\]
So, ~\eqref{eq:L-Emu} holds true. Moreover, since $(X,u_L)$ is an admissible process, using assumption~\ref{assumption:function}, we conclude that
\begin{equation}
	\label{incl:vLLp}
	v_L\in L^p([0,T]\times\Omega,\mathscr{B}_T\otimes\mathcal{F},\lambda\otimes\pp;\mathbb{R}^d).
\end{equation}

Now let us define a strategy $u_E\in B([0,T]\times\rd,\mathscr{B}_T\otimes\mathcal{B}(\rd);U)$ and  a velocity field $v_E\in B([0,T]\times\rd,\mathscr{B}_T\otimes\mathcal{B}(\rd);\mathbb{R}^d)$. To this end, we use the disintegration theorem (see \cite[Theorem 5.3.1]{Ambrosio} or \cite[III-70]{Dellacherie_Meyer}) and, given $t\in [0,T]$, find a system of probability measures $\{\pp_x^t\}_{x\in\rd}$ such that, for each Borel measurable map $\phi:\Omega\rightarrow [0,+\infty]$,
\begin{equation}\label{equality:exp_euler_P}\expect\phi=\int_{\rd}\int_{\Omega_x^t}\phi(\omega)\pp_x^t(d\omega)m(t,dx),\end{equation} where each probability $\pp_x^t$ is concentrated on the set 
$\Omega_x^t\triangleq \{\omega\in\Omega:\, X(t,\omega)=x\}.$
Using this, we put
\begin{equation}\label{intro:u_E}
	u_E(t,x)\triangleq \int_{\Omega_x^t} u_L(t,\omega)\pp_x^t(d\omega),
\end{equation}
\[ v_E(t,x)\triangleq \int_{\Omega_x^t} v_L(t,\omega)\pp_x^t(d\omega).\]  Assumption~\ref{assumption:convU} gives that $u_E(t,x)\in U$.
By the Jensen's inequality, we have that 
\[\begin{split}
	\|u_E\|_{L^p}^p&=\int_0^T\int_{\rd}\bigg\|\int_{\Omega_x^t} u_L(t,\omega)\pp_x^t(d\omega)\bigg\|^p m(t,dx)dt\\&\leq \int_\Omega\int_0^T\|u_L(t,\omega)\|^pdt\pp(d\omega). \end{split}\]
This and the inclusion $u_L\in\up$ yield that the Eulerian control $u_E$ lies in~$\eup{m(\cdot)}$. Due to~\ref{assumption:epi}, we have that
\[\begin{split}f(t,x,m(t),u_E(t,x))&=
	f\Bigg(t,x,m(t),\int_{\Omega_x^t}u_L(t,\omega)\pp_x^t(d\omega)\Bigg)\\&=
	\int_{\Omega_x^t} f(t,X(t,\omega),m(t),u_L(t,\omega))\pp_x^t(d\omega)\\&=
	\int_{\Omega_x^t} v_L(t,\omega)\pp_x^t(d\omega)=v_E(t,x)\end{split}\]
for each $t\in[0,T]$ and $m(t)$-a.e. $x$.

We claim that 
$m(\cdot)$ is a distributional solution of the equation
\begin{equation}\label{eq:cont_vE}
	\partial_tm(t)+\operatorname{div} (v_E(t,x)m(t))=0
\end{equation} on $[0,T]\times\rd$.
Indeed, choose a smooth
function $\varphi\in C^\infty_c((0,T)\times\rd)$.
Since $(X,u_L)$ is an admissible Lagrangian process, we have that, for $\pp$-a.e. $\omega\in\Omega$,
\[\int_0^T[\partial_t\varphi(t,X(t,\omega))+\nabla_x\varphi(t,X(t,\omega))f(t,X(t,\omega),m(t),u_L(t,\omega))]dt=0.
\] Integrating this equality against the probability $\pp$ and using the equality $v_L(t,\omega)=f(t,X(t,\omega),m(t),u_L(t,\omega))$, we obtain
\[\int_\Omega\int_0^T[\partial_t\varphi(t,X(t,\omega))+\nabla_x\varphi(t,X(t,\omega))v_L(t,\omega)]dt\pp(d\omega)=0.\]
Notice that, for every $t\in [0,T]$,
\[\begin{split}\int_\Omega\nabla_x\varphi(t,X(t,\omega))v_L(t,\omega)\pp(d\omega)&=\int_{\rd}\nabla_x\varphi(t,x)\int_{\Omega_x^t} v_L(t,\omega) \pp^t_x(d\omega)m(t,dx)\\&=\int_{\rd}\nabla_x\varphi(t,x)v_E(t,x)m(t,dx).\end{split}\]
Analogously,
\[\int_\Omega\partial_t\varphi(t,X(t,\omega))\pp(d\omega)=\int_{\rd}\partial_t\varphi(t,x)m(t,dx).\] Therefore, $m(\cdot)$ satisfies~\eqref{eq:cont_vE} in the distributional sense. Furthermore, since $v_L\in \erp{m(\cdot)}$, we have that the mapping $t\mapsto \norm{v(t,\cdot)}{L^p}$ lies in $L^p([0,T];\mathbb{R})$. Hence, from the fact that $m(\cdot)$ is the distributional solution of~\eqref{eq:cont_vE} (see  \cite[Theorem~8.3.1]{Ambrosio}), we deduce that $m(\cdot)\in\operatorname{AC}^p([0,T];\prd)$.

Thus, we have proved that $(m(\cdot),u_E)$ is an Eulerian control process.

Finally, let us show that
$J_L(X,u_L)\geq J_E(m(\cdot),u_E)$.

By construction, we have that
\begin{equation}\label{equality:sigmas_X_m}
	\int_\Omega\sigma(X(T,\omega),m(T))\pp(d\omega)=\int_{\rd}\sigma(x,m(T))m(T,dx).
\end{equation}
Furthermore, due to \eqref{equality:exp_euler_P}, we have that
\[\begin{split}\int_{0}^T&\int_\Omega f_0(t,X(t,\omega),m(t),u_L(t,\omega))\pp(d\omega)dt\\&=
	\int_{0}^T\int_{\rd}\int_{\Omega_x^t}	f_0(t,X(t,\omega),m(t),u_L(t,\omega))\pp_x^t(d\omega)m(t,dx)dt
	\\&=\int_{0}^T\int_{\rd}\int_{\Omega_x^t}  f_0(t,x,m(t),u_L(t,\omega))\pp_x^t(d\omega)m(t,dx)dt.\end{split}
\] The definition of the control $u_E$ (see~\eqref{intro:u_E}) and assumption~\ref{assumption:conv0} give that
\[\int_{\Omega_x^t} f_0(t,x,m(t),u_L(t,\omega))\pp_x^t(d\omega)\geq f_0(t,x,m(t),u_E(t,x)).\] Therefore,
\[\begin{split}
	\int_{0}^T\int_\Omega f_0(t,X(t,&\omega),m(t),u_L(t,\omega))\pp(d\omega)dt\\&\geq \int_{0}^T\int_{\rd}f_0(t,x,m(t),u_E(t,x))m(t,dx)dt.
\end{split}\]
Combining this with~\eqref{equality:sigmas_X_m}, we arrive at the inequality  $J_L(X,u_L)\geq J_E(m(\cdot),u_E)$. 
\end{proof}

Now we are ready to prove the main result of this section that is Theorem~\ref{th:Euler_min_Lagrangian}.
\begin{proof}[Proof of Theorem~\ref{th:Euler_min_Lagrangian}]
Since $(m^*(\cdot),u^*_E)$ is a strong local minimizer within the Eulerian approach, there exists  $\varepsilon>0$ such that, for each $(m(\cdot),u_E)\in \operatorname{Adm}_E(m_0)$,
\[J_E(m^*(\cdot),u_E^*)\leq J_E(m(\cdot),u_E)\] provided that $W_p(m^*(t),m(t))\leq\varepsilon$, $t\in [0,T]$.

Let
$(X,{u}_L)\in \operatorname{Adm}_L(m_0)$ satisfy  $W_p(X(t)\sharp \pp,X^*(t)\sharp \pp)\leq \varepsilon$  
for  all $t\in [0,T]$. By Lemma~\ref{lm:Euler2_correspondence}  there exists an Eulerian process $(m(\cdot),u_E)$ that improves $(X,u_L)$.
Furthermore, from~\eqref{eq:L-Emu} it follows $m(0)=X(0)\sharp \pp= X^*(0)\sharp \pp=m_0$, $(m(\cdot),u_E)\in \operatorname{Adm}_E(m_0),$ and $W_p(m(t),m^*(t))\leq \varepsilon$ on $[0,T]$, while $J_E(m(\cdot),u_E)\geq J_E(m^*(\cdot),u^*_E)$. Since 
$(X^*,u^*_L)$ realizes $(m^*(\cdot),u^*_E)$ and $(m(\cdot),u_E)$  improves $(X,u_L)$, we  obtain 
\[J_L(X,u_L)\geq J_E(m(\cdot),u_E)\geq J_E(m^*(\cdot),u^*_E)=J_L(X^*,u^*_L).\]

Thus, $(X^*,u^*_L)$ is a strong local $W_p$-minimizer at $m_0$ in the framework of the Lagrangian approach.
\end{proof}


\subsection{Pontryagin maximum principle for the Eulerian formulation}\label{subsect:Eulerian_PMP}

The formulation of the Pontryagin maximum principle within the Eulerian approach relies on the continuity equation for probabilities defined on $\rd\times\rds$. As above, we consider the solutions in the distributional sense, i.e., if $w$ is a velocity field defined on $[0,T]\times \rd\times\rds$ with values in $\rd\times\rds$, we say that $[0,T]\mapsto \nu(t)\in\mathcal{P}(\rd\times\rds)$ solves the continuity equation 
\[\partial_t\nu+\operatorname{div}(w(t,x,\psi)\nu)=0\] if, for every $\varphi\in C_c^\infty((0,T)\times\rd\times\rds)$,
\[\begin{split}
\int_0^T\int_{\rd\times\rds}\big[\partial_t\varphi(t&,x,\psi)+\nabla_x\varphi(t,x,\psi)w_x(t,x,\psi)\\&+w_\psi(t,x,\psi)\nabla_\psi\varphi(t,x,\psi)\big]\nu(t,d(x,\psi))dt=0.\end{split}\] Here $\nabla_x\varphi$ (respectively, $\nabla_\psi\varphi)$ stands for the derivative of the function $\varphi$ w.r.t. $x$ (respectively, w.r.t. $\psi$), while $w_x(t,x,\psi)\in\rd$ and $w_\psi(t,x,\psi)\in\rds$ are components of the vector filed $w$: $w(t,x,\psi)=(w_x(t,x,\psi),w_\psi(t,x,\psi))$. Recall that $\psi$ is a row-vector. Additionally,  we regard $\nabla_\psi\varphi(t,x,\psi)$ as a column-vector. Therefore, it is convenient to write the inner product of $\nabla_\psi\varphi(t,x,\psi)$ and $w_\psi(t,x,\psi)$ in the form $w_\psi(t,x,\psi)\nabla_\psi\varphi(t,x,\psi)$.

Additionally, $p\wedge q$ stands for the minimal number between  $p$ and its conjugate exponent $q=p/(p-1)$.

\begin{theorem}\label{th:PMP_Eulerian}
Let an Eulerian control process $(m^*(\cdot),u^*_E)\in \operatorname{Adm}_E(m_0)$ be a  strong local minimizer at $m_0$.
Then, there exists a  flow of probabilities $\nu^*(\cdot)\in \operatorname{AC}^{p\wedge q}([0,T];\mathcal{P}^{p\wedge q}(\rd\times\rds))$ satisfying the following conditions:
\begin{itemize}
	\item consistency with $m^*(\cdot)$:
	\begin{equation}\label{eq:b0_E} 
		\operatorname{p}^1\sharp \nu^*(t)=m^*(t)\qquad\forall t\in[0,T];
	\end{equation}
	\item joint state and costate continuity equation: $\nu^*(\cdot)$ is a distributional solution of the  continuity equation	\begin{equation}\label{eq:contE}
		\partial_t \nu^*+ \operatorname{div} (\mathscr{j}(t,x,\psi)\nu^*)=0, 
	\end{equation}   
	where the vector field $\mathscr{j}(t,x,\psi)=(\mathscr{j}_x(t,x,\psi),\mathscr{j}_\psi(t,x,\psi))$ is given by
	\[\begin{split} \mathscr{j}_x(t,x,&\psi)\triangleq f(t,x,m^*(t),u^*_E(t,x)),\\
		\mathscr{j}_\psi(t,x,&\psi)\triangleq -\psi\nabla_xf(t,x,m^*(t),u^*_E(t,x))\\+&\nabla_xf_0(t,x,m^*(t),u^*_E(t,x))\\
		-
		&\int_{\rd\times\rds}\zeta\nabla_m f(t,y,m^*(t),x,u^*_E(t,y))\nu^*(t,d(y,\zeta))\\+
		&\int_{\rd}\nabla_m f_0(t,y,m^*(t),x,u^*_E(t,y))m^*(t,dy);
	\end{split}
	\]
	\item transversality condition:
	\begin{equation}\label{eq:b1_E} 
		\begin{split}
			\operatorname{p}^2\sharp\nu^*(T)=  \bigg[-\nabla_x &\sigma (\cdot,m^*(T))\\&-\int_{\rd}\nabla_m 
			\sigma (y,m^*(T),\cdot) m^*(T,dy)\bigg]\sharp m^*(T);	\end{split}\end{equation} 
	\item maximization condition: for almost every $s\in [0,T]$ and $\nu^*(s)$-a.e. $(x,\psi)\in \rd\times\rds$,  \begin{equation}\label{condition:maximum_E}
		H(s,x,\psi,m^*(s),u^*_E(s,x))=\max_{u\in U}H(s,x,\psi,m^*(s),u)
	\end{equation}
	or, equivalently, for a.e. $s\in [0,T]$,
	\begin{equation}\label{condition:maximum_E_int}
		\begin{split}
			\int_{\rd\times\rds}H(s,x,\psi,m^*(s),&u_E^*(s,x))\nu^*(s,d(x,\psi))\\=	\max_{\upsilon\in L^p(\rd,\mathcal{B}(\rd),m^*(s);U)}&\int_{\rd\times\rds}H(t,x,\psi,m^*(s),\upsilon(x))\nu^*(s,d(x,\psi)).
		\end{split}
	\end{equation}
\end{itemize}
\end{theorem}
\begin{proof}
We choose a probability space $\probspace$ equal to $(\Gamma,\mathcal{B}(\Gamma),\eta^*)$, where $\eta^*$ is such that~\eqref{eq:L-Eu} holds true form $m(\cdot)=m^*(\cdot)$ and $u_E=u_E^*$. The existence of such measure directly follows from~\cite[Theorem 8.2.1]{Ambrosio}.
By Proposition~\ref{prop:Euler_correspondence}, there exists  a Lagrangian process $(X^*,u^*_L)$ defined on the probability space $(\Gamma,\mathcal{B}(\Gamma),\eta^*)$ that realizes $(m^*(\cdot),u^*_E)$. Moreover, $u^*_L(t,\gamma)=u^*_E(t,\gamma(t))$ and $\widehat{X^*}=\operatorname{id}_\Gamma$.
Since $(m^*(\cdot),u^*_E)\in \operatorname{Adm}_E(m_0)$ is a strong local minimizer at $m_0$, by Theorem~\ref{th:Euler_min_Lagrangian}, the process $(X^*,u^*_L)$ is a strong local  $W_p$-minimizer at $m_0$ in the framework of the Lagrangian approach. 
Applying Corollary~\ref{cor:PMP_w_min} to this process and taking into account the equalities $\pp=\eta^*$, $m^*(t)=e_t\sharp \pp$,   $X^*(t,\gamma)=\gamma(t)$, and $u^*_L(t,\gamma)=u^*_E(t,\gamma(t))$, we find
a function $\Psi:[0,T]\times\Gamma\rightarrow\rds$ satisfying $\eta^*$-a.s. the costate equation:
\begin{equation}\label{eq:Psi_LE1}
	\begin{split}
		\frac{d}{dt}\Psi(t,\gamma&)\\=-\Psi&(t,\gamma)\nabla_xf(t,\gamma(t),m^*(t),u^*_E(t,\gamma(t)))\\&+\nabla_xf_0(t,\gamma(t),m^*(t),u^*_E(t,\gamma(t)))\\&-
		\int_\Gamma \Psi(t,\gamma')\nabla_m f(t,\gamma'(t),m^*(t),\gamma(t),u^*_E(t,\gamma'(t)))\eta^*(d\gamma')\\&+
		\int_\Gamma \nabla_m f_0(t,\gamma'(t),m^*(t),\gamma(t),u^*_E(t,\gamma'(t)))\eta^*(d\gamma'),
	\end{split}
\end{equation} 
the transversality condition:
\begin{equation} \label{condition:trans_LE}
	\Psi(T,\gamma)=-\nabla_x\sigma(\gamma(T),m^*(T))- \int_\Gamma\nabla_m\sigma(\gamma'(T),m^*(T),\gamma(T))\eta^*(d\gamma')
\end{equation}   and the maximization of the Hamiltonian condition in the integral form which states that, for a.e. $s\in [0,T]$ and $\eta^*$-a.e. $\gamma\in\Gamma$:
\begin{equation}\label{condition:maximum_LE}\begin{split}
		H(s,\gamma(s),\Psi(s,&\gamma),m^*(s),u^*_E(s,\gamma(s)))\\&=\max_{u\in U}H(s,\gamma(s),\Psi(s,\gamma),m^*(s),u).\end{split}
\end{equation} Additionally, $\widehat{\Psi}\in L^q(\Gamma,\mathcal{B}(\Gamma),\pp;\Gamma^\star)$. Here, as above,  $\widehat{\Psi}$ denotes the mapping  assigning to $\gamma\in\Gamma$ the whole path $\Psi(\cdot,\gamma)\in\Gamma^\star$. Obviously, $\widehat{X^*}=\operatorname{id}_\Gamma$ lies in $L^p(\Gamma,\mathcal{B}(\Gamma),\eta^*;\Gamma)$. Therefore,
the mapping $(\widehat{X^*},\widehat{\Psi})$ lies in $L^{p\wedge q}(\Gamma,\mathcal{B}(\Gamma),\eta^*; \Gamma\times\Gamma^\star)$. Thus,  the measure $\chi^*\in\mathcal{P}(\Gamma\times\Gamma^\star)$ defined by the rule $\chi^*=(\widehat{X^*},\widehat{\Psi})\sharp\pp$ is an element of $\mathcal{P}^{p\wedge q}(\Gamma\times\Gamma^\star)$.  In particular, we have that $\operatorname{p}^1\sharp\chi^*=\pp=\eta^*$ and
$(\operatorname{p}^1\circ e_t)\sharp\chi^*=m^*(t)$.

We claim that the measure $\chi^*$ is supported on $ \operatorname{AC}^{p\wedge q}([0,T];\rd\times\rds)$. Indeed, the probability $\operatorname{p}^1\sharp\chi^*=\eta^*$ is concentrated on  $\operatorname{AC}^{p}([0,T];\rd)\subset \operatorname{AC}^{p\wedge q}([0,T];\rd)$. Due to the fact that $\Psi(\cdot)$ satisfies the costate equation~\eqref{eq:Psi_LE1} $\eta^*$-a.s., while $u_E\in \eup{m^*(\cdot)}$, using assumption~\ref{assumption:derivative_f_0}, we conclude that $\Psi(\cdot,\gamma)\in \operatorname{AC}^{q}([0,T];\rds)\subset \operatorname{AC}^{p\wedge q}([0,T];\rds)$ for $\eta^*$-a.e. $\gamma\in\Gamma$.

Now let us consider the continuity equation  
\begin{equation}\label{eq:continuity_nu}\partial_t \nu(t)+\operatorname{div}(w(t,x,\psi)\nu(t))=0\end{equation}
with the  vector field $w(t,x,\psi)=(w_x(t,x,\psi),w_\psi(t,x,\psi))$, where
\[\begin{split}
	w_x(t,x,\psi)=f(t,&x,m^*(t),u^*_E(t,x)),\\
	w_\psi(t,x,\psi)=-\psi&\nabla_xf(t,x,m^*(t),u^*_E(t,x))+\nabla_xf_0(t,x,m^*(t),u^*_E(t,x))\\&-
	\int_\Gamma \beta'(t)\nabla_m f(t,\gamma'(t),m^*(t),x,u^*_E(t,\gamma'(t)))\chi^*(d(\gamma',\beta'))\\&+
	\int_\Gamma \nabla_m f_0(t,\gamma'(t),m^*(t),x,u^*_E(t,\gamma'(t)))\chi^*(d(\gamma',\beta')).
\end{split}
\]
We claim that the flow of probabilities $\nu^*(\cdot)$ defined by the rule $\nu^*(t)\triangleq e_t\sharp\chi^*$ is a distributional solution of \eqref{eq:continuity_nu}.  To show this one can use direct computations, the  equality $\chi^*=(\widehat{X^*},\widehat{\Psi})\sharp\eta_*=(\operatorname{id}_\Gamma,\widehat{\Psi})\sharp\eta_*$ and the facts that $(X^*,u^*_L)$ is a Lagrangian process defined on $(\Gamma,\mathcal{B}(\Gamma),\eta^*)$, while $\widehat{\Psi}$  assigns  to each $\gamma\in\Gamma$ the curve $\Psi(\cdot,\gamma)$ satisfying \eqref{eq:Psi_LE1},~\eqref{condition:trans_LE}. 

Furthermore, we have that  $\nu^*(\cdot)$ satisfies~\eqref{eq:b0_E}. Indeed, we already proved that $\nu^*(\cdot)$ is a solution of~\eqref{eq:continuity_nu}. Simultaneously,   $w_x(t,x,\psi)=\mathscr{j}_x(t,x,\psi)$, while $w_\psi(t,x,\psi)=\mathscr{j}_\psi(t,x,\psi)$. The latter is due to equalities $\nu(t)=e_t\sharp\eta^*$,  $\operatorname{p}^1\sharp\nu^*(t)=m^*(t)=e_t\sharp\eta^*$.

The very definition of the flow of probabilities $\nu^*(\cdot)$ and~\eqref{condition:trans_LE} give the fact that $\nu^*$ satisfies the transversality condition in the Eulerian form~\eqref{eq:b1_E}. 

Now, let us prove maximization condition in the local form~\eqref{condition:maximum_E}. First, recall that ~\eqref{condition:maximum_LE} holds true for a.e. $s\in[0,T]$ and $\eta^*$-a.e. $\gamma\in\Gamma$. We fix $s\in [0,T]$ that satisfies this property. Now we consider a Borel set  $\Xi^0(s)\subset \rd\times\rds$ such that~\eqref{condition:maximum_E} is violated whenever $(x,\psi)\in\Xi^0(s)$. Since $\nu^*(s)=e_s\sharp (\operatorname{id}_\Gamma,\widehat{\Psi})\sharp \eta^*$, we have that 
\[\nu^*(s,\Xi^0(s))=\eta^*(\Xi^1(s)),\] where $\Xi^1(s)\in\mathcal{B}(\Gamma)$ contains all curves $\gamma\in\Gamma$ such that $(\gamma(s),\Psi(s,\gamma))\in\Xi^0$. By construction,~\eqref{condition:maximum_LE} is violated for $\gamma\in \Xi^1(s)$. Thus, since we chose $s$ such that~\eqref{condition:maximum_LE} holds true $\eta^*$-a.s.,  $\nu^*(s,\Xi^0(s))=0$.

Finally, the equivalence between~\eqref{condition:maximum_E} and~\eqref{condition:maximum_E_int} is proved in the same way as the equivalence of~\eqref{condition:maximum_integral} and~\eqref{condition:maximum_local} in Theorem~\ref{th:PMP_Lagrangian}.
\end{proof}	

\begin{remark}\label{rmk:Euler_Hamiltinian_j}
Let us  express the vector field $(\mathscr{j}_x,\mathscr{j}_\psi)$ as a Hamiltonian flow.   Indeed, we put
\[\mathscr{H}(t,\nu,u)\triangleq \int_{\rd\times\rds}H(t,x,\psi,\operatorname{p}^1\sharp\nu,u)\nu(d(x,\psi)).\] Using Proposition~\ref{prop:computation:int_phi}, we arrive at the equality
\begin{equation}\label{equality:Ham_int_totla_derivative}
	\begin{split}
		\nabla_\nu\mathscr{H}(t,\nu,x&,\psi,u)\\=(\nabla_x &H(t,x,\psi,\operatorname{p}^1\sharp\nu,u),\nabla_\psi H(t,x,\psi,\operatorname{p}^1\sharp\nu,u))\\&+\Bigl(\int_{\rd\times\rds} \nabla_m H(t,y,\zeta,\operatorname{p}^1\sharp\nu,x,u)\nu(d(y,\zeta)),0\Bigr),\end{split}
\end{equation} where
\begin{equation}\label{equality:Ham_int_x_psi_derivative}
	\begin{split}
		\nabla_x H(t,x,\psi,\operatorname{p}^1\sharp\nu,u)&= \psi\nabla_xf(t,x,\operatorname{p}^1\sharp\nu,u)-\nabla_xf_0(t,x,\operatorname{p}^1\sharp\nu,u),\\ 
		\nabla_\psi H(t,x,\psi,\operatorname{p}^1\sharp\nu,u)&=f(t,x,\operatorname{p}^1\sharp\nu,u),
	\end{split}
\end{equation}  while \begin{equation}\label{equality:Ham_int_nu_derivative}
	\begin{split}
		\nabla_m H(t,y,\zeta,&\operatorname{p}^1\sharp\nu,x,u)\\&=\zeta\nabla_m f(t,y,\operatorname{p}^1\sharp\nu,x,u)-\nabla_m f_0(t,y,\operatorname{p}^1\sharp\nu,x,u)\in \rds.\end{split}\end{equation} Furthermore, let $\mathbb{J}:\rds\times\rd\rightarrow\rd\times\rds$ be the linear function  defined by the rule: 
\[\mathbb{J}(\zeta,y)\triangleq (y,-\zeta).\] One can regard $\mathbb{J}$ as a unit symplectic matrix. Comparing the formulae for $\mathscr{j}_x$ and $\mathscr{j}_\psi$ with~\eqref{equality:Ham_int_totla_derivative}--\eqref{equality:Ham_int_nu_derivative}, we conclude that
\[\mathscr{j}(t,x,\psi)=\mathbb{J}\nabla_\nu\mathscr{H}(t,x,\psi,\nu(t),u)\text{ for }u=u^*_E(t,x).\]
\end{remark}


\section{Mean field type linear-quadratic regulator}\label{sect:MFT_LQR}
In this section, we come back to the Lagrangian approach and consider the model problem of linear-quadratic regulator with the additional terms describing the variance of the distribution of agents. We put $p=2$. Moreover, we fix a standard probability space $(\Omega,\mathcal{F},\pp)$ and an initial assignment $X_0\in L^2(\Omega,\mathcal{F},\pp;\rd)$. The dynamics of each agent is given by the equation
\begin{equation}\label{example:system}
\frac{d}{dt}X(t,\omega)=A(t)X(t,\omega)+B(t)u(t,\omega),
\end{equation} while the payoff functional is equal to
\begin{equation}\label{example:payoff}
\begin{split}
	\frac{1}{2}\expect\bigg(\int_0^T [X^T(t)&Q_x(t)X(t)+u^T(t)R(t)u(t)]dt+X^T(T)K_x(t)X(T)\bigg)\\+\frac{1}{2}\int_0^T&\expect[(X(t)-\expect X(t))^TQ_m(t) (X(t)-\expect X(t))]dt\\&+\frac{1}{2}\expect[(X(T)-\expect X(T))^TK_m ( X(T)-\expect X(T))].
\end{split}
\end{equation} Here $X(t,\omega)\in\rd$, $U=\mathbb{R}^{d'}$, $A(t),Q_x(t)$, $Q_m(t)$, $K_x$, $K_m$ are $(d\times d)$-matrices, $d'$ is natural, $B(t)$ is a $(d\times d')$-matrix, $R(t)\in\mathbb{R}^{d'\times d'}$. Additionally, the matrices $Q_x(t)$, $Q_m(t)$, $K_x$, $K_m$ and $R(t)$ are symmetric, while $R(t)>0$. Finally, we assume that the matrix-valued functions $A(\cdot)$, $B(\cdot)$, $R(\cdot)$, $Q_x(\cdot)$ and $Q_m(\cdot)$ are continuous on $[0,T]$. 

A problem of such class can be regarded as a deterministic variant of  social optimization problem studied in \cite{Social_choice}.

Notice that the first  term in~\eqref{example:payoff} refers to the individual controls of the agents. The second term is the integrated variance of the random variable $D(t)X(t)$, where $Q_m(t)=D^T(t)D(t)$. Finally, the third term is equal to the variance of the random variable $\Theta X(T)$ with $\Theta^T\Theta=K_m$. The last two terms evaluates the cooperative behavior of the agents. Since, for each symmetric matrix $\mathcal{Q}$ and every random variable $\xi$ with values in $\rd$, $\expect[(\xi-\expect\xi)^T\mathcal{Q}(\xi-\expect\xi)]=\expect(\xi^T\mathcal{Q}\xi)-(\expect\xi^T)Q(\expect\xi)$, we may set
\[\begin{split}
f_0(t,x,m,u)\triangleq \frac{1}{2}\Bigg[x^TQ_x&(t)x+u^TR(t)u+x^TQ_m(t)x\\&- \bigg(\int_{\rd}y^Tm(dy)\bigg)Q_m(t)\bigg(\int_{\rd}ym(dy)\bigg) \Bigg],\end{split}\]
\[\sigma(x,m)\triangleq \frac{1}{2}\Bigg[x^TK_xx+x^TK_mx-\bigg(\int_{\rd}y^Tm(dy)\bigg)K_m\bigg(\int_{\rd}ym(dy)\bigg)\Bigg].\] 

Below, to simplify notation, given a random variable $\xi$, we denote
\[\overline{\xi}\triangleq \mathbb{E}\xi.\]

\begin{theorem}\label{th:LQR} If $(X^*,u^*)$ is a Pontryagin local $L^2$-minimizer at some initial assignment $X_0$ for problem~(\ref{example:system}),~(\ref{example:payoff}) with an initial assignment $X_0$.  Then, 
\begin{equation}\label{example:u_optimal_representation}
	u^*(t,\omega)= -R^{-1}(t)B^T(t)\big[P_1(t)(X^*(t,\omega)-\overline{X}^*(t))+P_2(t)\overline{X}^*(t)\big],
\end{equation} where $P_1(\cdot)$ is the matrix-valued function solving the Ricatti differential equation 
\begin{equation}\label{example:Ricatti_1}
	\begin{split}
		\frac{d}{dt}P_1(t)=-P_1(t)A&(t)-A^T(t)P_1(t)\\&+P_1(t)B(t)R^{-1}(t)B(t)P_1(t)-(Q_x(t)+Q_m(t))\end{split}
\end{equation} with the boundary condition
\begin{equation}\label{example:boundary_ricatti_1}
	P_1(T)=K_x+K_m,
\end{equation}
while $P_2(\cdot)$ satisfies the Ricatti differential equation
\begin{equation}\label{example:Ricatti_2}
	\begin{split}
		\frac{d}{dt}P_2(t)=-P_2(t)A&(t)-A^T(t)P_2(t)\\&+P_2(t)B(t)R^{-1}(t)B(t)P_2(t)-Q_x(t)\end{split}\end{equation} and  the boundary condition
\begin{equation}\label{example:boundary_ricatti_2}
	P_2(T)=K_x.
\end{equation}
\end{theorem}
\begin{proof}
We will use Theorem~\ref{th:PMP_Lagrangian} to determine the optimal control. Notice that  the Hamiltonian $H(t,x,\psi,m,u)$ for problem~(\ref{example:system}),~(\ref{example:payoff}) is equal to
\[\begin{split}
	H(t,x,\psi,m,u)\triangleq \psi A(t)x+&\psi B(t)u\\-\frac{1}{2}\bigg[ &x^TQ_x(t)x+u^TR(t)u+x^TQ_m(t)x\\&-\bigg(\int_{\rd}y^Tm(dy)\bigg)Q_m(t)\bigg(\int_{\rd}ym(dy)\bigg)\bigg].
\end{split}
\] 

Below, to use the matrix notation, we work with the vector $\Upsilon(t)=\Psi^T(t)$. 

The maximization condition implies that 
\[R(t)u^*(t,\omega)=B^T(t)\Upsilon(t,\omega).\] Since $R(t)>0$, we have that
\begin{equation}\label{example:u_optimal_Upsilon}
	u^*(t,\omega)=R^{-1}(t)B^T(t)\Upsilon(t,\omega).
\end{equation}
Plugging this control to equation~(\ref{example:system}), we obtain 
\begin{equation}\label{example:system_optimal_psi}
	\frac{d}{dt}X^*(t,\omega)=A(t)X^*(t,\omega)+B(t)R^{-1}(t)B^T(t)\Upsilon(t,\omega).
\end{equation} Recall that $X^*$ satisfies the initial condition
\begin{equation}\label{example:X_initial}
	X^*(0,\omega)=X_0(\omega).
\end{equation}

Using the formula for the derivative of the function depending on mean (see Proposition~\ref{prop:computation:phi_average}), we conclude that the transposed costate variable $\Upsilon(\cdot,\omega)$ satisfies the equation
\begin{equation}\label{example:costate_eq_first}
	\frac{d}{dt}\Upsilon(t,\omega)= (Q_x(t)+Q_m(t))X^*(t,\omega)-A^T(t)\Upsilon(t,\omega)-Q_m(t)\overline{X}^*(t)
\end{equation} and the boundary condition
\begin{equation}\label{example:costate_boundary_first}
	\Upsilon(T,\omega)=-(K_x+K_m)X^*(T,\omega)+K_m\overline{X}^*(T).
\end{equation}
For each $\omega\in\Omega$, system~\eqref{example:system_optimal_psi},~\eqref{example:costate_eq_first} is a nonhomogeneous system of linear equations. To analyze it, we   take expectation in equations~(\ref{example:system_optimal_psi}),~(\ref{example:costate_eq_first}) and in boundary conditions~(\ref{example:X_initial}),~(\ref{example:costate_boundary_first}). This leads to the following system on $\overline{X}^*$ and $\overline{\Upsilon}$:
\begin{equation}\label{example:system_averaged_X}
	\frac{d}{dt}\overline{X}^*(t)=A(t)\overline{X}^*(t)+B(t)R^{-1}(t)B(t)\overline{\Upsilon}(t), \end{equation}
\begin{equation}\label{example:system_averaged_Psi}\frac{d}{dt}\overline{\Upsilon}(t)= Q_x(t)\overline{X}^*(t)-A^T(t)\overline{\Upsilon}(t) 
\end{equation} equipped with the boundary conditions
\begin{equation}\label{example:boundary_averaged}
	\overline{X}^*(0)=\overline{X_0}, \ \ \overline{\Upsilon}(T)=-K_x\overline{X}^*(T).
\end{equation}
Subtracting~\eqref{example:system_averaged_X} from~\eqref{example:system_optimal_psi} and~\eqref{example:system_averaged_Psi} from~\eqref{example:costate_eq_first}, we obtain that the differences $X^*(t,\omega)-\overline{X}^*(t)$ and $\Upsilon(t,\omega)-\overline{\Upsilon}(t)$ satisfies the following system of ODEs
\begin{equation}\label{example:diff_system}
	\begin{split}
		&\frac{d}{dt}[X^*(t,\omega)-\overline{X}^*(t)]=A(t)[X^*(t,\omega)-\overline{X}^*(t)]\\&{}\hspace{150pt}+B(t)R^{-1}(t)B^T(t)[\Upsilon(t,\omega)-\overline{\Upsilon}(t)],\\
		&\frac{d}{dt}[\Upsilon(t,\omega)-\overline{\Upsilon}(t)]= (Q_x(t)+Q_m(t))[X^*(t,\omega)-\overline{X}^*(t)]\\&{}\hspace{205pt}-A^T(t)[\Upsilon(t,\omega)-\overline{\Upsilon}(t)].
	\end{split} 
\end{equation} Furthermore,
\begin{equation}\label{example:diff_conditions}
	\begin{split}
		&X^*(0,\omega)-\overline{X}^*(0)=X_0(\omega)-\overline{X}_0,\\ &\Upsilon(T,\omega)-\overline{\Upsilon}(T)=-(K_x+K_m)[X^*(T,\omega)-\overline{X}^*(T)].\end{split}
\end{equation}
From the  theory of a finite dimensional LQ regulator (see \cite[\S 6.1.1, 6.1.2]{concise_control}), we have that 
\begin{equation}\label{example:diff_representation}
	\Upsilon(t,\omega)-\overline{\Upsilon}(t)=-P_1(t)[X^*(t,\omega)-\overline{X}^*(t)],
\end{equation} where $P_1(\cdot)$ satisfies~(\ref{example:Ricatti_1}) and~(\ref{example:boundary_ricatti_1}).

Additionally,~\eqref{example:system_averaged_X}--\eqref{example:boundary_averaged} and results of  \cite[\S 6.1.1, 6.1.2]{concise_control}  yield that 
\[\overline{\Upsilon}(t)=-P_2(t)\overline{X}^*(t),\] where $P_2(\cdot)$ satisfies~(\ref{example:Ricatti_2}),~\eqref{example:boundary_ricatti_2}. 
Plugging $\overline{\Upsilon}(t)$ into~\eqref{example:diff_representation}, we conclude that
\[\Upsilon(t,\omega)=-P_1(t)[X^*(t,\omega)-\overline{X}^*(t)]-P_2(t)\overline{X}^*(t).\] This and~\eqref{example:u_optimal_Upsilon} imply~\eqref{example:u_optimal_representation}.
\end{proof}

\begin{remark}
The strategy described by  synthesis~\eqref{example:u_optimal_representation} looks as  a solution of this mean field optimal control problem. To check this directly, one should  analyze the Bellman equation in the Wasserstein space. This problem lies beyond the scope of the paper.\end{remark}

\begin{acknowledgement}
	We would like to thank  anonymous referees for their valuable and helpful comments.
\end{acknowledgement}

\appendix

\section{Some properties of intrinsic derivative}\label{appendix:derivative}
\renewcommand{\thetheorem}{A.\arabic{theorem}}
\begin{proposition}\label{prop:lipschitz} 
	Assume that $\Phi:\prd\rightarrow \mathbb{R}$ has a intrinsic derivative that is continuous and bounded by a constant $\widehat{C}$. Then $\Phi$ is Lipschitz continuous with the constant equal to $\widehat{C}$.
\end{proposition}
\begin{proof}
	Let $m,m'\in\prd$, and let $\pi_0\in \Pi(m',m)$ be an optimal plan between $m$ and $m'$ for the cost function equal to $\|x-y\|^p$. The existence of the optimal plan is due to \cite[Theorem 4.1]{Villani}.
	We have that 
	\[\begin{split}
		\Phi(m')-\Phi(m)=\int_0^1\int_{\rd\times\rd}\Bigg[\frac{\delta\Phi}{\delta m}((&1-s)m+sm',y')\\-&\frac{\delta\Phi}{\delta m}((1-s)m+sm',y)\Bigg]\pi_0(d(y',y))ds.\end{split}\] Furthermore, notice that
	\[\begin{split}\bigg[\frac{\delta\Phi}{\delta m}((1-s)m+sm',y')&-\frac{\delta\Phi}{\delta m}((1-s)m+sm',y)\bigg]\\=&\nabla_m\Phi((1-s)m+sm',y+\tilde{r}(y'-y))\cdot (y'-y),\end{split}\] where $\tilde{r}\in (0,1)$ depends on $y'$ and $y$. By assumption  $\nabla_m\Phi$  is bounded by some constant $\widehat{C}$. Since $\pi$ is an optimal plan between $m'$ and $m$, using the Jensen's inequality when $p>1$, we obtain
	\[\begin{split}
		\Phi(m')-\Phi(m)&\leq \int_{\rd\times\rd}\widehat{C}\|y'-y\|\pi_0(d(y',y))\\&\leq \widehat{C} \left[\int_{\rd\times\rd}\|y'-y\|^p\pi_0(d(y',y))\right]^{1/p}=\widehat{C}W_p(m',m).\end{split}\] 
	Interchanging the measures $m$ and $m'$, we derive the Lipschitz continuity of the function $\Phi$.
\end{proof}

Now, let us compute the intrinsic derivative for a function depending on the first moment of a probability measure. 
\begin{proposition}\label{prop:computation:phi_average}
	Assume that 
	\begin{itemize} 
		\item the function $\phi_1:\rd\rightarrow\mathbb{R}$ is differentiable, 
		\item $\Phi_1(m)\triangleq \phi_1\bigg(\int_{\rd}zm(dz)\bigg).$\end{itemize}  Then, 
	\[
	\nabla_m\Phi_1(m,y)=\nabla_x\phi_1\bigg(\int_{\rd}zm(dz)\bigg).
	\] 
\end{proposition}
\begin{proof}
	Indeed, we have that, for every probabilities $m,m'\in\mathcal{P}^p(\rd)$,
	\[\begin{split}\lim_{s\downarrow 0}&\frac{\Phi(m+s(m'-m))-\Phi(m)}{s}\\&=\lim_{s\downarrow 0}\frac{1}{s}\bigg[\phi_1\bigg(\int_{\rd}z((1-s)m+sm')(dz)\bigg)-\phi_1\bigg(\int_{\rd}zm(dz)\bigg)\bigg]\\
		&=\nabla_x\phi_1\bigg(\int_{\rd}zm(dz)\bigg)\cdot \int_{\rd} y[ m'(dy)-m(dy)].
	\end{split}\] Thus, \[\frac{\delta\Phi_1}{\delta m}(m,y)=\nabla_x\phi_1\bigg(\int_{\rd}zm(dz)\bigg) y.\] This yields the statement of the proposition. \end{proof}

Furthermore, we compute the intrinsic derivative of the mean of the function depending also on a probability.
\begin{proposition}\label{prop:computation:int_phi}
	Let 
	\begin{itemize}
		\item $\phi_2:\rd\times\mathcal{P}^p(\rd)\rightarrow\mathbb{R}$  be continuous and differentiable w.r.t. $x$ and $m$;
		\item $|\phi_2(x,m)|\leq \overline{C}_1(1+\|x\|^p+\mathcal{M}_p^p(m))$;
		\item $|\nabla_m\phi_2(x,m,y)|\leq \overline{C}_1(1+\|x\|^p+\mathcal{M}_p^p(m)+\|y\|^p)$;
		\item $\Phi_2(m)\triangleq \int_{\rd}\phi_2(x,m)m(dx)$.
	\end{itemize} Here $\overline{C}_1$ is a positive constant.
	Then,
	\[
	\nabla_m\Phi_2(m,y)=\nabla_x\phi_2(y,m)+\int_{\rd}\nabla_m\phi_2(x,m,y)m(dx).\] 
\end{proposition}
\begin{proof} Since, as we mentioned above, the flat derivative is defined up to an additive constant, we within this proof assume that, for each $x\in\rd$, $m\in \mathcal{P}^p(\rd)$,
	\begin{equation}\label{A:convention}
		\frac{\delta\phi_2}{\delta m}(x,m,0)=0.
	\end{equation}
	
	Now, let us compute $\frac{\delta\Phi_2}{\delta m}$. We have that, given a probability $m'$,
	\begin{equation}\label{A:Phi:m_diff_1}
		\begin{split}
			\lim_{s\downarrow 0}&\frac{\Phi_2(m+s(m'-m))-\Phi_2(m)}{s}\\&{}\hspace{20pt}=\lim_{s\downarrow 0}\int_{\rd}\phi_2(x,m+s(m'-m))[m'(dx)-m(dx)]\\&{}\hspace{40pt}+
			\lim_{s\downarrow 0}\frac{1}{s}\bigg[\int_{\rd}[\phi_2(x,m+s(m'-m))-\phi_2(x,m)]m(dx)\bigg].
		\end{split} 
	\end{equation} Furthermore, 
	for each $s$ and $x$, we have that 
	\[\begin{split}
		\phi_2(x,m+&s(m'-m))-\phi_2(x,m)\\&=s\int_0^1\int_{\rd}\frac{\delta\phi_2}{\delta m}(x,m+rs(m'-m),y)[m'(dy)-m(dy)]dr\\&=
		s\int_0^1\int_{\rd}\frac{\delta\phi_2}{\delta m}(x,m+rs(m'-m),0)[m'(dy)-m(dy)]dr.\end{split}\] Plugging this into the right-hand-side of~\eqref{A:Phi:m_diff_1}, we arrive at the equality:
	\begin{equation}\label{A:Phi_2_start}
		\begin{split}
			\lim_{s\downarrow 0}&\frac{\Phi_2(m+s(m'-m))-\Phi_2(m)}{s}\\&{}\hspace{-3pt}
			=\lim_{s\downarrow 0}\int_{\rd}\phi_2(x,m+s(m'-m))[m'(dx)-m(dx)]\\&{}\hspace{5pt}+
			\lim_{s\downarrow 0}\int_0^1\int_{\rd}\int_{\rd}\frac{\delta\phi_2}{\delta m}(x,m+rs(m'-m),y)m(dx)[m'(dy)-m(dy)]dr.
		\end{split} 
	\end{equation}
	Notice that   the function  $\rd\times (0,1]\ni(x,s)\mapsto \phi_2(x,m+s(m'-m))$ is continuous and is bounded by the function $\overline{C}_1'(1+\|x\|^p+\|y\|^p)$, where $\overline{C}_1'$ is a positive constant dependent on $\overline{C}_1$ and $p$.		
	Furthermore, the function $\rd\times \rd\times (0,1]\times [0,1]\ni(x,y,s,r)\mapsto\frac{\delta\phi_2}{\delta m}(x,m+rs(m'-m),y)$ is also continuous. Let us show that it grows not faster than $\|x\|^p+\|y\|^p$. Indeed, due to convention~\eqref{A:convention},
	\[
	\frac{\delta\phi_2}{\delta m}(x,m+rs(m'-m),y)=\int_0^1\nabla_m\phi_2(x,m+rs(m'-m),\alpha y)d\alpha.
	\] The growth condition on $\nabla_m\phi_2$ implies that the function $\rd\times \rd\times (0,1]\times [0,1]\ni(x,y,s,r)\mapsto\frac{\delta\phi_2}{\delta m}(x,m+rs(m'-m),y)$ is bounded by the function $\overline{C}_1''(1+\|x\|^p+\|y\|^p)$. Due to the dominated convergence theorem, one can pass  to the limit in the right-hand side of \eqref{A:Phi_2_start}  as $s\rightarrow 0$. Thus,
	\[
	\begin{split}
		\lim_{s\downarrow 0}\frac{\Phi_2(m+s(m'-m))-\Phi_2(m)}{s}&{}\\=
		\int_{\rd}\phi_2(y,m)[m'(dy)-&m(dy)]\\+\int_{\rd}\int_{\rd}\frac{\delta\phi_2}{\delta m}&(x,m,y)m(dx)[m'(dy)-m(dy)].
	\end{split}\]
	Therefore,
	\[\frac{\delta\Phi_2}{\delta m}(m,y)=\phi_2(y,m)+\int_{\rd}\frac{\delta\phi_2}{\delta m}(x,m,y)m(dx).\] Taking the derivative w.r.t. $y$, we obtain the statement of the proposition.
\end{proof} 

We complete this section with the formula of derivative of function depending on push-forward measure. 
\begin{proposition}\label{prop:derivative_push_forward} Let
	\begin{itemize}
		\item $\probspace$ be a probability space;
		\item $p>1$, $q$ be conjugate to $p$;
		\item $\Phi:\mathcal{P}^p(\rd)\rightarrow\mathbb{R}$ be such that  $\nabla_m\Phi$ is continuous and, for each $m\in\mathcal{P}^p(\rd)$, $y\in\rd$,
		\[\|\nabla_m\Phi(m,y)\|^q\leq \overline{C}_2(1+\mathcal{M}_p^p(m)+\|y\|^p),\] where $\overline{C}_2$ is a positive constant.
	\end{itemize} Then, there exists the Gateaux derivative of the mapping $L^p(\Omega,\mathcal{F},\pp;\rd)\ni X\mapsto \Phi(X\sharp\pp)$ and
	\[\nabla_X\Phi(X\sharp \pp)=\nabla_m\Phi(X\sharp\pp,X).\]
	
\end{proposition}
This proposition is a slight extension of \cite[Proposition 2.2.3]{Master_cdll} where only the case of bounded derivative is considered. Certainly, the proof follows the method used in \cite{Master_cdll}.
\begin{proof}
	Let $X, Y\in L^p(\Omega,\mathcal{F},\pp;\rd)$. We shall prove that
	\begin{equation}\label{A:limit:Phi_difference}
		\lim_{h\downarrow 0}\frac{\Phi((X+hY)\sharp \pp)-\Phi(X\sharp \pp)}{h}=\expect [\nabla_m\Phi(X\sharp\pp,X)\cdot Y]. 
	\end{equation} 
	
	For $h>0$, $s\in [0,1]$, we denote $m\triangleq X\sharp\pp$, $m^h\triangleq (X+hY)\sharp \pp$, $\mu^{s,h}\triangleq m+s(m^h-m)$. Due to~\eqref{equality:Phi_integral}, we have that 
	\[\begin{split}
		\Phi(&(X+hY)\sharp \pp)-\Phi(X\sharp \pp)\\&=\int_0^1\int_{\rd} \frac{\delta\Phi}{\delta m}(\mu^{s,h},y)[m^h(dy)-m(dy)]ds\\&=
		\int_0^1\mathbb{E} \bigg[\frac{\delta\Phi}{\delta m}(\mu^{s,h},X+hY)-\frac{\delta\Phi}{\delta m}(\mu^{s,h},X)\bigg]ds.
	\end{split}\] 
	Notice that for each $x,y\in\rd$,
	\[\frac{\delta\Phi}{\delta m}(\mu^{s,h},x+hy)-\frac{\delta\Phi}{\delta m}(\mu^{s,h},x)=h\int_0^1\nabla_m\Phi(\mu^{s,h},x+rhy)ydr.\]
	Therefore,
	\[\frac{\Phi((X+hY)\sharp \pp)-\Phi(X\sharp \pp)}{h}=\int_0^1\int_0^1\mathbb{E}\Big[ \nabla_m\Phi(\mu^{s,h},X+rhY)Y\Big]dsdr.\]	
	This equality implies that 
	\begin{equation}\label{A:ineq:lim_diff_Phi}
		\begin{split}
			\bigg|&\frac{\Phi((X+hY)\sharp \pp)-\Phi(X\sharp \pp)}{h}-\expect [\nabla_m\Phi(X\sharp\pp,X)\cdot Y]\bigg|\\&\leq 
			\int_0^1\int_0^1\expect \Big[ \|\nabla_m\Phi(\mu^{s,h},X+hr Y)- \nabla_m\Phi(m,X)\| \|Y\|\Big]dsdr \\ &\leq
			\Bigg[\int_0^1\int_0^1\expect \|\nabla_m\Phi(\mu^{s,h},X+hr Y)- \nabla_m\Phi(m,X)\|^q dsdr\Bigg]^{1/q}\|Y\|_{L^p}.
		\end{split}
	\end{equation} Now assume that $h\in (0,1]$. Notice that $\pp$-a.s.
	\[\begin{split}
		\|\nabla_m&\Phi(\mu^{s,h},X+hr Y)- \nabla_m\Phi(m,X)\|^q\\&\leq 
		2^{q-1}\|\nabla_m\Phi(\mu^{s,h},X+hr Y)\|^q+2^{q-1}\|\nabla_m\Phi(m,X)\|^q.
	\end{split}\]
	Using the assumption of the proposition, we evaluate the right-hand side of this inequality and obtain	that the following inequality holds $\pp$-a.s.:
	\begin{equation}\label{A:ineq:nabla_Phi_diff}\\ \begin{split}
			\|\nabla_m&\Phi(\mu^{s,h},X+hr Y)- \nabla_m\Phi(m,X)\|^q\\&\leq
			2^{q-1}\overline{C}_2(2+\mathcal{M}_p^p(\mu^{s,h})+\mathcal{M}_p^p(m)+\|X+hr Y\|^p+\|X\|^p).\end{split}
	\end{equation} Furthermore, we have that $\mathcal{M}_p^p(m)=\mathcal{M}_p^p(X\sharp\pp)=\norm{X}{L^p}^p$, while, since $\mu^{s,h}=(X\sharp\pp)+s(((X+hY)\sharp \pp)-(X\sharp\pp))$,
	$\mathcal{M}_p^p(\mu^{s,h})\leq \mathcal{M}_p^p(X\sharp \pp)+\mathcal{M}_p^p((X+hY)\sharp \pp)\leq (1+2^{p-1})\norm{X}{L^p}^p+2^{p-1}\norm{Y}{L^p}$. 
	Plugging this estimates into right-hand side of~\eqref{A:ineq:nabla_Phi_diff}, we obtain that $\pp$-a.s. 
	\[\begin{split}
		\|\nabla_m\Phi(\mu^{s,h},&X+hr Y)- \nabla_m\Phi(m,X)\|^q \\ &\leq 
		\overline{C}_2'(1+\norm{X}{L^p}^p+\norm{Y}{L^p}^p+\|X\|^p+\|Y\|^p),
	\end{split} 
	\] where $\overline{C}_2'$ is a constant dependent only on $\overline{C}_2$ and $p$. Thus, the random variable
	\[\|\nabla_m\Phi(\mu^{s,h},X+hr Y)- \nabla_m\Phi(m,X)\|^q\] is bounded by a summable random variable. Furthermore, the assumption that $\nabla_m\Phi$ is continuous yields that,  for $\lambda\otimes\lambda\otimes\pp$-a.e. $(r,s,\omega)\in [0,1]\times[0,1]\times\Omega$,
	\[\nabla_m\Phi(\mu^{s,h},X(\omega)+hr Y(\omega))\rightarrow \nabla_m\Phi(m,X(\omega))\text{ as }h\rightarrow 0.\] Therefore, due to the dominated convergence theorem, we obtain that 
	\[
	\int_0^1\int_0^1\expect\|\nabla_m\Phi(\mu^{s,h},X+hr Y)- \nabla_m\Phi(m,X)\|^qdsdr \] tends to 0 while $h\rightarrow 0$. This means that the right-hand side of (\ref{A:ineq:lim_diff_Phi}) tends to 0 and yields~\eqref{A:limit:Phi_difference}.
\end{proof}

\section{Properties of the perturbed dynamics}\label{appendix:proofs_from_Sect_5}	

\subsection{Dense set of the spike variations}\label{appendix:Lebesgue}
\renewcommand{\thetheorem}{B.\arabic{theorem}}
In this section, we work the Lagrangian approach introduced in Section~\ref{sect:lagrangian}. 
\begin{proof}[Proof of Proposition \ref{prop:N_T}]
	First, we claim that the space	 $L^p(\Omega,\mathcal{F},\pp;U)$ is separable. Indeed, \cite[Proposition 1.2.29]{Analysis_in_Banach_Spaces}  states that  $L^p(\Omega,\mathcal{F},\pp;U)$ is separable whenever $\mathcal{F}$ is countably generated. Taking into account the assumption that $\probspace$ is standard and, thus, due to \cite[Example 6.5.2]{Bogachev}, $\mathcal{F}$ is countably generated, we obtain the desired separability of  $L^p(\Omega,\mathcal{F},\pp;U)$. In the following, let $\mathcal{N}$ be a dense countable subset of $L^p(\Omega,\mathcal{F},\pp;U)$.

	Furthermore, we consider the control process $(X^*,u^*)$. By \cite[Theorem II.2.9]{Diestel_Uhl}, there exists a set $\mathcal{T}^*\subset [0,T]$ such that
	\begin{itemize}
		\item $\lambda([0,T]\setminus \mathcal{T}^*)=0$;
		\item for each $s\in\mathcal{T}^*$, equalities \eqref{equality:f_star}, \eqref{equality:f_0_star} hold true.
	\end{itemize} Additionally, without loss of generality, one can assume that, for each $s\in\mathcal{T}^*$,
	\[\|u^*(s)\|_{L^p}<+\infty.\]
	Analogously, for each $\nu\in \mathcal{N}$, we consider the pair $(X^*,\nu)$. By \cite[Theorem II.2.9]{Diestel_Uhl}, we conclude that there exists a set $\mathcal{T}_\nu\subset [0,T]$ satisfying the following conditions
	\begin{itemize}
		\item $\lambda([0,T]\setminus \mathcal{T}_\nu)=0$;
		\item for each $s\in\mathcal{T}_\nu$, equalities \eqref{equality:f_nu}, \eqref{equality:f_0_nu} hold true.
	\end{itemize} Letting 
	\[\mathcal{T}\triangleq \mathcal{T}_*\bigcap\Bigg[\bigcap_{\nu\in\mathcal{N}}\mathcal{T}_\nu\Bigg],\] we complete the proof.
\end{proof}

\subsection{Prior estimates of the perturbed dynamics}\label{appendix:sub:prior}
This section is concerned with the proof of Proposition \ref{prop:Z_h_bounds}. It uses the Lipschitz continuity of the function $f$ w.r.t. $x$ and $m$. Recall that assumption~\ref{assumption:derivative_f} and Proposition~\ref{prop:lipschitz} yield that, for every $t\in [0,T]$, $x_1,x_2\in \rd$, $m_1,m_2\in\prd$, $u\in U$,
\begin{equation}\label{ineq:Lip_f}
	\|f(t,x_1,m_1,u)-f(t,x_2,m_2,u)\|\leq C_x\|x_1-x_2\|+C_m W_p(m_1,m_2).
\end{equation} Here $C_x$ and $C_m$ are upper bounds for the derivatives of the function $f$ w.r.t. $x$ and $m$ respectively.

\begin{proof}[Proof of Proposition \ref{prop:Z_h_bounds}]
	First, notice that
	\[\int_0^T\expect\|u_\nu^h(t)\|^pdt\leq  \int_0^T\expect\|u^*(t)\|^pdt+\|\nu\|_{L^p}^ph.\] Thus,
	\begin{equation}\label{ineq:u_bound}
		\norm{u_\nu^h}{\up}\leq C_u\triangleq \norm{u_*}{\up}+T^{1/p}\|\nu\|_{L^p}.
	\end{equation}

	Due to assumption~\ref{assumption:sublinear} and equality $Z^h_\nu(s)=X^*(s)$, we have the following estimate  $\pp$-a.s.:
	\[\begin{split}
		\|Z^h_\nu(t&)-X^*(s)\|\\&\leq \int_s^t\|f(\tau,Z^h(\tau),Z^h(\tau)\sharp \pp,u^h_\nu(\tau))\|d\tau \\ &\leq
		C_\infty(t-s)+C_\infty\int_s^t\big(\|Z^h(\tau)\|+\|Z^h(\tau)\|_{L^p}+\|u^h_\nu(\tau)\|\big)d\tau.
	\end{split}
	\] Hence, using the triangle  inequality, we conclude that, if $t\in [s,T]$,
	\begin{equation}\label{ineq:Z_h_integral}
		\begin{split}
			\|Z^h_\nu&(t)-X^*(s)\|_{L^p}\\&\leq C_\infty(t-s)+2C_\infty\int_s^t\|Z^h_\nu(\tau)\|_{L^p}d\tau+C_\infty \int_s^t\|u^h_\nu(\tau)\|_{L^p}d\tau.\end{split}
	\end{equation}
	Thanks to~(\ref{ineq:u_bound}), we obtain
	\begin{equation}\label{ineq:Z_h_nu_star_t}
		\|Z^h_\nu(t)-X^*(s)\|_{L^p}\leq C_\infty(t-s)+C_\infty C_u+2C_\infty\int_s^t\|Z^h_\nu(\tau)\|_{L^p}d\tau.
	\end{equation}
	Since $X^*(s)\in L^p(\Omega,\mathcal{F},\pp;\rd)$, estimate~(\ref{ineq:Z_h_nu_star_t}) together with the Gronwall's inequality give the first statement of the proposition.

	Estimating the right-hand side of~(\ref{ineq:Z_h_integral}) according to the first statement of the proposition, we obtain that, for $t\in [s,s+h]$
	\[\|Z^h_\nu(t)-X^*(s)\|_{L^p}\leq C_\infty(1+2C_0+\|\nu\|_{L^p})(t-s).\] This proves the second statement of the proposition. 
	
	To prove the third statement of the proposition, we use the assumption that $s\in\mathcal{T}$. In particular, for such $s$ equality \eqref{equality:f_star} holds true. Thus, one can find $\bar{h}$ such that, for any $h\in (0,\bar{h}]$,
	\begin{equation}\label{ineq:bar_h}
		\begin{split}
			\expect\Bigl\|\frac{1}{h}\int_s^{s+h} f(\tau,&X^*(\tau),X^*(\tau)\sharp \pp,u^*(\tau))d\tau\\&-f(s,X^*(s),X^*(s)\sharp \pp,u^*(s))\Bigr\|^p\leq 1.\end{split}\end{equation} Additionally, the inclusion $s\in \mathcal{T}$ assures, in particular (see \eqref{ineq:u_bound}), that $\norm{u^*(s)}{L^p}<+\infty$. This,~(\ref{ineq:bar_h}) and  assumption~\ref{assumption:sublinear} give that
	\[\begin{split}
		\norm{X^*(s+h)-&X^*(s)}{L^p}= \bigg\|\int_s^{s+h}f(\tau,X^*(\tau),X^*(\tau)\sharp \pp,u^*(\tau))d\tau\bigg\|_{L^p}\\&\leq
		\bigg\|\int_s^{s+h}f(\tau,X^*(\tau),X^*(\tau) \sharp \pp,u^*(\tau))d\tau\\&\hspace{120pt}-hf(s,X^*(s),X^*(s)\sharp \pp,u^*(s))\bigg\|_{L^p}\\&\hspace{50pt}+h\|f(s,X^*(s),X^*(s)\sharp \pp,u^*(s))\|_{L^p}\\ &\leq h+C_\infty(1+2\|X^*(s)\|_{L^p}+\|u^*(s)\|_{L^p})h=C_1'h.
	\end{split}\]

	This and the second statement of the proposition imply that, if $h\in (0,\bar{h}]$,
	\[\|Z^h_\nu(s+h)-X^*(s+h)\|_{L^p}\leq (C_1+C_1') h.\] Furthermore,   $u^h_\nu(t)=u^*(t)$ when $t\in [s+h,T]$. This together with Lipschitz continuity of the function $f$ (see (\ref{ineq:Lip_f})) yield the following inequality, for $t\in [s+h,T]$:
	\[\begin{split}
		\|Z^h&{}_\nu(t)-X^*(t)\|_{L^p}\\&\leq
		\|Z^h_\nu(s+h)-X^*(s+h)\|_{L^p}+\bigg\|\int_{s+h}^t [f(\tau,Z^h_\nu(\tau),Z^h_\nu(\tau)\sharp \pp,u^*_\nu(\tau))\\&\hspace{176pt}- f(\tau,X^*(\tau),X^*(\tau)\sharp \pp,u^*_\nu(\tau))]d\tau\bigg\|
		\\&\leq (C_1+C_1')h+(C_x+C_m)\int_{s+h}^t\|Z^h_\nu(\tau)-X^*(\tau)\|_{L^p}d\tau.\end{split}\] Using the Gronwall's inequality, we obtain the third statement of the proposition.
	
\end{proof}


\subsection{Derivative of the perturbed process}\label{appendix:sub:process}

This section is concerned with the proof of Proposition~\ref{prop:derivative_y_h}. We will use the following property.

If $X\in  \xp$, then its restriction on $[s,r]\times\Omega$  lies in $L^p([s,r]\times\Omega,\mathcal{B}_\lambda([s,r])\otimes\mathcal{F},\lambda\otimes \pp;\rd)$. Additionally, the mapping $[s,r]\ni t\mapsto X(t,\omega)$ lies in $C([s,r];\rd)$ for $\pp$-a.e. $\omega\in\Omega$. We denote by $\|X\|_{L^p,s,r}$  the $L^p$-norm of the restriction of $X$ on $[s,r]$ regarded as an element of $L^p([s,r]\times\Omega,\mathcal{B}_\lambda([s,r])\otimes\mathcal{F},\lambda\otimes \pp;\rd)$, i.e.,
\[\|X\|_{L^p,s,r}\triangleq \bigg[\int_s^r\expect\|X(t)\|^p\bigg]^{1/p}dt.\]
The following relation between $\norm{X}{\xp}$ and $\norm{X}{L^p,s,r}$ is fulfilled:
\begin{equation*}\label{ineq:norm_L^p_X_p}
	\|X\|_{L^p,s,r}\leq (r-s)^{1/p}\norm{X}{\xp}. 
\end{equation*}

Furthermore, assume that a measurable function $X:[s,r]\times\Omega\rightarrow\rd$ is such that \begin{itemize}
	\item  for each  $t\in [0,T]$, $X(t)\in L^p(\Omega,\mathcal{F},\pp;\rd)$,
	\item the function $t\mapsto \norm{X(t)}{L^p}$ is  bounded.
\end{itemize} Then, $X\in L^p([s,r]\times\Omega,\mathcal{B}_\lambda([s,r])\otimes\mathcal{F},\lambda\otimes \pp;\rd)$ for every $s,r\in [0,T]$, $s<r$, and
\begin{equation}\label{ineq:norm_L^p_sup}
	\|X\|_{L^p,s,r}\leq (r-s)^{1/p}\sup_{t\in [s,r]}\|X(t)\|_{L^p}.
\end{equation}

\begin{proof}[Proof of Proposition \ref{prop:derivative_y_h}]
	For simplicity, put
	\begin{equation}\label{intro:F_full}
		F(t,\omega)\triangleq \\ f^*_x(t,\omega)Y_\nu(t,\omega)+(f_m^*\diamond Y_\nu) (t,\omega).
	\end{equation} 
	
	Notice that, due to Proposition~\ref{prop:Y_nu}, $\|Y_\nu(t)\|_{L^p}$ is uniformly  bounded. Furthermore, the functions $f^*_x$ and $f_m^*$ are bounded (see  assumption \ref{assumption:derivative_f}). Therefore,
	\begin{equation}\label{ineq:F_bound}\|F(t)\|_{L^p}\leq C_4,\end{equation} where $C_4$ is a constant (certainly, dependent on $(X^*,u^*)$).
	
	Choose $t\in (s,T]$.
	Let $N$ be such that, for every $n>N$, we have that $t>s+h_n$. 
	
	Since $(X^*,u^*)$ is an admissible Lagrangian process, $Z^{h_n}_\nu$ satisfies \eqref{eq:ODE_omega}, $Y_\nu$ is a solution of \eqref{eq:derivative:Y}, we have that 
	\begin{equation*}
		\begin{split}
			\frac{1}{h_n}\|Z_\nu^{h_n}(s+h_n)-X&{}^*(s+h_n)-h_n Y_\nu(s+h_n))\|_{L^p}\\ \leq 
			\frac{1}{h_n}\bigg\|\int_s^{s+h_n} \big[f(\tau,&Z_\nu^h(\tau),Z_\nu^h(\tau)\sharp \pp,\nu)- f(\tau,X^*(\tau),X^*(\tau)\sharp \pp,u^*(\tau))\big]d\tau
			\\&{}\hspace{105pt}-h_n \Delta^s_\nu f^*-h_n\int_s^{s+h_n}F(\tau)d\tau\bigg\|_{L^p}.
		\end{split}
	\end{equation*} Since $f$ is Lipschitz continuous w.r.t. $x$ and $m$ with constants $C_x$ and $C_m$ respectively (see~\eqref{ineq:Lip_f}), $\|F(t)\|_{L^p}$ is bounded (see~(\ref{ineq:F_bound})),  Proposition~\ref{prop:Z_h_bounds}, we derive the following
	\begin{equation}\label{ineq:Z_h_s_s_plus_h}
		\frac{1}{h_n}\|Z_\nu^h(s+h_n)-X^*(s+h_n)-h_n Y_\nu(s+h_n)\|_{L^p} \leq a_n^{(1)},\end{equation}
	where
	\begin{equation}\label{intro:a_1_def}\begin{split}
			a_n^{(1)}\triangleq	\| f(s,X^*&(s),X^*(s)\sharp \pp,\nu)\\&\hspace{40pt}-f(s,X^*(s),X^*(s)\sharp \pp,u^*(s))- \Delta^s_\nu f^*\|_{L^p}\\ +\frac{1}{h_n} &{}\bigg\|\int_s^{s+h_n}\big[f(s,X^*(s),X^*(s)\sharp \pp,u^*(s))\\\\&\hspace{70pt}-f(\tau,X^*(\tau),X^*(\tau)\sharp \pp,u^*(\tau))\big]d\tau\bigg\|_{L^p}
			\\ +\frac{1}{h_n} &{}\bigg\|\int_s^{s+h_n}\big[f(s,X^*(s),X^*(s)\sharp \pp,\nu)\\\\&\hspace{70pt}-f(\tau,X^*(\tau),X^*(\tau)\sharp \pp,\nu)\big]d\tau\bigg\|_{L^p}
			\\+(&C_xC_1+C_mC_1+C_xC_2+C_mC_2+C_4)h_n.
		\end{split}
	\end{equation} 
	The first term in the previous formula is equal to $0$ (see~(\ref{intro:variations_f})). The second and the third terms in the right-hand side of \eqref{intro:a_1_def}
	tend to 0 due to the fact that $\nu\in\mathcal{N}$, while $s\in\mathcal{T}$ (see equalities \eqref{equality:f_star}, \eqref{equality:f_nu} in Proposition~\ref{prop:N_T}). Simultaneously, $(C_xC_1+C_mC_1+C_xC_2+C_mC_2+C_4)h_n\rightarrow 0$ as $n\rightarrow\infty.$
	Thus, the sequence $\{a_n^{(1)}\}_{n=1}^\infty$ converge to 0 when $n\rightarrow \infty$.
	
	Furthermore, 
	\begin{equation}\label{ineq:Z_h_Z_star_Lp:triangle}
		\begin{split}
			\frac{1}{h_n}\|&Z_\nu^{h_n}(t)-X^*(t)-h_n Y_\nu(t))\|_{L^p}\\&\leq \frac{1}{h_n}\|Z_\nu^{h_n}(s+h_n)-(X^*(s+h_n)+h_n Y_\nu(s+h_n))\|_{L^p}\\&{}\hspace{30pt}+
			\frac{1}{h_n}\bigg\|\int_{s+h_n}^t[f(\tau,Z_\nu^{h_n}(\tau),Z_\nu^{h_n}(\tau)\sharp \pp,\nu)\\&{}\hspace{77pt}- f(\tau,X^*(\tau),X^*(\tau)\sharp \pp,u^*(\tau))-h_nF(\tau)]d\tau\bigg\|_{L^p}\\ &\leq a_n^{(1)}+
			\frac{1}{h_n}\bigg\|\int_{s+h_n}^t[f(\tau,Z_\nu^{h_n}(\tau),Z_\nu^{h_n}(\tau)\sharp \pp,\nu)\\&{}\hspace{77pt}- f(\tau,X^*(\tau),X^*(\tau)\sharp \pp,u^*(\tau))-h_nF(\tau)]d\tau\bigg\|_{L^p}.
		\end{split}
	\end{equation}

	Now we evaluate the second term in the right-hand side of~(\ref{ineq:Z_h_Z_star_Lp:triangle}). First, notice that, by  \cite[Theorem II.2.4, (i)]{Diestel_Uhl},
	\begin{equation}\label{ineq:Z_h_Z_star_Lp:second:1}\begin{split}
			\bigg\|\int_{s+h_n}^t &[f(\tau,Z_\nu^{h_n}(\tau),Z_\nu^{h_n}(\tau)\sharp \pp,u^*(\tau))\\&{}\hspace{77pt}- f(\tau,X^*(\tau),X^*(\tau)\sharp \pp,u^*(\tau))-h_nF(\tau)]d\tau\bigg\|_{L^p}\\ \leq 
			&\int_{s+h_n}^t\|f(\tau,Z_\nu^{h_n}(\tau),Z_\nu^{h_n}(\tau)\sharp \pp,u^*(\tau))\\&{}\hspace{77pt}- f(\tau,X^*(\tau),X^*(\tau)\sharp \pp,u^*(\tau))-h_nF(\tau)\|_{L^p}d\tau.
		\end{split}
	\end{equation}

	Simultaneously, the following equality holds true $\pp$-a.s.:
	\[\begin{split}f(\tau,Z_\nu^{h_n}(\tau),Z_\nu^{h_n}(\tau)&\sharp \pp,u^*(\tau))- f(\tau,X^*(\tau),X^*(\tau)\sharp \pp,u^*(\tau))\\=
		(f(\tau,Z_\nu^{h_n}(\tau)&,Z_\nu^{h_n}(\tau)\sharp \pp,u^*(\tau))- f(\tau,X^*(\tau),Z_\nu^h(\tau)\sharp \pp,u^*(\tau)))\\&+ (f(\tau,X^*(\tau),Z_\nu^{h_n}(\tau)\sharp \pp,u^*(\tau))\\&- f(\tau,X^*(\tau),X^*(\tau)\sharp \pp,u^*(\tau))).\end{split}\]
	Therefore, by the triangle inequality and definition of the function $F$ (see~(\ref{intro:F_full})), we have that
	\begin{equation}\label{ineq:Z_h_Z_star_Lp:second:Holder}
		\begin{split}
			\|f(\tau,&Z_\nu^{h_n}(\tau),Z_\nu^{h_n}(\tau)\sharp \pp,u^*(\tau))\\&{}\hspace{30pt}- f(\tau,X^*(\tau),X^*(\tau)\sharp \pp,u^*(\tau))-{h_n}F(\tau)\|_{L^p}\\ \leq 
			&\|f(\tau,Z_\nu^{h_n}(\tau),Z_\nu^{h_n}(\tau)\sharp \pp,u^*(\tau))\\&{}\hspace{30pt}- f(\tau,X^*(\tau),Z_\nu^{h_n}(\tau)\sharp \pp,u^*(\tau))-{h_n} f_x^*(\tau)Y_\nu(\tau)\|_{L^p} \\&+
			\|f(\tau,X^*(\tau),Z_\nu^{h_n}(\tau)\sharp \pp,u^*(\tau))\\&{}\hspace{30pt}- f(\tau,X^*(\tau),X^*(\tau)\sharp \pp,u^*(\tau))-{h_n}(f_m^*\diamond Y_\nu) (\tau,\omega)\|_{L^p}.
		\end{split}
	\end{equation}
	Since $f$ is continuously differentiable w.r.t. $x$, we conclude that
	\[
	\begin{split}f(\tau,&Z_\nu^{h_n}(\tau),Z_\nu^{h_n}(\tau)\sharp \pp,u^*(\tau))- f(\tau,X^*(\tau),Z_\nu^{h_n}(\tau)\sharp \pp,u^*(\tau))\\&=
		\int_0^1\nabla_xf(\tau,y^n_1(r,\tau),Z^{h_n}_\nu(\tau)\sharp \pp,u^*(\tau))(Z_\nu^h(\tau)-X^*(\tau))dr.
	\end{split} 
	\] Above we put
	\[y^n_1(r,\tau,\omega)\triangleq X^*(\tau,\omega)+r(Z_\nu^{h_n}(\tau,\omega)-X^*(\tau,\omega))\] and omit the dependence on $\omega$. Notice that, due to Corollary \ref{corollary:almost everywhere}, $y^n_1(r,\tau,\omega)$ tends to $X^*(\tau,\omega)$ for $\lambda\otimes\lambda\otimes\pp$-a.e. $r$, $\tau$ and $\omega$ as $n\rightarrow\infty$.
	
	Taking into account the definition of $f_x^*$ (see~(\ref{intro:f_x})), we have
	\[
	\begin{split}
		\|f(\tau,Z_\nu^{h_n}&(\tau),Z_\nu^{h_n}(\tau)\sharp \pp,u^*(\tau))\\&{}\hspace{30pt}- f(\tau,X^*(\tau),Z_\nu^{h_n}(\tau)\sharp \pp,u^*(\tau))-{h_n} f_x^*(\tau)Y_\nu(\tau)\|\\\leq 
		\int_0^r&\|\varpi_{x}^n(r,\tau)\|dr \cdot\|Z_\nu^{h_n}(\tau)-X^*(\tau)\|+
		\|f_x^*(\tau)(Z_\nu^{h_n}(\tau)-X^*(\tau))\|.
	\end{split}
	\]
	Here we denote
	\begin{equation*}\label{intro:varpi_x}
		\begin{split}
			\varpi_{x}^n(r,\tau,\omega)\triangleq \nabla_xf(\tau,y^n_1(\tau,r,\omega),&Z^{h_n}_\nu(\tau)\sharp \pp,u^*(\tau,\omega))\\&-\nabla_x f(\tau,X^*(\tau,\omega),X^*(\tau)\sharp \pp,u^*(\tau,\omega))
		\end{split}
	\end{equation*} and omit the dependence on $\omega$. 
	Therefore, using the H\"older inequality, one can estimate the integral over $[s+h,t]$ of the first term in the right-hand side of~(\ref{ineq:Z_h_Z_star_Lp:second:Holder})
	\begin{equation}\label{ineq:Z_h_Z_star_Lp:second:Holder:first:1}
		\begin{split}
			\int_{s+h_n}^t\|f(\tau,&Z_\nu^{h_n}(\tau),Z_\nu^{h_n}(\tau)\sharp \pp,u^*(\tau))\\&- f(\tau,X^*(\tau),Z_\nu^{h_n}(\tau)\sharp \pp,u^*(\tau))-{h_n} f_x^*(\tau)Y_\nu\|_{L^p}dt \\ \leq
			\bigg[\int_s^T&\int_\Omega\int_0^1\|\varpi^n_{x}(r,\tau,\omega)\|^q dr\pp(d\omega)d\tau\bigg]^{1/q} \|Z_\nu^{h_n}-X^*\|_{L^p,s,T}\\&+
			\int_{s+h_n}^t\|f_x^*(\tau)\|_{L^q}\|Z_\nu^{h_n}(\tau)-X^*(\tau)-h_nY_\nu(\tau))\|_{L^p}d\tau.
		\end{split}
	\end{equation}
	
	Recall  (see Proposition~\ref{prop:Z_h_bounds}) that, for every $\tau\in [s,T]$, $\|Z_\nu^{h_n}(\tau)-X^*(\tau)\|_{L^p}\leq C_2h$. Therefore, by~\eqref{ineq:norm_L^p_sup},
	\[\|Z_\nu^{h_n}-X^*\|_{L^p,s,T}\leq T^{1/p}C_2h_n.\]
	Additionally, thanks  to assumption~\ref{assumption:derivative_f},
	\begin{equation}\label{ineq:norm_f_x_star}
		\|f_x^*(\tau)\|\leq C_x.
	\end{equation}
	Since 
	\begin{itemize}
		\item $f_x$ is continuous, 
		\item $y^n_1(r,\tau,\omega)$ tends to $X^*(\tau,\omega)$ for $\lambda\otimes\lambda\otimes\pp$-a.e. $r$, $\tau$ and $\omega$ as $n\rightarrow\infty$,
		\item $\|Z^{h_n}_\nu(\tau)-X^*(\tau)\|_{L^p}\rightarrow 0$ as $n\rightarrow\infty$ uniformly w.r.t. time variable,
	\end{itemize} the sequence $\{\varpi^n_{x}(r,\tau,\omega)\}$ converges to zero $\lambda\otimes\lambda\otimes \pp$-a.e. when $n\rightarrow \infty$. Moreover, due to assumption~\ref{assumption:derivative_f},
	\[\|\varpi_{x}^n(r,\tau,\omega)\|\leq 2C_x.\] Therefore, by the dominated convergence theorem, 
	the quantity
	\begin{equation}\label{intro:a_2}a_n^{(2)}\triangleq C_2T^{1/p}\bigg[\int_s^T\int_\Omega\int_0^1\|\varpi_{x}^n(r,\tau,\omega)\|^q dr\pp(d\omega)d\tau\bigg]^{1/q}\end{equation} tends to zero as $n\rightarrow \infty$. Plugging this estimate and~\eqref{ineq:norm_f_x_star} into~(\ref{ineq:Z_h_Z_star_Lp:second:Holder:first:1}), we conclude that
	\begin{equation}\label{ineq:Z_h_Z_star_Lp:second:Holder:first:final}
		\begin{split}
			\int_{s+h_n}^t\|f(\tau,&Z_\nu^{h_n}(\tau),Z_\nu^{h_n}(\tau)\sharp \pp,u^*(\tau))\\&{}\hspace{30pt}- f(\tau,X^*(\tau),Z_\nu^{h_n}(\tau)\sharp \pp,u^*(\tau))-{h_n} f_x^*(\tau)Y_\nu\|_{L^p}dt \\ &\leq
			a_n^{(2)}h+C_x       \int_{s+h_n}^t\|Z_\nu^{h_n}(\tau)-X^*(\tau)-h_nY_\nu(\tau))\|_{L^p}d\tau.
		\end{split}
	\end{equation}
	
	Now let us evaluate the  integral over $[s+h,t]$ of the second term in the right-hand side of~(\ref{ineq:Z_h_Z_star_Lp:second:Holder}).
	Since the function $m\mapsto f(\tau,x,m,u)$ is continuously differentiable w.r.t. $m$, letting, for the given $\tau \in [s,T]$ and $\theta\in [0,1]$, 
	\[m^n(\theta,\tau)\triangleq \theta Z_\nu^{h_n}(\tau)\sharp \pp+(1-\theta)X^*(\tau)\sharp \pp,\] we obtain
	\[\begin{split}
		f(&\tau,X^*(\tau),Z_\nu^{h_n}(\tau)\sharp \pp,u^*(\tau))- f(\tau,X^*(\tau),X^*(\tau)\sharp \pp,u^*(\tau))\\&=
		\int_0^1\int_{\rd} \frac{\delta f}{\delta m}(t,X^*(\tau),m^n(\theta,\tau),y,u)((Z_\nu^{h_n}(\tau)\sharp \pp)(dy)\\&{}\hspace{160pt}-(X^*(\tau)\sharp \pp)(dy))d\theta\\&=
		\int_0^1\int_{\Omega} \bigg[\frac{\delta f}{\delta m}(\tau,X^*(\tau),m^n(\theta,\tau),Z_\nu^{h_n}(\tau,\omega'),u^*(\tau))\\&\hspace{60pt}-\frac{\delta f}{\delta m}(\tau,X^*(\tau),m^n(\theta,\tau),X^*(\tau,\omega'),u^*(\tau)\bigg]\pp(d\omega')d\theta\\ &=
		\int_0^1\int_{\Omega}\int_0^1 \nabla_mf(t,X^*(\tau),m^n(\theta,\tau),y^n_2(r,\tau,\omega'),u^*(\tau))\\&\hspace{130pt}(Z_\nu^h(\tau,\omega')-X^*(\tau,\omega'))dr\pp(d\omega')d\theta.
	\end{split}
	\] Here we put
	\begin{equation}\label{intro:y_n_2}y^n_2(r,\tau,\omega')\triangleq X^*(\tau,\omega')+r(Z_\nu^{h_n}(\tau,\omega')- X^*(\tau,\omega')\end{equation}
	Denote
	\begin{equation}\label{inro:varpi_m_n}
		\begin{split}
			\varpi_{m}^n(\theta,r,&\tau,\omega,\omega')\\ \triangleq\nabla_m&f(\tau,X^*(\tau,\omega),\theta Z_\nu^{h_n}(\tau)\sharp \pp+(1-\theta)X^*(\tau)\sharp \pp ,y^n_2(r,\tau,\omega'),u^*(\tau,\omega)) \\&-\nabla_mf(\tau,X^*(\tau,\omega),X^*(\tau)\sharp \pp,X^*(\tau,\omega'),u^*(\tau,\omega)).
		\end{split}
	\end{equation} Therefore, using the definitions of $\nabla_m f$ and $Y_\nu$ (see~(\ref{intro:langle})), we have that
	\[
	\begin{split}
		\|f(\tau,X^*(\tau),&Z_\nu^{h_n}(\tau)\sharp \pp,u^*(\tau))\\&\hspace{30pt}- f(\tau,X^*(\tau),X^*(\tau)\sharp \pp,u^*(\tau))-{h_n} (f_m^*\diamond Y_\nu) (\tau,\omega)\|\\\leq 
		\int_0^1\int_{\Omega}&\int_0^1\|\varpi_m^n(\theta,r,\tau,\omega')\|\|Z_\nu^h(\tau,\omega')-X^*(\tau,\omega')\| dr\pp(d\omega')d\theta \\+
		&\int_\Omega \|f_m^*(\tau,\omega,\omega')\|\|Z_\nu^h(\tau,\omega')-X^*(\tau,\omega')-h_nY_\nu(\tau,\omega')\|\pp(d\omega').
	\end{split}
	\] This, the H\"older's inequality and assumption~\ref{assumption:derivative_f} give that
	\begin{equation}\label{ineq:Z_h_Z_star_Lp:second:Holder:second:1}
		\begin{split}
			\int_{s+h_n}^t\|f(\tau,&X^*(\tau),Z_\nu^{h_n}(\tau)\sharp \pp,u^*(\tau))\\&{}\hspace{20pt}- f(\tau,X^*(\tau),X^*(\tau)\sharp \pp,u^*(\tau))-{h_n}(f_m^*\diamond Y_\nu) (\tau,\omega)\|_{L^p}d\tau \\ \leq
			\bigg[\int_{s+h_n}^t &\int_\Omega\int_0^1\int_{\Omega}\int_0^1\|\varpi_m^n(\theta,r,\tau,\omega,\omega')\|^q dr\pp(d\omega')d\theta \pp(d\omega)\bigg]^{1/q}\\&{}\hspace{70pt}\|Z_\nu^h(\tau,\omega')-X^*(\tau,\omega')d\tau\|_{L^p,s,T}\\&+C_m
			\int_{s+h_n}^t\|Z_\nu^{h_n}(\tau)-X^*(\tau)-h_nY_\nu(\tau))\|_{L^p}d\tau.
		\end{split}
	\end{equation} Denote 
	\begin{equation}\label{intro:a_3}
		\begin{split}
			a_n^{(3)}\triangleq C_4T^{1/p}\bigg[\int_{s+h_n}^T \int_\Omega\int_0^1\int_{\Omega}\int_0^1\|\varpi_m(\theta,&r,\tau,\omega,\omega')\|^qdr\\&{}\pp(d\omega')d\theta \pp(d\omega)d\tau\bigg]^{1/q}.\end{split}
	\end{equation} Notice that
	\[\|\varpi_m^n(\theta,r,\tau,\omega,\omega')\|\leq 2C_m,\] while the sequence $\{h_n\}$ is such that 
	$Z^{h_n}(\tau,\omega)\rightarrow X^*(\tau,\omega)$ $\lambda\otimes \pp$-a.e. as $n\rightarrow\infty$. This, continuity of the function $\nabla_mf$, the fact that $\norm{Z^{h_n}_\nu(t)-X^*(t)}{L^p}$ converges to zero uniformly w.r.t. time and the very definition of the function $\varpi_{m}^n$ (see~\eqref{intro:y_n_2} and~\eqref{inro:varpi_m_n}) imply that $\varpi_{m}^n(\theta,r,\tau,\omega,\omega')\rightarrow 0$ as $n\rightarrow \infty$ $\lambda\otimes\lambda\otimes\lambda\otimes\pp\otimes\pp$-a.e. Hence, due to the dominated convergence theorem, we have that 
	\begin{equation}\label{convergence:a_3}
		a_n^{(3)}\rightarrow 0\text{ as }n\rightarrow\infty.
	\end{equation}
	
	Applying the H\"older inequality for the first term in  the right-hand side of~(\ref{ineq:Z_h_Z_star_Lp:second:Holder:second:1}), using definition~(\ref{intro:a_3}) and the third statement of Proposition~\ref{prop:Z_h_bounds}, we deduce that
	\begin{equation}\label{ineq:Z_h_Z_star_Lp:second:Holder:second:final}
		\begin{split}
			\int_{s+h_n}^t\|f(&\tau,X^*(\tau),Z_\nu^{h_n}(\tau)\sharp \pp,u^*(\tau))\\&{}\hspace{20pt}- f(\tau,X^*(\tau),X^*(\tau)\sharp \pp,u^*(\tau))-{h_n}(f_m^*\diamond Y_\nu) (\tau,\omega)\|_{L^p}d\tau \\ \leq
			&a_n^{(3)}h+C_m
			\int_{s+h_n}^t\|Z_\nu^{h_n}(\tau)-X^*(\tau)-h_nY_\nu(\tau))\|_{L^p}d\tau.
		\end{split}
	\end{equation}
	
	Combining~(\ref{ineq:Z_h_Z_star_Lp:triangle}),~(\ref{ineq:Z_h_Z_star_Lp:second:1}),~(\ref{ineq:Z_h_Z_star_Lp:second:Holder}),~(\ref{ineq:Z_h_Z_star_Lp:second:Holder:first:final}),~(\ref{ineq:Z_h_Z_star_Lp:second:Holder:second:final}), we arrive at the following fact:
	\[
	\begin{split}
		\frac{1}{h_n}\|Z_\nu^{h_n}(t)-&X^*(t)-h_n Y_\nu(t))\|_{L^p}\leq (a_n^{(1)}+a_n^{(2)}+a_n^{(3)})\\&+(C_x+C_m)
		\int_{s+h_n}^t\frac{1}{h_n}\|Z_\nu^{h_n}(\tau)-X^*(\tau)-h_nY_\nu(\tau))\|_{L^p}d\tau.
	\end{split}
	\]
	Applying to this estimate the Gronwall's inequality, we obtain
	\[\frac{1}{h_n}\|Z_\nu^{h_n}(t)-(X^*(t)+h_n Y_\nu(t))\|_{L^p}\leq (a_n^{(1)}+a_n^{(2)}+a_n^{(3)})e^{(C_x+C_m)(t-s)}.\] This and the fact that the sequences $\{a_n^{(1)}\}$, $\{a_n^{(2)}\}$, $\{a_n^{(3)}\}$ converge to zero (see~(\ref{intro:a_1_def}),~(\ref{intro:a_2}),~\eqref{convergence:a_3}) give the statement of the proposition.
\end{proof}


\subsection{Derivative of the perturbed running cost}\label{appendix:sub:integral}
The aim of this section is to give the proof of Proposition \ref{prop:derivative_integral}. It relies on the following auxiliary statement.

\begin{lemma}\label{lm:f_0_derivative_bounds} For every $r_1,r_2\in [s,T]$, $r_1<r_2$, one has that 
	\begin{itemize}	
		\item $f_{0,x}^*\in \Lpspace{q}{r_1}{r_2}{\rds}$, 
		\item $f_{0,m}^*\in\Lpspacedouble{q}{r_1}{r_2}{\rds}$. 
	\end{itemize}	
	Moreover, $\norm{f_{0,x}^*}{L^q,r_1,r_2}$ and $\norm{f_{0,m}^*}{L^q,r_1,r_2}$ are bounded uniformly w.r.t.~$r_1$ and~$r_2$.
\end{lemma}
\begin{proof}
	We consider only $f_{0,x}^*$. The case of $f_{0,m}^*$ is the same. 
	
	The fact that $f_{0,x}^*\in B([r_1,r_2]\times\Omega,\mathcal{B}_\lambda([r_1,r_2])\otimes\mathcal{F};\rd)$ follows from assumption~\ref{assumption:function} and the very definition of this function (see~\eqref{intro:f_0_x}). This definition together with  assumption~\ref{assumption:derivative_f_0} gives that 
	\[\|f_{0,x}^*(t,\omega)\|^q\leq C^0_x(1+\|X^*(t,\omega)\|^p+\|X^*(t)\|^p_{L^p}+\|u^*(t,\omega)\|^p).\] Using the first statement of Proposition~\ref{prop:Z_h_bounds} for $h=0$ and the fact that $u^*\in\up$, we obtain that
	\begin{equation*}\label{ineq:f_0_L_q}
		\begin{split}
			\|f_{0,x}^*\|_{L^q,r_1,r_2}^q&=\int_s^{s+h_n} \expect\|f_{0,x}^*(t,\omega)\|^q\pp(d\omega)dt\\&\leq \int_{r_1}^{r_2} \expect \Big[ C^0_x(1+\|X^*(t,\omega)\|^p+\|X^*(t)\|^p_{L^p}+\|u^*(t,\omega)\|^p)\Big]dt\\&\leq 
			C^0_x(T+2TC_0^p+\norm{u^*}{\up}^p)<+\infty.
		\end{split}
	\end{equation*}
	Thus, $\|f_{0,x}^*\|_{L^q,r_1,r_2}$ is bounded by a constant that does not depend on $r_1$,~$r_2$.
\end{proof}

\begin{proof}[Proof of Proposition \ref{prop:derivative_integral}] We split the proof into the five steps.
	\begin{enumerate}[label=Step \arabic*.]
		\item Notice that, $u^{h_n}_\nu(t)=u^*(t)$ when $t\notin [s,s+h_n]$, and $u^{h_n}_\nu(t)=\nu$ for $t\in [s,s+h]$. Moreover, $Z^{h_n}_\nu(t)=X^*(t)$ on $[0,s]$. Therefore,
		\begin{equation}\label{ineq:f_0_Z_h_X_star_step_1}
			\begin{split}
				\bigg|\bigg[\int_0^T\mathbb{E}[f_0&(t,Z^{h_n}_\nu(t),Z^{h_n}_\nu(t)\sharp \pp,u^{h_n}_\nu(t))dt\\&-\int_0^T\mathbb{E}f_0(t,X^*(t),X^*(t)\sharp \pp,u^*(t))]dt\bigg]\\ &-
				h_n\mathbb{E}\Delta^s_\nu f^*_0-h_n\int_s^T\mathbb{E}[f_{0,x}^*(t)Y_\nu(t)+ (f_{0,m}^*\diamond Y_\nu)(t)]dt\bigg| \\ \leq G^{(1)}_n&+G^{(2)}_n+G^{(3)}_n+G^{(4)}_n, 
			\end{split}
		\end{equation} where we denote
		\begin{equation}\label{intro:G_1_n}
			\begin{split}
				\hspace{-8pt}	G_n^{(1)}\triangleq\bigg|\int_s^{s+h_n} \mathbb{E}[f_0(t,&Z^{h_n}_\nu(t),Z^{h_n}_\nu(t)\sharp \pp,\nu)\\-f_0(&t,X^*(t),X^*(t)\sharp \pp,u^*(t))]dt-h_n\mathbb{E}\Delta^s_\nu f^*_0\bigg|,\end{split}
		\end{equation}
		\begin{equation}\label{intro:G_2_n}
			\hspace{-96pt}	G_n^{(2)}\triangleq h_n\int_s^{s+h_n}\mathbb{E}|f_{0,x}^*(t)Y_\nu(t)+ (f_{0,m}^*\diamond Y_\nu)(t)|dt,
		\end{equation}
		\begin{equation}\label{intro:G_3_n}
			\begin{split}
				G_n^{(3)}\triangleq \int_{s+h_n}^T \mathbb{E}|f_0&(t,Z^{h_n}_\nu(t),Z^{h_n}_\nu(t)\sharp \pp,u^*(t))\\-f&{}_0(t,X^*(t),Z^{h_n}_\nu(t)\sharp \pp,u^*(t))-h_nf_{0,x}^*(t)Y_{\nu}(t)|dt,
			\end{split}
		\end{equation}
		\begin{equation}\label{intro:G_4_n}
			\begin{split}
				G_n^{(4)}\triangleq \int_{s+h_n}^T \mathbb{E} |f_0(t,X^*(t),Z^{h_n}_\nu(t)\sharp \pp,u^*(t&))\\ -f_0(t,X^*(t),X^*(t)\sharp \pp,u^*(t&))-h_n( f_{0,m}^*\diamond Y_{\nu})(t)|dt.
			\end{split}
		\end{equation} 
		
		In the following, we will show that $G^{(i)}_n/h_n\rightarrow 0$ as $n\rightarrow\infty$.
		\item Now choose $t\in [s,s+h]$.
		Notice that
		\begin{equation}\label{ineq:G_1_n:sum_nu}
			\begin{split}
				|\mathbb{E}[f_0(t,&Z^{h_n}_\nu(t),Z^{h_n}_\nu(t)\sharp \pp,\nu)-f_0(s,X^*(s),X^*(s)\sharp \pp,\nu)]|\\ \leq \expect|&f_0(t,Z^{h_n}_\nu(t),Z^{h_n}_\nu(t)\sharp \pp,\nu)- f_0(t,X^*(t),Z^{h_n}_\nu(t)\sharp \pp,\nu)|\\+&\expect|f_0(t,X^*(t),Z^{h_n}_\nu(t)\sharp \pp,\nu)-f_0(t,X^*(t),X^*(t)\sharp \pp,\nu)|\\+&
				\expect|f_0(t,X^*(t),X^*(t)\sharp \pp,\nu)-f_0(s,X^*(s),X^*(s)\sharp \pp,\nu)|.
			\end{split}
		\end{equation}
		
		Since $s\in\mathcal{T}$, while $\nu\in\mathcal{N}$, we have that (see equality \eqref{equality:f_0_nu} in Proposition \ref{prop:N_T}) 
		\begin{equation}\label{ineq:G_1_n:1}
			\begin{split}
				a_n'\triangleq \frac{1}{h_n}\int_s^{s+h_n}\expect|&f_0(t,X^*(t),X^*(t)\sharp \pp,\nu)\\&-f_0(s,X^*(s),X^*(s)\sharp \pp,\nu)|dt\rightarrow 0\text{ as }n\rightarrow\infty.
			\end{split}
		\end{equation}

		Since $f_0$ is continuously differentiable w.r.t. $x$, we have that
		\[\begin{split}
			|f_0(&t,Z^{h_n}_\nu(t,\omega),Z^{h_n}_\nu(t)\sharp \pp,\nu(\omega))-f_0(t,X^*(t,\omega),Z^{h_n}_\nu(t)\sharp \pp,\nu(\omega))| \\ &\leq 
			\int_{0}^1 \|\nabla_x f_0(t,X^*(t)+r(Z^{h_n}_\nu(t)-X^*(t)),Z^{h_n}_\nu(t)\sharp \pp,\nu)\|\\ &{}\hspace{210pt} \|Z^{h_n}_\nu(t)-X^*(t)\|dr.
		\end{split}\]
		
		Using the H\"older's inequality, we obtain
		\begin{equation*}
			\begin{split}
				\expect|&f_0(t,Z^{h_n}_\nu(t),Z^{h_n}_\nu(t)\sharp \pp,\nu)-f_0(t,X^*(t),Z^{h_n}_\nu(t)\sharp \pp,\nu)| \\ &\leq 
				\expect\bigg[\int_{0}^1 \|\nabla_x f_0(t,X^*(t)+r(Z^{h_n}_\nu(t)-X^*(t)),Z^{h_n}_\nu(t)\sharp \pp,\nu)\|dr\\&{}\hspace{210pt}\|Z^{h_n}_\nu(t)-X^*(t)\|\bigg] \\ &\leq 
				\bigg[\expect\int_{0}^1 \|\nabla_x f_0(t,X^*(t)+r(Z^{h_n}_\nu(t)-X^*(t)),Z^{h_n}_\nu(t)\sharp \pp,\nu)\|^qdr\bigg]^{1/q}\\&{}\hspace{210pt}\|Z^{h_n}_\nu(t)-X^*(t)\|_{L^p}.
			\end{split}\
		\end{equation*} Thanks to assumption~\ref{assumption:derivative_f_0} and Proposition~\ref{prop:Z_h_bounds}, we conclude
		\begin{equation}\label{ineq:G_1_n:2}
			\begin{split}
				\expect|f_0(t,Z^{h_n}_\nu(t),Z^{h_n}_\nu(t)\sharp \pp,&\nu)-f_0(t,X^*(t),Z^{h_n}_\nu(t)\sharp \pp,\nu)|\\& \leq  
				[C_x^0(1+2C_0^p+\norm{\nu}{L^p}^p)]^{1/q}C_2h_n.
			\end{split}\
		\end{equation}
		
		Since $f_0$ is continuously differentiable w.r.t. $m$, the following estimate holds true $\pp$-a.s.:
		\[\begin{split}
			|f_0(&t,X^*(t,\omega),Z^{h_n}_\nu(t)\sharp \pp,\nu(\omega))-f_0(t,X^*(t,\omega),X^*(t)\sharp \pp,\nu(\omega))|\\ &=
			\bigg|\int_0^1\int_\Omega\bigg[\frac{\delta f_0}{\delta m}(t,X^*(t,\omega),m^n(\theta,t),Z^{h_n}_\nu(t,\omega'),\nu(\omega))\\&{}\hspace{50pt}-\frac{\delta f_0}{\delta m}(t,X^*(t,\omega),m^n(\theta,t),X^*(t,\omega'),\nu(\omega))\bigg] \pp(d\omega')d\theta\bigg| \\ &\leq
			\int_0^1\int_\Omega\int_0^1 \|\nabla_m f_0(t,X^*(t,\omega),m^n(\theta,t),y^n_3(r,t,\omega'),\nu(\omega))\|\\ &{}\hspace{140pt}\|Z^{h_n}_\nu(t,\omega')-X^*(s,\omega')\|dr\pp(d\omega')d\theta.
		\end{split}
		\] Above we denoted
		\[y^n_3(r,t,\omega')\triangleq X^*(t,\omega')+r(Z^{h_n}_\nu(t,\omega')-X^*(t,\omega'))\]
		\[m^n(\theta,t)\triangleq \theta Z_\nu^{h_n}(t)\sharp \pp+(1-\theta)X^*(t)\sharp \pp.\] Applying the H\"older inequality, we obtain that
		\[\begin{split}
			\int_{\Omega}|&f_0(t,X^*(t,\omega),Z^{h_n}_\nu(t)\sharp \pp,\nu(\omega))-f_0(t,X^*(s,\omega),X^*(t)\sharp \pp,\nu(\omega))|\pp(d\omega) \\ &\leq 
			\bigg[\int_\Omega\int_0^1\int_\Omega\int_0^1 \|\nabla_m f_0(t,X^*(s,\omega),m^{n}(t,\theta),y^n_3(r,t,\omega'),\nu(\omega))\|^q\\&{}\hspace{130pt}dr\pp(d\omega')d\theta \pp(d\omega)\bigg]^{1/q}\|Z^{h_n}_\nu(t)-X^*(t)\|_{L^p}.
		\end{split}
		\] Using the estimates from assumption~\ref{assumption:derivative_f_0}, Proposition~\ref{prop:Z_h_bounds}, we deduce the following inequality
		\begin{equation}\label{ineq:G_1_n:3}
			\begin{split}
				\expect|f_0(t,X^*(t),Z^{h_n}_\nu&(t)\sharp \pp,\nu)\\&-f_0(t,X^*(t,\omega),X^*(t)\sharp \pp,\nu)| \\ \leq [C_m^0(1&+3C_0^p+\|\nu\|_{L^p}^p)]^{1/q}C_2h_n.
			\end{split}
		\end{equation} Combining~(\ref{ineq:G_1_n:sum_nu})--(\ref{ineq:G_1_n:3}), we have that, for $t\in [s,s+h_n]$,
		\begin{equation}\label{ineq:G_1_n:sum_nu_final}
			\begin{split}
				\frac{1}{h_n}\int_s^{s+h_n}|\mathbb{E}[f_0(t,&Z^{h_n}_\nu(t),Z^{h_n}_\nu(t)\sharp \pp,\nu)\\&-f_0(s,X^*(s),X^*(t)\sharp \pp,\nu)]|\leq a_n^{(4)},
			\end{split}
		\end{equation} where
		\begin{equation}
			\begin{split}
				a_n^{(4)}\triangleq a_n'&+[C_m^0(1+2C_0^p+\norm{\nu}{L^p}^p)]^{1/q}C_1h_n\\&+[C_x^0(1+3C_0^p +\norm{\nu}{L^p}^p)]^{1/q}C_1h_n.\end{split}
		\end{equation} Notice that the sequence $\{a_n^{(4)}\}_{n=1}^\infty$ converges to 0.

		Therefore, recalling the definition of $\Delta^s_\nu f^*_0$, we obtain the estimate
		\begin{equation}\label{ineq:G_1_n_prefinal}
			\begin{split}
				G_n^{(1)}\leq \int_s^{s+h_n}  \expect|f_0(t,X^*(t),X^*(&t)\sharp \pp,u^*(t))\\-f_0(&s,X^*(s),X^*(s)\sharp \pp,u^*(s))|dt\\+a_n^{(4)}\cdot h_n&.
			\end{split}
		\end{equation} This and  equality \eqref{equality:f_0_star} in Proposition~\ref{prop:N_T} imply that 
		$G_n^{(1)}/h_n\rightarrow 0$ as $n\rightarrow\infty$.

		
		\item Let us estimate
		\[\int_s^{s+h_n}\expect|f_{0,x}^*(t)Y_\nu(t)+ (f_{0,m}^*\diamond Y_\nu) (t)|dt.\] 
		We have that
		\begin{equation}\label{ineq:G_2_X_f_0_Y}
			\int_s^{s+h_n} \expect|f_{0,x}^*(t)Y_\nu(t)|dt\leq \|f_{0,x}^*\|_{L^q,s,s+h_n}\cdot \|Y_\nu\|_{L^p,s,s+h_n}.
		\end{equation} By Lemma~\ref{lm:f_0_derivative_bounds}, the values $\|f_{0,x}^*\|_{L^q,s,s+h_n}$ are uniformly bounded. 
		Furthermore, due to Proposition~\ref{prop:Y_nu} and~\eqref{ineq:norm_L^p_sup}, 
		\begin{equation}\label{ineq:G_2_X_Y_nu}
			\|Y_\nu\|_{L^p,s,s+h_n}\leq (h_n)^{1/p}C_3.
		\end{equation}
		Combining this,~\eqref{ineq:G_2_X_f_0_Y} and~\eqref{ineq:G_2_X_Y_nu} we arrive at the estimate
		\begin{equation}\label{ineq:G_2_X_prefinal}
			\int_s^{s+h_n} \expect|f_{0,x}^*(t)Y_\nu(t)|\pp(d\omega)dt\leq C_5h_n^{1/p},
		\end{equation} where
		$C_5$ is a constant.
		
		Analogously, we have
		\[\begin{split}
			\int_s^{s+h_n}\expect&|(f_{0,m}^*\diamond Y_\nu)(t)|dt\\&\leq 
			\int_s^{s+h_n}\int_\Omega\int_\Omega| f_{0,m}^*(t,\omega,\omega')Y_\nu(t,\omega')|\pp(d\omega')\pp(d\omega)dt
			\\&\leq \|f_{0,m}^*\|_{L^q,s,s+h_n}\cdot \|Y_\nu\|_{L^p,s,s+h_n}.\end{split}\] Using Lemma~\ref{lm:f_0_derivative_bounds}, we obtain that 
		$\|f_{0,m}^*\|_{L^q,s,s+h_n}$ are uniformly bounded. This and~\eqref{ineq:G_2_X_f_0_Y} give the estimate
		\[\int_s^{s+h_n}\expect|(f_{0,m}^*\diamond Y_\nu)(t)|dt\leq C_6h_n^{1/p},\] where $C_6$ is a constant (certainly dependent on $(X^*,u^*)$). Using this inequality and~\eqref{ineq:G_2_X_prefinal}, we conclude that
		\[\int_s^{s+h_n}\expect\big|f_{0,x}^*(t)Y_\nu(t)+ (f_{0,m}^*\diamond Y_\nu) (t)\big|dt\leq (C_5+C_6)h_n^{1/p}.\] 
		Therefore, $G_n^{(2)}$ defined by~(\ref{intro:G_2_n}) is such that
		\[G_n^{(2)}/h_n\rightarrow 0\text{ as }n\rightarrow\infty.\]
		
		\item We have that
		\[\begin{split}
			f_0(t,Z^{h_n}_\nu(t),&Z^{h_n}_\nu(t)\sharp \pp,u^*(t))-f_0(t,X^*(t),Z^{h_n}_\nu(t)\sharp \pp,u^*(t)) \\ =\int_0^1 \nabla_x f_0&(t,y^n_4(r,t),Z^{h_n}_\nu(t)\sharp \pp,u^*(t))(Z^{h_n}_\nu(t)-X^*(t))dr.
		\end{split}\] Here we use the designation
		\[y^n_4(r,t,\omega)\triangleq	X^*(t)+r(Z^{h_n}_\nu(t)-X^*(t))\]  omitting the dependence on $\omega$.
		Denote
		\[\begin{split}
			\varpi_{0,x}^n(r,t,\omega)\triangleq \nabla_x f_0(t,y^n_4(r,t,&\omega),Z^{h_n}_\nu(t)\sharp \pp,u^*(t,\omega))\\-\nabla_x f_0&(t,X^*(t,\omega),Z^{h_n}_\nu(t)\sharp \pp,u^*(t,\omega)).
		\end{split}\] 
		Therefore,
		\begin{equation}\label{ineq:G_3_n_first}\begin{split}
				G^{(3)}_n\leq \expect \int_{s}^T\int_0^1 |\varpi_{0,x}^n(r&,t)(Z^{h_n}_\nu(t)-X^*(t,\omega))|dr dt\\+
				\int_{s}^T\expect |f_{0,x}^*(t&) (Z^{h_n}_\nu(t)-X^*(t)-h_nY_\nu(t))|dt \\ \leq 
				\Bigg[\expect\int_{s}^T\int_0^1 \|\varpi_{x}^{0,n}&(r,t)\|^qdr dt\Bigg]^{1/q}\|Z^{h_n}_\nu-X^*\|_{L^p,s,T}\\+
				\|&f_{0,x}^*\|_{L^q,s,T} \|Z^{h_n}_\nu-X^*-h_nY_\nu\|_{L^p,s,T}.
			\end{split}
		\end{equation}
		
		Notice that, due to the choice of the sequence $\{h_n\}_{n=1}^\infty$, $Z^{h_n}_\nu\rightarrow X^*$ $\lambda\otimes \pp$-a.e. Therefore,  $\varpi_{0,x}^n$ converges to zero $\lambda\otimes\lambda\otimes\pp$-a.e. as $n\rightarrow\infty$. Moreover, $\pp$-a.s.
		\[\begin{split}
			\|\varpi_{0,x}^n(r,t&)\|\\\leq \|\nabla_x &f_0(t,X^*(t)+r(Z^{h_n}_\nu(t)-X^*(t)),Z^{h_n}_\nu(t)\sharp \pp,u^*(t))\|\\&+\|\nabla_x f_0(t,X^*(t),Z^{h_n}_\nu(t)\sharp \pp,u^*(t))\|.
		\end{split}\] Using assumption~\ref{assumption:derivative_f_0}, the Jensens's inequality, Proposition~\ref{prop:Z_h_bounds} and the fact that $u^*\in\up$, we obtain that
		\[\expect\int_{s}^T\int_{0}^1\|\varpi_{0,x}^n(r,t)\|^qdrdt\leq 2 C_x^0(1+2 C_0+\norm{u^*}{\up})<+\infty.\] Therefore, by the dominated convergence theorem
		\[\expect\int_{s}^T\int_{0}^1\|\varpi_{0,x}^n(r,t)\|^qdrdt\rightarrow 0\text{ as }n\rightarrow\infty.\] Furthermore, by the third statement of Proposition ~\ref{prop:Z_h_bounds} and~\eqref{ineq:norm_L^p_sup},
		\[\|Z^{h_n}_\nu(t)-X^*(t)\|_{L^p}\leq T^{1/p}C_1h_n.\] 
		
		Additionally, by Lemma~\ref{lm:f_0_derivative_bounds},
		$\|f_{0,x}^*\|_{L^q,s,T}<+\infty$. Finally, thanks to Proposition~\ref{prop:derivative_y_h} and~(\ref{ineq:norm_L^p_sup}), 
		$ \|Z^{h_n}_\nu(t)-X^*(t)-h_nY_\nu(t)\|_{L^p,s,T}/h_n\rightarrow 0$ as $n\rightarrow\infty$.
		
		Combining the above estimates of the right-hand side of estimate~(\ref{ineq:G_3_n_first}), we conclude that 
		\[\frac{G^{(3)}_n}{h_n}\rightarrow 0\text{ as }n\rightarrow\infty.\]

		\item 
		As above, we have that, for $\pp$-a.e. $\omega\in\Omega$,
		\[\begin{split}
			f_0(t,&X^*(t,\omega),Z^{h_n}_\nu(t)\sharp \pp,u^*(t,\omega))-f_0(t,X^*(t,\omega),X^*(t)\sharp \pp,u^*(t,\omega)) \\ &=
			\int_0^1\int_\Omega\int_0^1\nabla_m f_0(t,X^*(t,\omega),m^n(t,\theta), y^n_5(r,t,\omega'),u^*(t,\omega))\\ &{}\hspace{260pt}dr\pp(d\omega')d\theta,
		\end{split}
		\]
		where we denote
		\[y^n_5(r,t,\omega')=X^*(t,\omega')+r(Z_\nu^{h_n}(t,\omega')-X^*(t,\omega')),\]
		\[m^n(t,\theta)\triangleq \theta Z_\nu^{h_n}(t)\sharp \pp+(1-\theta)X^*(t)\sharp \pp.\]
		Arguing as in the proof of estimate~(\ref{ineq:Z_h_Z_star_Lp:second:Holder:second:final}), we have
		\[\begin{split}
			\expect\int_{s}^T| f_0&(t,X^*(t),Z^{h_n}_\nu(t)\sharp \pp,u^*(t))\\&{}\hspace{50pt}-f_0(t,X^*(t),X^*(t)\sharp \pp,u^*(t))-h_n(f_{0,m}^*\diamond Y_\nu)(t)|dt \\ \leq
			\bigg[&\int_\Omega\int_{s}^T\int_0^1\int_\Omega\int_0^1 \|\varpi_{0,m}^n(\theta,r,t,\omega,\omega')\|^qdr\pp(d\omega')d\theta d t\pp(d\omega)\bigg]^{1/q}\\ \\&{}\hspace{210pt} \|Z_\nu^{h_n}-X^*\|_{L^p,s,T}\\ &{}\hspace{60pt}+
			\norm{f_{0,m}^*}{L^q,s,T}\norm{Z_\nu^{h_n}(t)-X^*(t)-h_nY_\nu(t)}{L^p,s,T},
		\end{split}\]
		Here we put
		\[\begin{split}
			\varpi_{0,m}^n(\theta,r,t,&\omega,\omega')\\\triangleq \nabla_m &f_0(t,X^*(t,\omega),m^n(t,\theta), y^n_5(r,t,\omega'),u^*(t,\omega))\\&-
			\nabla_m f_0(t,X^*(t,\omega),X^*(t)\sharp \pp, X^*(t,\omega'),u^*(t,\omega)).\end{split}\]
		Furthermore, $\{\varpi_{0,m}^n\}$ converges to zero $\lambda\otimes\lambda\otimes\lambda\otimes\pp\otimes\pp$-a.e., and the functions $\|\varpi_{0,m}^n\|^q$ are uniformly integrable (here we use the same arguments as in Step 3).
		Thus, due to the dominated convergence theorem,
		\[\begin{split}
			\int_\Omega\int_{s+h_n}^T\int_0^1\int_\Omega\int_0^1 \|\varpi_{0,m}^n(\theta,r,t,\omega,\omega')\|^qdr\pp(d\omega')d\theta d t&\pp(d\omega)\rightarrow 0\\&\text{ as }n\rightarrow\infty.\end{split}\] Recall that the third statement of Proposition~\ref{prop:Z_h_bounds} says that 
		$\|Z_\nu^{h_n}(t)-X^*(t)\|_{L^p}\leq C_2 h_n$. Moreover, by Lemma~\ref{lm:f_0_derivative_bounds}, $\norm{f_{0,m}^*}{L^q,s,T}<+\infty$. Finally, Proposition~\ref{prop:derivative_y_h} states that $\|Z_\nu^{h_n}(t)-X^*(t)-h_nY_\nu(t)\|_{L^p}/h_n$ tends to $0$ uniformly w.r.t time variable. This and~(\ref{ineq:norm_L^p_sup}) give that 
		\[\frac{1}{h_n}\norm{Z_\nu^{h_n}(t)-X^*(t)-h_nY_\nu(t)}{L^p,s,T}\rightarrow 0\text{ as }n\rightarrow\infty.\]
		Therefore, $G^{(4)}_n/h_n$ tends to 0 when $n\rightarrow \infty$.
		
	\end{enumerate}
	Steps 1--5 imply that
	\[\begin{split}
		\lim_{n\rightarrow\infty}\bigg|\frac{1}{h_n}\bigg[\int_0^T\mathbb{E}[f_0(t,&Z^{h_n}_\nu(t),Z^{h_n}_\nu(t)\sharp \pp,u^{h_n}_\nu(t))dt\\-&\int_0^T\mathbb{E}f_0(t,X^*(t),X^*(t)\sharp \pp,u^*(t))]dt\bigg]\\-&
		\mathbb{E}\Delta^s_\nu f^*_0+\int_s^T\mathbb{E}[f_{0,x}^*(t,\omega)Y_\nu(t,\omega)+ (f_{0,m}^*\diamond Y_\nu)(t)]dt\bigg|\\
		\leq\frac{1}{h_n}(G_n^{(1)}&+G_n^{(2)}+G_n^{(3)}+G_n^{(4)})\rightarrow 0\text{ as } n\rightarrow\infty.
	\end{split}\] This completes the proof.
\end{proof}

\newpage
\phantomsection

\end{document}